\documentclass[a4paper,10pt, twoside]{article}
\usepackage{amsmath,amsthm}
\usepackage{amssymb,latexsym}
\usepackage{mathrsfs}
\usepackage{enumerate}
\usepackage[colorlinks,citecolor=blue]{hyperref}
\usepackage[numbers,sort&compress]{natbib}
\usepackage{indentfirst}
\usepackage{color}
\usepackage{esint}

\DeclareMathAlphabet{\mathpzc}{OT1}{pzc}{m}{it}

\newtheorem{theorem}{Theorem}[section]
\newtheorem{corollary}{Corollary}[section]
\newtheorem{lemma}{Lemma}[section]

\newtheorem{definition}{Definition}[section]
\newtheorem{remark}{Remark}[section]

\newtheorem{example}{Example}[section]
\newtheorem{assumption}{Assumption}[section]
\theoremstyle{definition}
\theoremstyle{remark}

\numberwithin{equation}{section}
\allowdisplaybreaks

\date{}

\begin{document}

\markboth{ }{}

\date{}
\baselineskip 0.15in

\title{Stochastic maximal $L^p$-regularity for non-autonomous evolution equations  with fractional derivative in UMD spaces}
\author{Lu Lu Tao$^{1}$, Jia Wei He$^{1,2,}\footnote{E-mail: jwhe@gxu.edu.cn} $
\\[1.8mm]
\footnotesize  {$^1$ School of Mathematics, Guangxi University, Nanning 530004,  China}\\[1.5mm]
\footnotesize  {$^2$ Guangxi Mathematical Research Center, Guangxi University, Nanning 530004, China}	
}
\date{}
\maketitle

\begin{abstract}
This paper is concerned with the maximal regularity theory for non-autonomous stochastic evolution equations with a generalized fractional derivative in UMD spaces. The generalized time-fractional derivative provides a unified framework covering both the classical Riemann-Liouville and Caputo fractional derivatives, which accommodates a wider class of anomalous diffusion processes with intermediate memory effects.  Based on the singularities of the initial term and the stochastic convolution kernel, a time-weighted space and a regular-singular decomposition are used to obtain the well-posedness, space-time regularity, and stochastic maximal $L^p$-regularity results. Our results are applied to non-autonomous stochastic diffusion equation and stochastic fractional reaction-diffusion SIR model.
\end{abstract}

\baselineskip 0.2in
\section{Introduction}\label{sec1}

Let $X_0$ be a Banach space. In this paper, we are concerned with non-autonomous abstract stochastic fractional evolution equations
\begin{equation}\label{e1.1}
 \left\{
\begin{aligned}
&D_{0+}^{\lambda,\mu}u(t)+A(t)u(t)=F(t, \omega, u(t))+ G(t, \omega, u(t))dW_H(t), && t\in(0,T],  \\
&I_{0+}^{(1-\lambda)(1-\mu)}u(0)=u_0,
\end{aligned}\right.
\end{equation}
where $D_{0+}^{\lambda,\mu}$ is the
generalized fractional derivative of Hilfer type with order $0< \mu<1$
and type $0\leq\lambda\leq1$.
$A(t)$ is a time-dependent operator in $X_0$. The nonlinear drift $F$ and diffusion coefficient $G$ satisfy the assumptions specified below. The initial value $u_0$ is strongly $\mathcal F_0$-measurable. Let $(\Omega,\mathcal F,(\mathcal F_t)_{t\in[0,T]},\mathbb P)$ be a
complete filtered probability space, and let $W_H$ be an $H$-cylindrical Brownian motion adapted to $(\mathcal F_t)_{t\in[0,T]}$,
where $H$ is a separable Hilbert space.


Fractional calculus extends the classical concepts of differentiation and integration to non-integer orders. The weakly singular kernels of fractional derivatives provide a natural way to describe memory and hereditary effects and have led to the development of fractional partial differential equations with applications in engineering, physics, finance, and biology, see \cite{Caponetto, Delbosco, Diethelm, Dong, Hilfer, Kilbas, Pablo, Zhou2014} and the references therein.
Time-dependent material parameters and diffusion coefficients give rise to non-autonomous fractional partial differential equations in many applications. A  typical example (see e.g. \cite{Dong,Kim}) is
\begin{equation*}
\partial_t^{\alpha}u(t,x)
+D_j\bigl(a^{ij}(t,x)D_i u(t,x)\bigr)
=f(t,x,u(t,x)),  \quad t>0,
\end{equation*}
where the fractional derivative $\partial_t^\alpha$ describes the anomalous diffusion and memory effects, and the coefficient $a^{ij}(t,x)$ reflects the time-varying nature of the physical environment. From an abstract functional analysis perspective, fractional evolution equations have garnered increasing attention in recent years. One important reason is that many fractional partial differential equations, including the aforementioned non-autonomous fractional diffusion equation, as well as fractional wave equation, Navier-Stokes equation and Schr\"{o}dinger equation can be abstracted as fractional evolution equations. For more details, we refer to \cite{Alvarez, Hernandez, He2022, He2024, Li, Lizama, Kim, Sousa}  and the references therein.

On the other hand, non-autonomous fractional evolution equations have been studied less  than their autonomous counterparts because of the high complexity of time-dependent operators.
El-Borai \cite{El-Borai} previously proved the existence of fundamental solutions for linear non-autonomous fractional evolution equations. Based on this framework, Chen et al. \cite{Chen2017, Chen2019} proved the existence of mild solutions for fractional non-autonomous equations by employing the theory of fractional calculus and the fixed point theorem with respect to $k$-set contractive operators. Furthermore, He and Zhou \cite{He2022} introduced a novel representation for the solutions by the Mittag-Leffler function, the Wright-type function, and the theory of analytic semigroups. Based on this expression, they proved the existence, uniqueness  and maximal H\"{o}lder regularity of classical solutions for non-autonomous evolution equations with Caputo fractional derivatives of order $\alpha \in (0,1)$. The solvability result of non-autonomous fractional evolution equations was also extended to the cases of Caputo fractional derivatives of order $\alpha \in (1,2)$ in \cite{He20222} and Riemann-Liouville fractional derivatives of order $\alpha \in (0,1)$ in \cite{Zhou2023}. For superlinear growth nonlinearities, Feng and Chen \cite{Feng} showed the existence of mild solutions for a class of nonlinear fractional partial differential equations with nonlocal initial conditions by the fractional power theory and the Banach fixed point theorem in the interpolation space.

We next observe the following two examples before considering the generalized fractional derivatives.
\begin{example}\label{exam1.1}
Consider the Riemann-Liouville fractional equation
$${}^{RL}\!D_{0+}^{\alpha}u(t)+\rho t^\sigma u(t)=0,
\quad t>0,$$
where $0<\alpha<1$, $\sigma>-\alpha$ and $\rho\in\mathbb R$. For $\sigma\neq0$, the coefficient $t^\sigma$ depends on time, then it is a simple non-autonomous fractional equation. By \cite[Theorem 4.7]{Kilbas}, we have the solution
$$u(t)=t^{\alpha-1}
E_{\alpha,\,1+\sigma/\alpha,\,1+(\sigma-1)/\alpha}
\bigl(-\rho t^{\alpha+\sigma}\bigr),$$
where $E_{\alpha,m,l}$ denotes the Mittag-Leffler function. Then $$t^{1-\alpha}u(t)
=E_{\alpha,\,1+\sigma/\alpha,\,1+(\sigma-1)/\alpha}
\bigl(-\rho t^{\alpha+\sigma}\bigr)\to 1,~~{\rm as}\ t\to0+.$$
\end{example}

\begin{example}\label{exam1.2}
Let $(\Omega,\mathcal F,(\mathcal F_t)_{t\in[0,T]},\mathbb P)$ be a filtered probability space, let
$\mathcal O\subset\mathbb R^N$ be a bounded domain with $C^2$
boundary and let $W_H$ be an $H$-cylindrical Brownian motion adapted to $(\mathcal F_t)_{t\in[0,T]}$. Consider the following fractional non-autonomous stochastic partial differential equations
\begin{equation}\label{e1.2}
\left\{
\begin{aligned}
&{}^C\!D_{0+}^{\alpha}u(t)+A(t)u(t)
=F(t,\omega,u(t))+G(t,\omega,u(t))\,dW_H(t),
&&t\in(0,T],\\
&u(t)|_{\partial\mathcal O}=0,
&&t\in(0,T],\\
& u(0)=u_0,
\end{aligned}
\right.
\end{equation}
where ${}^C\!D_{0+}^{\alpha}$ is the Caputo fractional derivative of order $0<\alpha<1$,
$u_0$ is a strong $\mathcal{F}_0$-measurable $X_0$-valued random variable. Coefficients $a,c:[0,T]\to(0,\infty)$
are time-dependent. For $q\in(2,\infty)$, set $X_0=L^q(\mathcal O)$, $X_1=W^{2,q}(\mathcal O)\cap W_0^{1,q}(\mathcal O)$ and for $v\in X_1$, define
$$A(t)v:=-a(t)\Delta v+c(t)v.$$
The nonlinearities $F$ and $G$ act between suitable fractional power spaces and satisfy some assumptions. Then \eqref{e1.2} can be written as the abstract problem \eqref{e1.1} in $X_0$.
\end{example}

In classical evolution equations, initial data are usually prescribed by pointwise values. Example \ref{exam1.1} shows that a solution to a fractional evolution equation may be singular at $t=0$, while its weighted trace $t^{1-\alpha}u(t)$ still has a finite limit. By contrast, the Caputo fractional equation in Example \ref{exam1.2} admits the usual pointwise initial condition $u(0)=u_0$ without introducing a time weight. Therefore, these examples motivate a unified framework that accommodates both weighted fractional traces and classical pointwise initial values. The generalized fractional derivative introduced by Hilfer \cite{Hilfer}, also called the Hilfer fractional derivative, provides such a framework. It is denoted by
$D_{t}^{\lambda,\mu}$ with order $\mu\in(0,1)$ and type $\lambda\in[0,1]$ and is defined by
$$D_{t}^{\lambda,\mu}u(t)
=I_{t}^{\lambda(1-\mu)}\frac{d}{dt}
 I_{t}^{(1-\lambda)(1-\mu)}u(t).$$
This derivative interpolates between two classical fractional derivatives. In particular, it reduces to the Riemann-Liouville fractional derivative for $\lambda = 0$ and to the Caputo derivative for $\lambda=1$. Hence, the generalized fractional derivatives provide flexibility in the choice of function spaces and initial conditions \cite{Fu,Gu}.

In recent years, the stochastic evolution equations driven by Wiener processes have become one of the core research topics in the field of infinite dimensional dynamical systems. For more details on stochastic evolution equations, we refer to the books \cite{Gr, Da, Sobczyk, Liu} and the papers \cite{Taniguchi, Luo, Ren, Chen1, Veraar, Br, Van, Veraar1, Dou}.
For the results of  Banach spaces,
Brze\'{z}niak \cite{Br} first established the theoretical framework of stochastic evolution equations in type 2 Banach spaces. van Neerven, Veraar, and Weis \cite{Van} established stochastic integral theory and maximal regularity in UMD spaces with type 2. This theoretical system was subsequently extended to non-autonomous systems by Veraar et al. \cite{Veraar, Veraar1}, which effectively solved the problem of existence of mild solutions for operators changing with time. For the non-autonomous stochastic evolution equations, Zhang and Feng \cite{Zhang} further investigated non-autonomous stochastic evolution equations with additive noise and obtained the well-posedness and asymptotic stability of mild solutions. Recently, Chen, Li, and Zhang \cite{Chen2}  used the non-compactness measure and fixed point techniques to solve stochastic non-autonomous evolution equations governed by noncompact evolution families.

Based on the theory of well-posedness, many researchers have focused on maximal regularity for non-autonomous stochastic evolution equations. Since the theory of maximal regularity includes operator terms, stochastic convolution, and the regularity estimates for solutions, it provides significant analytical tools for the study of nonlinear stochastic partial differential equations. van Neerven, Veraar, and Weis \cite{Van2} first established a theory of stochastic maximal $L^p$-regularity for autonomous operators in UMD spaces. Then they extended the results to non-autonomous operators with piecewise relative continuity in time and obtained the existence, uniqueness, and maximal regularity of strong solutions. T\d{a}, Yagi, and Yamamoto \cite{Ta} further investigated non-autonomous linear stochastic evolution equations in UMD spaces with type 2. Under suitable assumptions on the time dependence of the operators, they established a unique strict solution and maximal regularity. For stochastic parabolic equations with rough time-dependent coefficients, Portal and Veraar \cite{Portal} employed semigroup methods and partial differential equation methods to investigate semilinear stochastic parabolic problems driven by cylindrical Brownian motion. Under the condition that the stochastic coefficients are only progressively measurable with respect to time, they established $L^p(L^q)$ estimates and optimal space–time regularity. Subsequently, Agresti and Veraar \cite{Agresti} applied deterministic and stochastic maximal regularity methods to nonlinear parabolic stochastic evolution equations in critical spaces. They established a well-posedness theory for semilinear and quasi-linear problems, which can solve rough initial conditions and instantaneous regularization of solutions.

However, to the best of our knowledge,  the results on non-autonomous stochastic fractional evolution equations with generalized fractional derivative of Hilfer type remain limited. There are many difficulties to be solved. Firstly, classical Laplace transform methods and $C_0$-semigroup theory fail for the time dependence of the non-autonomous operator $A(t)$.  Secondly, the initial component associated with the generalized fractional derivative is singular as $t \to 0^+$, while the stochastic convolution also has singularity as $s\to t$. Hence, both the fixed point estimate and the stochastic convolution estimate are not directly applicable in the general continuous function space. Thirdly, the factorization method of the classical stochastic convolution relies on the properties of semigroups and convolution identities, which are unavailable for non-autonomous operators. Finally, the existing theory of stochastic $L^p$-regularity is mainly based on stochastic convolutions of autonomous systems or classical non-autonomous parabolic equations, and cannot be applied directly to the singular component generated by the generalized initial condition and the
non-autonomous fractional kernels.

In this paper, we establish the well-posedness and regularity theory for mild solutions of non-autonomous stochastic evolution equations \eqref{e1.1} with generalized fractional derivative in UMD spaces with type 2. The main contributions of this paper can be summarized in the following four aspects. First, we obtain a representation formula for solutions and operator theory for non-autonomous systems by Mittag-Leffler functions and Volterra integral equations. And we establish the existence and uniqueness of classical solutions for non-autonomous generalized fractional evolution equations. Second, to handle the singularities of the initial term and the stochastic convolution kernel, we introduce the space of weighted continuous functions. Then we prove the existence and uniqueness of mild solutions in the weighted space by using the Burkholder-Davis-Gundy (BDG) inequality and the Banach fixed point theorem. By separating the initial component from the regular component, we further obtain spatial regularity in fractional power spaces and temporal H\"{o}lder continuity of the regular component. By a localization argument, we further obtain a unique progressively measurable mild solution when the initial datum is merely strongly $\mathcal F_0$-measurable. Third, under deterministic maximal regularity assumptions, we establish stochastic maximal $L^p$-regularity for the linear non-autonomous fractional equation. To the best of our knowledge, this result is new in the stochastic fractional setting. Finally, a fixed point argument extends the maximal
regularity result to the semilinear problem.

The rest of this paper is organized as follows. In Section \ref{sec2}, we introduce some preliminary concepts, including fractional calculus,  sectorial operators, fractional power spaces, and stochastic integration in UMD spaces. In Section \ref{sec3}, we construct the solution operators for the non-autonomous linear equation and establish the existence and uniqueness of classical solutions. We prove the existence, uniqueness, and time-space regularity of mild solutions in the weighted continuous function space in Section~\ref{sec4}. Section~\ref{sec5} develops stochastic maximal $L^p$-regularity for linear and semilinear non-autonomous fractional equations. Finally, we apply our main results to a stochastic diffusion equation and a fractional reaction-diffusion SIR model in Section~\ref{sec6}.

\section{Preliminaries}\label{sec2}
In this section, we recall some basic notions concerning fractional integrals and derivatives, special functions, sectorial operators, and the function spaces. Let $X_0$ be a Banach space with norm $\|\cdot\|_{X_0}$. The symbol $L(X_0, Y)$ is the space of all bounded linear operators from $X_0$ to another Banach space $Y$ with norm $\|\cdot\|_{L(X_0, Y)}$ and let $L(X_0) := L(X_0,X_0)$, for short $\|\cdot\|_{L(X_0)}$. The domain and range of an operator $A$ are denoted by $D(A)$ and $R(A)$, respectively. We denote the resolvent set of $A$ by $\rho(A)$. For each $z\in\rho(A)$, the resolvent operator of $A$ is defined by $R(z;A):=(zI-A)^{-1}$.

\subsection{Probability Space, UMD Spaces, and   \texorpdfstring{$\gamma$}{gamma}-Radonifying Operators}
Let $(\Omega, \mathcal{F}, \{\mathcal{F}_t\}_{t \in [0,T]}, \mathbb{P})$ be a complete filtered probability space. For $p \ge 1$, we denote by $L^p(\Omega; X_0)$ the Banach space of all strongly $\mathcal{F}_T$-measurable, $p$-integrable $X_0$-valued random variables, equipped with the norm $\|u\|_{L^p(\Omega;X_0)} = (\mathbb{E}\|u\|_{X_0}^p)^{1/p}$.

A Banach space $X_0$ is said to have type 2 if there exists a constant $C_{T} > 0$ such that for all finite sequences $(x_n)_{n=1}^N \subset X_0$, it holds that$$\Big(\mathbb{E} \Big\| \sum_{n=1}^N k_n x_n \Big\|_{X_0}^2 \Big)^{1/2} \le C_{T} \Big( \sum_{n=1}^N \|x_n\|_{X_0}^2 \Big)^{1/2},$$
where $(k_n)_{n=1}^N$ is a Rademacher sequence  taking values in $\{-1, 1\}$. Furthermore, $X_0$ is called a UMD (unconditional martingale differences) space if for every $p \in (1, \infty)$, there exists a constant $C_p > 0$ such that for all $X_0$-valued martingale difference sequences $(d_n)_{n=1}^N$ and all signs $\varepsilon_n \in \{-1, 1\}$,$$\mathbb{E} \Big\| \sum_{n=1}^N \varepsilon_n d_n \Big\|_{X_0}^p \le C_p^p \, \mathbb{E} \Big\| \sum_{n=1}^N d_n \Big\|_{X_0}^p.$$

Let $H$ be a separable Hilbert space. We introduce the space of $\gamma$-radonifying operators to define the stochastic convolution driven by a cylindrical Brownian motion $W_H$.  A bounded linear operator $R \in L(H, X_0)$ is called a $\gamma$-radonifying operator if for some orthonormal basis $(h_n)_{n=1}^\infty$ of $H$, $\sum_{n=1}^\infty \gamma_n R h_n$ converges in $L^2(\Omega; X_0)$, where $(\gamma_n)_{n=1}^\infty$ is a sequence of independent Gaussian random variables on a probability space $(\Omega, \mathcal{F}, \mathbb{P})$. The space of all $\gamma$-radonifying operators from $H$ into $X_0$ is denoted by $\gamma(H, X_0)$. It is a Banach space equipped with the norm:
$$\|R\|_{\gamma(H, X_0)} := \Big( \mathbb{E}\Big\| \sum_{n=1}^\infty \gamma_n R h_n \Big\|_{X_0}^2 \Big)^{1/2}.$$

The following estimate will be used repeatedly for stochastic convolutions.
\begin{lemma}[{\cite[Proposition 2.8]{Veraar}}]\label{lem2.1}
Let $Y$ be a UMD space of type $2$, let $p\in(1,\infty)$, and let $W_H$ be an $H$-cylindrical Brownian motion. Suppose that
$\Phi:[0,T]\times\Omega\to\gamma(H,Y)$ is progressively measurable and $\Phi\in L^p(\Omega;L^2(0,T;\gamma(H,Y)))$. Then $\Phi$ is stochastically integrable with respect to $W_H$, its
stochastic integral has a continuous adapted version, and
\begin{equation*}
\left(\mathbb E\sup_{0\leq t\leq T}
\left\|\int_0^t\Phi(s)\,dW_H(s)\right\|_Y^p\right)^{1/p}
\leq C_{p,Y}
\|\Phi\|_{L^p(\Omega;L^2(0,T;\gamma(H,Y)))}.
\end{equation*}
\end{lemma}

\subsection{Fractional Calculus and Special Functions}
\begin{definition}[\cite{Podl}]\label{def1}
The Riemann-Liouville fractional integral of order $\lambda>0$ with the lower limit $a$ for a function $f :[a,+\infty)\to X_0$ is defined as
$$
I^\lambda_{a+} f(t)=\frac{1}{\Gamma(\lambda)}\int_a^t(t-s)^{\lambda-1}f(s)ds \quad t>a,
$$
provided the right side is pointwise defined on $[a,+\infty)$ and $\Gamma(\cdot)$ is the Gamma function.
\end{definition}

\begin{definition}[\cite{Podl}]
The Riemann-Liouville fractional derivative of order $\lambda>0$  for a function $f :[a,+\infty)\to X_0$ is defined as
$$
{}^{RL}\!D_{a+}^{\lambda} f(t)=\frac{1}{\Gamma(n-\lambda)} \frac{d^{n}}{d t^{n}} \int_{a}^{t} (t-s)^{n-\lambda-1}f(s) ds,
\quad t>a,\ n-1<\lambda<n.
$$
\end{definition}

\begin{definition}[\cite{Podl}]
The Caputo fractional derivative of order $\lambda>0$  for a function $f :[a,+\infty)\to X_0$ is defined as
$$
{}^C\!D^\lambda_{a+} f(t)={}^{RL}\!D_{a+}^\lambda\Big( f (t) - \sum_{k=0}^{n-1} \frac{f^{(k)}(a)}{k!}(t-a)^k\Big),\quad t>a,\ n-1<\lambda<n.
$$
\end{definition}

\begin{definition}[\cite{Hilfer}]
The generalized fractional derivative of Hilfer type with order $\mu\in(0,1)$ and type $\lambda\in[0,1]$ for a function $f :[a,+\infty)\to E$ is defined as
$$D^{\lambda,\mu}_{a+} f(t)=\left(I_{a+}^{\lambda(1-\mu)}\frac{d}{dt}I_{a+}^{(1-\lambda)(1-\mu)}f\right)(t),\quad t>a.
$$
\end{definition}

\begin{remark}\label{rem2.1}
\leavevmode
\begin{itemize}
\item [{\rm (i)}] When $\lambda=0$ and $0 <\mu < 1$, the generalized fractional derivative corresponds to the classical Riemann-Liouville fractional derivative
$$
D^{0,\mu}_{a+} f(t)=:{}^{RL}\!D^{\mu}_{a+} f(t).
$$
\item [{\rm (ii)}] When $\lambda=1$ and $0 <\mu < 1$, the generalized fractional derivative corresponds to the classical Caputo fractional derivative
$$
D^{1,\mu}_{a+} f(t)=:\!^CD^{\mu}_{a+} f(t).
$$
\end{itemize}
\end{remark}

The Wright-type function \cite{Mainardi} $M_\mu(z)$ is defined by
$$M_\mu(z)=\sum_{n=1}^\infty\frac{(-z)^{n-1}}{(n-1)!\Gamma(1-\mu n)},\quad0<\mu<1,\quad z\in\mathbb{C}.$$

The Mittag-Leffler function $E_{\alpha,\beta}(\cdot):\mathbb C\to\mathbb C$ for $\alpha>0$,  $\beta\in\mathbb R$ is defined by
$$
E_{\alpha,\beta}(z) =\sum_{k=0}^\infty \frac{z^k}{\Gamma(\alpha k +\beta)}= \frac{1}{2\pi i}\int_{\mathcal C}e^\lambda\frac{ \lambda ^{\alpha -\beta}}{\lambda^\alpha-z}d\lambda,\quad z \in\mathbb C,
$$
where $\mathcal C$ is a counterclockwise contour enclosing the disk of radius $|z|^{1/\alpha}$ and extending to $-\infty$ at both ends.

According to \cite[Theorem 1.3-1.4]{Podl}, the asymptotic expansion of $E_{\alpha,\beta}(z)$ as $z \to \infty$ for $0 < \alpha < 1$, $\beta > 0$, $N \ge 1$ is given by
$$E_{\alpha,\beta}(z) =  \begin{cases}  \alpha^{-1}z^{(1-\beta)/\alpha}e^{z^{1/\alpha}} + \epsilon_{\alpha,\beta}(z), & |\arg(z)| \le \frac{1}{2}\alpha\pi, \\  \epsilon_{\alpha,\beta}(z), & |\arg(-z)| \le \left(1 - \frac{\alpha}{2}\right)\pi,  \end{cases}$$
where$$\epsilon_{\alpha,\beta}(z) = -\sum_{k=1}^N \frac{z^{-k}}{\Gamma(\beta - \alpha k)} + O(|z|^{-N-1}), \quad \text{as } z \to \infty.$$

\begin{lemma}[{\rm\cite{Podl}}]\label{lem2.2}
If $0 <\alpha < 2$ and $\beta\in\mathbb R$, $\pi\alpha/2 < \omega < \min(\pi,\pi\alpha)$, then there is a constant $C>0$ such that
$$\left|E_{\alpha,\beta}(z)\right|\leq \frac{C}{1 + |z|},\quad z\in\mathbb{C},~~\omega\leq|{\rm arg}z|\leq\pi.
$$
\end{lemma}

\begin{lemma}[\cite{He2024}]\label{lem2.3}
For $\alpha\in(0,1)$, the following hold.
\begin{itemize}
  \item [{\rm(i)}] $M_\alpha(\upsilon)\geq0$, where~$\upsilon\in[0,\infty).$
  \item [{\rm(ii)}]
      $\displaystyle\int_0^\infty M_\alpha(\upsilon)e^{-z\upsilon}d\upsilon=E_{\alpha,1}(-z),~~z\in \mathbb{C}.$
  \item [{\rm(iii)}]$\displaystyle\int_0^\infty\alpha\upsilon M_\alpha(\upsilon)e^{-z\upsilon}d\upsilon=E_{\alpha,\alpha}(-z),~~z\in \mathbb{C}.$
  \item [{\rm(iv)}] For~$-1<\delta<\infty$, we have $\int^{\infty}_{0}\upsilon^{\delta} M_\alpha(\upsilon)d\upsilon=\frac{\Gamma
(1+\delta)}{\Gamma (1+\alpha\delta)}$.


\end{itemize}
\end{lemma}

\subsection{Sectorial Operators}
\begin{definition}\label{def2}
Let $0<\omega<\frac{\pi}{2}$.
For every  $t \geq 0$, let $A(t)$  be a closed densely defined linear operator in the Banach space $X_0$, we call $A(t)$ is a sectorial operator if the spectral set
$\sigma(A(t))$ is contained in an open sectorial domain
$S_\omega:=\{z \in \mathbb{C} \backslash\{0\}:|\arg z|<\omega\}$ and
its resolvent satisfies the estimate
\begin{equation}\label{e2.1}
\|(z-A(t))^{-1}\|_{L(X_0)} \leq \frac{C}{1+|z|}, \quad z \notin S_\omega,
\end{equation}
where the constant $C\geq1$ is independent of $t\in[0,T]$. Moreover,
$$\sup_{t\in[0,T]}\|A(t)^{-1}\|_{L(X_0)}\leq C.$$
\end{definition}

Since  $A(\cdot)$  is a sectorial operator, it is known that the resolvent set  $\rho(A(\cdot))$  of  $A(\cdot)$  contains  $\Sigma^{\prime}= \{z \in \mathbb{C}: \theta \leq|\arg z| \leq \pi\} \cup\{0\}, 0<\theta<\pi / 2$  and one can define the linear bounded operators  $\left\{T_{A(\tau)}(t)\right\}_{t \geq 0}$  by Dunford integral:
\begin{equation*}
T_{A(\tau)}(t)=\frac{1}{2 \pi i} \int_{\mathcal{C}} e^{-t z}(z-A(\tau))^{-1} d z
\end{equation*}
where  $\mathcal{C}$ is a smooth contour running in  $\Sigma^{\prime}$  from  $\infty e^{i \theta}$  to  $\infty e^{-i \theta}$  oriented counter-clockwise. Then  $T_{A(\tau)}(t)$  is an analytic semigroup for  $t \geq 0$  generated by  $-A(\tau)$. Moreover, for $t>0$, there is a constant $C> 0$ such that
\begin{equation}\label{e2.2}
\|T_{A(\tau)}(t)\|_{L(X_0)}\leq C,\quad \|A(\tau)T_{A(\tau)}(t)\|_{L(X_0)}\leq \frac{C}{t}.
\end{equation}

Furthermore, for $0 < \alpha < 1$, we introduce fractional powers $A(t)^{-\alpha} \in {L}(X_0)$ of operator $A(t)$ as$$A(t)^{-\alpha} := \frac{\sin \pi \alpha}{\pi} \int_0^\infty s^{-\alpha}(sI + A(t))^{-1}ds, \quad 0 \leq t \leq T. $$
We set $A(t)^\alpha$ as the inverse of $A(t)^{-\alpha}$ with the domain $D(A(t)^\alpha)$.

For every $s>0$, there exists a
constant $C>0$, independent of $t\in[0,T]$, such that
\begin{align}\label{e2.3}
\|A(t)^\alpha T_{A(t)}(s)\|_{{L}(X_0)} \leq C s^{-\alpha}.
\end{align}

For a fixed terminal time $T > 0$ and $1 \le p \le \infty$, we denote by $L^p(0, T; X_0)$ the Lebesgue space of all strongly measurable functions $f: [0, T] \to X_0$ such that $\|f(\cdot)\|_{X_0} \in L^p(0, T; \mathbb{R})$. $C([0, T]; X_0)$ denotes the space of all continuous functions from $[0, T]$ to $X_0$, equipped with the norm $$\|f\|_{C([0, T]; X_0)} := \sup_{t \in [0, T]} \|f(t)\|_{X_0}.$$

For any $\rho \in (0, 1]$, we introduce the H\"{o}lder seminorm for a function $f: [0, T] \to X_0$ defined by
$$[f]_{C^\rho([0, T]; X_0)} := \sup_{t, r \in [0, T], t \neq r} \frac{\|f(t) - f(r)\|_{X_0}}{|t - r|^\rho}.$$
The corresponding H\"{o}lder space $C^\rho([0, T]; X_0)$ is the Banach space of all continuous functions $f$ for which this seminorm is finite, equipped with the norm
$$\|f\|_{C^\rho([0, T]; X_0)} := \|f\|_{C([0, T]; X_0)} + [f]_{C^\rho([0, T]; X_0)}.$$
We say $f$ is locally $\rho$-H\"{o}lder continuous on $(0, T]$ if $f \in C^\rho([\varepsilon, T]; X_0)$ for $\varepsilon > 0$.

In order to obtain the results in this paper, we use the following hypotheses of $A(t)$.
\begin{assumption}\label{assump2.1}
\leavevmode
\begin{itemize}
\item [$(A1)$] Operator $A(t)$ is sectorial, and its domain $D(A(t))=X_1$ is independent of $t$ for $t\in[0,T]$.
Moreover, the graph norms are uniformly equivalent, that is, there exist $C_1,C_2>0$ such that
$$C_1\|x\|_{X_1}\leq \|x\|_{X_0}+\|A(t)x\|_{X_0}
\leq C_2\|x\|_{X_1},
\quad x\in X_1,\,\, t\in[0,T].$$
\item [$(A2)$] Operator $A(t)$ is H\"{o}lder continuous, i.e., there exists $C>0$ and $\vartheta\in(0,1)$ such that
\begin{equation}\label{e2.4}
\|(A(t)-A(s)) A(\tau)^{-1}\|_{L(X_0)}\leq C |t-s|^\vartheta,\ \ \ \forall t,s,\tau\in[0,T].
\end{equation}
\end{itemize}
\end{assumption}

\subsection{Function Spaces}
For $\alpha\in(0,1)$ and $p\in[1,\infty]$, let
$$X_\alpha=[X_0,X_1]_\alpha,
\quad X_{\alpha,p}=(X_0,X_1)_{\alpha,p}$$
denote the complex and real interpolation spaces, respectively. By convention, $[X_0,X_1]_0=X_0$, $[X_0,X_1]_1=X_1$.

\begin{assumption}\label{assump2.2}
Assume the space $X_0$ is a UMD space with type $2$.
Suppose that each $A(t)$ admits a bounded
$H^\infty$-calculus of angle $\phi\in(0,\pi/2)$ and that the corresponding $H^\infty$-calculus bounds are uniform in $t\in[0,T]$.
\end{assumption}

\begin{lemma}\label{lem2.4}
Let Assumptions~\ref{assump2.1} and~\ref{assump2.2} hold. Then, for every $\alpha\in[0,1]$,
\begin{equation}\label{e2.5}
D(A(t)^\alpha)=X_\alpha,
\end{equation}
for $t\in[0,T]$, that is, there exists a constant $C_\alpha>0$ independent of $t\in[0,T]$ such that
$$C_\alpha^{-1}\|x\|_{X_\alpha}
\leq \|A(t)^\alpha x\|_{X_0}
\leq C_\alpha\|x\|_{X_\alpha},
\quad .$$
for $x\in X_\alpha$ and $t\in[0,T]$.
Moreover, $X_\alpha$ is a UMD space with type $2$, and
$X_\beta\hookrightarrow X_\alpha$ for $0\leq\alpha<\beta\leq1$.
\end{lemma}
\begin{proof}
We first consider $0<\alpha<1$. By Assumption \ref{assump2.2}, we have
$D(A(t)^\alpha)=
[X_0,D(A(t))]_\alpha$.
Combining with Assumption \ref{assump2.1}, we obtain $D(A(t)^\alpha)
=[X_0,X_1]_\alpha
=X_\alpha$
and for $x\in X_\alpha$, $t\in[0,T]$,
$$C_\alpha^{-1}\|x\|_{X_\alpha}
\leq
\|A(t)^\alpha x\|_{X_0}
\leq
C_\alpha\|x\|_{X_\alpha}.$$
For $\alpha=0$, $D(A(t)^0)=X_0$. For $\alpha=1$, $D(A(t))=X_1$. Then the results also hold.

Since $A(0):X_1\to X_0$ is an isomorphism, we deduce that $X_1$ is a UMD space with type $2$. Since both the UMD property and type 2 are preserved under complex interpolation, $X_\alpha=[X_0,X_1]_\alpha$ is a UMD space with type $2$. Finally, for $0\le\alpha<\beta\le1$ and $x\in X_\beta$, it follows from the boundedness of the negative fractional power $A(t)^{-(\beta-\alpha)}$ that
$$\|x\|_{X_\alpha}
\le C_{\alpha,\beta}\|x\|_{X_\beta}.$$
Then we deduce $X_\beta\hookrightarrow X_\alpha$. The proof is complete.
\end{proof}

\begin{corollary}\label{cor2.1}
Let Assumptions~\ref{assump2.1} and~\ref{assump2.2} hold. For
$\alpha\in[0,1]$ and $t\in[0,T]$, the map
\[
\gamma(H,X_\alpha)\to\gamma(H,X_0),
\quad R\mapsto A(t)^\alpha R,
\]
is an isomorphism. Moreover, there exists a constant $C_\alpha>0$ independent of $t$, such that
\begin{equation}\label{e2.6}
C_\alpha^{-1}\|R\|_{\gamma(H,X_\alpha)}
\leq \|A(t)^\alpha R\|_{\gamma(H,X_0)}
\leq C_\alpha\|R\|_{\gamma(H,X_\alpha)}
\end{equation}
for $R\in\gamma(H,X_\alpha)$. Throughout the paper, $A(t)^\alpha R$ denotes the composition
$A(t)^\alpha\circ R$.
\end{corollary}
\begin{proof}
By Lemma~\ref{lem2.4}, $A(t)^\alpha:X_\alpha\to X_0$ is an
isomorphism, and the norms of $A(t)^\alpha$ in $L(X_\alpha,X_0)$ and
of $A(t)^{-\alpha}$ in $L(X_0,X_\alpha)$ are bounded uniformly in $t$.
For $R\in\gamma(H,X_\alpha)$, it follows from the ideal property of
$\gamma$-radonifying operators that $A(t)^\alpha R\in\gamma(H,X_0)$ and
$$\|A(t)^\alpha R\|_{\gamma(H,X_0)}
\leq\|A(t)^\alpha\|_{L(X_\alpha,X_0)}
\|R\|_{\gamma(H,X_\alpha)}
\leq C_\alpha\|R\|_{\gamma(H,X_\alpha)}.$$
Moreover, we deduce $R=A(t)^{-\alpha}(A(t)^\alpha R)$. It yields
\begin{align*}
\|R\|_{\gamma(H,X_\alpha)}\leq
\|A(t)^{-\alpha}\|_{L(X_0,X_\alpha)}
\|A(t)^\alpha R\|_{\gamma(H,X_0)}\leq C_\alpha
\|A(t)^\alpha R\|_{\gamma(H,X_0)}.
\end{align*}
Hence, we deduce \eqref{e2.6}. The proof is complete.
\end{proof}

Let $X$ be a UMD space, $1<p<\infty$, and $s\in\mathbb R$. The $X$-valued Bessel potential space on $\mathbb R$ is defined by
$$
H^{s,p}(\mathbb R;X)
=
\left\{
u\in\mathcal S'(\mathbb R;X):
\mathcal F^{-1}
\left[
(1+|\xi|^2)^{s/2}\widehat u(\xi)
\right]
\in L^p(\mathbb R;X)
\right\},
$$
where $\mathcal F$ denotes the Fourier transform and $\mathcal S'(\mathbb R;X)$ is the space of $X$-valued tempered distributions. It is equipped with the norm
$$
\|u\|_{H^{s,p}(\mathbb R;X)}
=
\left\|
\mathcal F^{-1}
\left[
(1+|\xi|^2)^{s/2}\widehat u
\right]
\right\|_{L^p(\mathbb R;X)}.
$$For $I=(0,T)$, the restriction space is defined by $$H^{s,p}(I;X)=\left\{
u|_I:u\in H^{s,p}(\mathbb R;X)
\right\}$$ with the norm
$\|u\|_{H^{s,p}(I;X)}=\inf
\left\{\|\widetilde u\|_{H^{s,p}(\mathbb R;X)}:
\widetilde u|_I=u
\right\}$.
We use the convention
$H^{0,p}(I;X)=L^p(I;X)$.
If $s>1/p$, the trace operator
$\text{tr}_0:H^{s,p}(I;X)\to X,
\text{tr}_0u=u(0)$ is well defined and bounded. In this case, we set
$${}_0H^{s,p}(I;X)=
\left\{u\in H^{s,p}(I;X):
\text{tr}_0u=0
\right\}.
$$

\begin{lemma}\label{lem2.5}
Let $X_0$ be a UMD space, $1<p<\infty$, $0<\mu<1$, and $\mu>\frac1p$.
Then the Riemann-Liouville fractional integral $(I_{0+}^{\mu}f)(t)
=
\frac{1}{\Gamma(\mu)}
\int_0^t(t-s)^{\mu-1}f(s)\,ds$
defines a bounded operator
$$
I_{0+}^{\mu}:
L^p(0,T;X_0)
\to
{}_0H^{\mu,p}(0,T;X_0),
$$
that is, there exists $C=C(\mu,p,T,X_0)>0$ such that
$$
\|I_{0+}^{\mu}f\|_{H^{\mu,p}(0,T;X_0)}
\leq
C\|f\|_{L^p(0,T;X_0)}.
$$\
\end{lemma}
\begin{proof}
Consider the operator $B$ on $L^p(0,T;X_0)$, defined by
$Bu:=\partial_tu=u'$ with domain
$$D(B):={}_0W^{1,p}(0,T;X_0)=
\left\{
u\in W^{1,p}(0,T;X_0):u(0)=0
\right\}.$$
The operator $-B$ generates the semigroup $(S(r))_{r\geq0}$ and
$$(S(r)f)(t)=
\begin{cases}
f(t-r),&r<t,\\
0,&r\geq t.
\end{cases}
$$
Let $S(r)=e^{-rB}$. Since $B$ is invertible and sectorial, its negative fractional power is
$$
B^{-\mu}f=\frac{1}{\Gamma(\mu)}
\int_0^\infty r^{\mu-1}e^{-rB}f\,dr.
$$
For $t\in(0,T)$, we obtain
$$
(B^{-\mu}f)(t)=
\frac{1}{\Gamma(\mu)}
\int_0^t r^{\mu-1}f(t-r)\,dr=
\frac{1}{\Gamma(\mu)}
\int_0^t(t-s)^{\mu-1}f(s)\,ds=
(I_{0+}^{\mu}f)(t).
$$
Therefore, we deduce $B^{-\mu}=I_{0+}^{\mu}$.
For $\mu>1/p$, by \cite[Theorem~6.8]{Lindemulder} and its finite interval counterpart obtained by restriction and bounded extension, we obtain $D(B^\mu)={}_0H^{\mu,p}(0,T;X_0)$. Moreover, the graph norm of $B^\mu$ is equivalent to the norm of
${}_0H^{\mu,p}(0,T;X_0)$ and
$$
\|u\|_{{}_0H^{\mu,p}(0,T;X_0)}
\leq C\left(\|u\|_{L^p(0,T;X_0)}+
\|B^\mu u\|_{L^p(0,T;X_0)}
\right).$$
Set $u=I_{0+}^{\mu}f=B^{-\mu}f$. Then $u\in D(B^\mu)$ and $B^\mu u=B^\mu B^{-\mu}f=f$.

Let $g_\mu(t)=t^{\mu-1}/\Gamma(\mu)$. Since $g_\mu\in L^1(0,T)$, it follows from Young's convolution inequality that
\begin{align*}
\|I_{0+}^{\mu}f\|_{L^p(0,T;X_0)}
\leq
\|g_\mu\|_{L^1(0,T)}
\|f\|_{L^p(0,T;X_0)}
=\frac{T^\mu}{\Gamma(\mu+1)}
\|f\|_{L^p(0,T;X_0)}.
\end{align*}
Hence, we obtain
\begin{align*}
\|I_{0+}^{\mu}f\|_{{}_0H^{\mu,p}(0,T;X_0)}\leq
C\left(
\|I_{0+}^{\mu}f\|_{L^p(0,T;X_0)}+
\|B^\mu I_{0+}^{\mu}f\|_{L^p(0,T;X_0)}
\right)\leq C_{\mu,p,T,X_0}
\|f\|_{L^p(0,T;X_0)}.
\end{align*}
The proof is complete.
\end{proof}

\section{ A representation of solutions  to linear problems}\label{sec3}
In this section, we consider the following non-autonomous problem with
the generalized fractional derivative
\begin{equation}\label{e3.1}
 \left\{
\begin{aligned}
&D_{0+}^{\lambda,\mu}u(t)+A(t)u(t)=f(t),\\
&I_{0+}^{(1-\lambda)(1-\mu)}u(0)=u_0,
\end{aligned}\right.
\end{equation}
where $0<\mu<1$, $0\leq\lambda\leq1$. Let $\nu=\mu+\lambda(1-\mu)$, then $1-\nu=(1-\lambda)(1-\mu)$.
For $s\in[0,T]$ and $t>0$, denote operators by
$$\psi_{\mu,A(s)}(t) = t^{\mu-1} \frac{1}{2\pi i} \int_{\mathcal C} E_{\mu, \mu}(-z t^\mu) (z - A(s))^{-1} dz,$$
$$\psi_{\nu,A(s)}(t) = t^{\nu-1} \frac{1}{2\pi i} \int_{\mathcal C} E_{\mu, \nu}(-z t^\mu) (z - A(s))^{-1} dz,$$
and
$$S_{\mu,A(s)}(\tau)= \frac{1}{2\pi i} \int_{\mathcal{C}} E_{\mu, 1}(-z \tau^\mu) (z - A(s))^{-1} dz.$$

\begin{lemma}[\cite{He2022}]\label{lem3.1}
Let \eqref{e2.2}-\eqref{e2.4} be satisfied. For $t, t_1, t_2,\tau,\tau_1, \tau_2\in(0,T]$, the following results hold.
\begin{itemize}
\item [{\rm (i)}]
$\psi_{\mu,A(\tau)}(t)= t^{\mu-1}\int_0^\infty \mu\upsilon M_\mu(\upsilon)T_{A(\tau)}(t^\mu\upsilon)d\upsilon$.
\item [{\rm (ii)}] $\|\psi_{\mu,A(\tau)}(t)\|_{L(X_0)} \le C t^{\mu-1}$.
\item [{\rm (iii)}] $\|A(\tau)\psi_{\mu,A(\tau)}(t)\|_{L(X_0)} \le C t^{-1}$.
\item [{\rm (iv)}]
$\|(A(t_1)-A(t_2)) \psi_{\mu,A(\tau)}(t)  \|_{L(X_0)}\le C |t_1-t_2|^\vartheta t^{-1}$.
\item [{\rm (v)}]
$\|A(\tau)(\psi_{\mu,A(\tau_1)}(t)-\psi_{\mu,A(\tau_2)}(t))\|_{L(X_0) }\le C |\tau_2-\tau_1|^\vartheta t^{-1}$.
\item [{\rm (vi)}]
$\| A(\tau)(\psi_{\mu,A(\tau)}(t_2) -\psi_{\mu,A(\tau)}(t_1))\|_{L(X_0)}\le C |t_1^{-1}-t_2^{-1}|$.
\item [{\rm (vii)}] $\|S_{\mu,A(\tau)}(t)\|_{L(X_0)} \le C $ and $S_{\mu,A(\tau)}(0)=I$.
\item [{\rm (viii)}] $\frac{d}{dt} S_{\mu,A(\tau)}(t) = -A(\tau)\psi_{\mu,A(\tau)}(t)$.
\item [{\rm (ix)}] $\|S_{\mu,A(\tau_1)}(t)-S_{\mu,A(\tau_2)}(t)\|_{L(X_0)} \le C|\tau_1-\tau_2|^{\vartheta}.$
\end{itemize}
\end{lemma}

\begin{lemma}\label{lem3.2}
Let \eqref{e2.2}-\eqref{e2.4} be satisfied. For $s,t, r_1, r_2\in(0,T]$, the following assertions hold.
\begin{itemize}
\item [{\rm (i)}] $\|\psi_{\nu,A(s)}(t)\|_{L(X_0)}\leq Ct^{\nu-1}$.
\item [{\rm (ii)}]
$\|A(r_1)\psi_{\mu,A(r_1)}(t)-
A(r_2)\psi_{\mu,A(r_2)}(t)
\|_{ L(X_0)}\leq
C|r_1-r_2|^{\vartheta}t^{-1}$.
\item [{\rm (iii)}]
For every $x\in X_0$, we have
$$\lim_{t\to0}
t^{1-\mu}\psi_{\mu,A(s)}(t)x
=\frac{x}{\Gamma(\mu)},\quad \lim_{t\to0}
t^{1-\nu}\psi_{\nu,A(s)}(t)x
=\frac{x}{\Gamma(\nu)} \,\,\text{and}\,\,\lim_{t\to0}
I_{0+}^{1-\nu}
\bigl(\psi_{\nu,A(s)}(\cdot)x\bigr)(t)
=x.$$
\end{itemize}
\end{lemma}
\begin{proof}
If $\lambda=0$, then $\nu=\mu$ and $\psi_{\nu,A(s)}(t)=\psi_{\mu,A(s)}(t)$. Therefore, it follows from Lemma \ref{lem3.1}(ii) that (i) holds. For $0<\lambda \leq 1$, we have $\nu-\mu>0$. Then
$$\psi_{\nu,A(s)}(t) = I_{0+}^{\nu-\mu} \psi_{\mu,A(s)}(t) = \frac{1}{\Gamma(\nu-\mu)} \int_0^t (t-\tau)^{\nu-\mu-1} \psi_{\mu,A(s)}(\tau) d\tau.$$
By Lemma \ref{lem3.1} (ii), we deduce
\begin{align*}
\|\psi_{\nu,A(s)}(t)\|_{L(X_0)}
\le \frac{C}{\Gamma(\nu-\mu)} \int_0^t (t-\tau)^{\nu-\mu-1} \tau^{\mu-1} d\tau\le \frac{C \cdot B(\nu-\mu, \mu)}{\Gamma(\nu-\mu)} t^{\nu-1}\le  t^{\nu-1}.
\end{align*}
Therefore, we have (i). Next, we check (ii). Note that
\begin{align*}
&A(r_1)\psi_{\mu,A(r_1)}(t) - A(r_2)\psi_{\mu,A(r_2)}(t) \\
=& A(r_1)\psi_{\mu,A(r_1)}(t) - A(r_2)\psi_{\mu,A(r_1)}(t) + A(r_2)\psi_{\mu,A(r_1)}(t) - A(r_2)\psi_{\mu,A(r_2)}(t) \\
=& (A(r_1) - A(r_2))\psi_{\mu,A(r_1)}(t) + A(r_2)(\psi_{\mu,A(r_1)}(t) - \psi_{\mu,A(r_2)}(t)).
\end{align*}
Then we have
\begin{align*}
&\|A(r_1)\psi_{\mu,A(r_1)}(t) - A(r_2)\psi_{\mu,A(r_2)}(t)\|_{L(X_0)} \\
\leq &\ \|(A(r_1) - A(r_2))\psi_{\mu,A(r_1)}(t)\|_{L(X_0)} + \|A(r_2)(\psi_{\mu,A(r_1)}(t) - \psi_{\mu,A(r_2)}(t))\|_{L(X_0)}. \end{align*}
It follows from the Lemma \ref{lem3.1} (iv) and (v) that
$$\|A(r_1)\psi_{\mu,A(r_1)}(t)- A(r_2)\psi_{\mu,A(r_2)}(t)\|_{L(X_0)} \leq 2C|r_1-r_2|^{\vartheta}t^{-1}.$$

Note that
$t^{1-\mu}\psi_{\mu,A(s)}(t)x = \int_0^\infty \mu v M_\mu(v) T_{A(s)}(t^\mu v)x dv$. Then, by Lemma \ref{lem2.3}, we obtain
$$\lim_{t \to 0} t^{1-\mu}\psi_{\mu,A(s)}(t)x = \int_0^\infty \mu v M_\mu(v) x dv= x \left( \mu \cdot \frac{1}{\mu\Gamma(\mu)} \right) = \frac{x}{\Gamma(\mu)}.$$
If $\lambda=0$, then $\nu=\mu$ and
$\lim_{t\to0}t^{1-\nu}\psi_{\nu,A(s)}(t)x=\frac{x}{\Gamma(\nu)}$.
For $\lambda>0$, and set $b:=\nu-\mu=\lambda(1-\mu)>0$.
Since $\psi_{\nu,A(s)}=
I_{0+}^{b}\psi_{\mu,A(s)}$, we have
$$\psi_{\nu,A(s)}(t)x
=\frac1{\Gamma(b)}\int_0^t
(t-\tau)^{b-1}\psi_{\mu,A(s)}(\tau)x\,d\tau.$$
Let $\tau=tr$, for every fixed $r\in(0,1]$, we obtain
$$t^{1-\nu}\psi_{\nu,A(s)}(t)x=
\frac1{\Gamma(b)}
\int_0^1
(1-r)^{b-1}r^{\mu-1}
\left[
(tr)^{1-\mu}
\psi_{\mu,A(s)}(tr)x
\right]dr,$$
and $\lim_{t\to0}
(tr)^{1-\mu}
\psi_{\mu,A(s)}(tr)x
=\frac{x}{\Gamma(\mu)}$. By Lemma \ref{lem3.1}(ii), we have
$$\left\|
(tr)^{1-\mu}
\psi_{\mu,A(s)}(tr)x
\right\|_{X_0}\le(tr)^{1-\mu}
\|\psi_{\mu,A(s)}(tr)\|_{L(X_0)}
\|x\|_{X_0}\le C\|x\|_{X_0}.$$
It follows from Lebesgue dominated convergence theorem that
\begin{align*}
\lim_{t\to0}
t^{1-\nu}\psi_{\nu,A(s)}(t)x=
\frac{x}{\Gamma(b)\Gamma(\mu)}
\int_0^1
(1-r)^{b-1}r^{\mu-1}\,dr=
\frac{x}{\Gamma(\nu)}.
\end{align*}

If $\lambda=0$, then $\nu=\mu$, $I_{0+}^{1-\nu}\psi_{\nu,A(s)} = I_{0+}^{1-\mu}\psi_{\mu,A(s)} = S_{\mu,A(s)}$ with $S_{\mu,A(s)}(\tau)=E_{\mu,1}\bigl(-A(s)\tau^\mu\bigr).$

For $0<\lambda < 1$, since
\begin{align*}
I_{0+}^{1-\nu} (\psi_{\nu,A(s)}(\cdot)x)(t) &= \frac{1}{2\pi i} \int_{\mathcal C} \left[ I_{0+}^{1-\nu} \left( \tau^{\nu-1} E_{\mu,\nu}(-z \tau^\mu) \right) \right](t) (z - A(s))^{-1} x dz\\
& = \frac{1}{2\pi i} \int_{\mathcal C} E_{\mu,1}(-z t^\mu) (z - A(s))^{-1} xdz=E_{\mu,1}\bigl(-A(s)t^\mu\bigr).
\end{align*}
It follows from $E_{\mu,1}(0) = I$ that
$\lim_{t \to 0} E_{\mu,1}(-A(s)t^\mu)x = I x = x$. For $\lambda=1$,  $\nu=1$ and
\begin{align*}
\lim_{t \to 0} I_{0+}^{1-\nu}(\psi_{\nu,A(s)}(\cdot)x)(t) = \lim_{t \to 0} E_{\mu,1}(-A(s)t^\mu)x = x.
\end{align*}
Hence, we have (iii). The proof is complete.
\end{proof}

\begin{lemma}\label{lem3.3}
Let \eqref{e2.2}-\eqref{e2.4} be satisfied. For $s,t\in(0,T]$, we have
\begin{align}\label{e3.2}
 \|A(s)\psi_{\nu,A(s)}(t)\|_{L(X_0)}
\leq Ct^{\lambda(1-\mu)-1}.
\end{align}
Furthermore, for $0<t_1<t_2\leq T$, we deduce
$$\|A(s)\left(\psi_{\nu,A(s)}(t_2) - \psi_{\nu,A(s)}(t_1)\right)\|_{L(X_0)} \le C t_1^{\nu-\mu-2}(t_2 - t_1).$$
\end{lemma}
\begin{proof}
We first consider the case of $\lambda>0$. Define the integration contour $\mathcal{C}_{\theta,\delta_0}$. Let $\omega$ be the spectral angle of the sectorial operator $A(s)$, and choose a fixed angle $\theta \in (\omega, \pi/2)$ along with a sufficiently small, fixed radius $\delta_0 > 0$. The contour $\mathcal{C}_{\theta,\delta_0}$, oriented counter-clockwise and running entirely within the resolvent set $\rho(A(s))$, consists of three parts
$$\mathcal{C}_{\theta,\delta_0} = \mathcal{C}_1 \cup \mathcal{C}_2 \cup \mathcal{C}_3,$$
where $\mathcal{C}_1 = \{ r e^{-i\theta} : r \in [\delta_0, \infty) \}$, $\mathcal{C}_2 = \{ \delta_0 e^{i\varphi} : \varphi \in [-\theta, \theta] \}$, and $\mathcal{C}_3 = \{ r e^{i\theta} : r \in [\delta_0, \infty) \}$.

Let $b = \nu - \mu = \lambda(1 - \mu) > 0$. Fix $\theta \in (0, \pi/2)$ and denote the closed sector
$$ \Sigma_\theta = \{y \in \mathbb{C} \setminus \{0\} : |\arg y| < \theta\}. $$
For any $y \in \overline{\Sigma_\theta}$, we have $|\arg(-y)| \ge \pi - \theta > \mu\pi/2$. Set $\varphi_0\in (\frac{\mu\pi}{2},\mu\pi)$, then  $\varphi_0 \le |\arg z| \le \pi$. By the asymptotic expansion of Mittag-Leffler function, we have
$$E_{\mu,\nu}(z) = -\sum_{k=1}^N \frac{z^{-k}}{\Gamma(\nu-\mu k)} + O(|z|^{-N-1}).$$
Taking $z = -y$ and  $N = 1$, we obtain
$$E_{\mu,\nu}(-y) = -\frac{(-y)^{-1}}{\Gamma(\nu-\mu)} + O(|y|^{-2}) = \frac{1}{\Gamma(\nu-\mu)}\frac{1}{y} + O(|y|^{-2})$$
as $|y| \to \infty$. Let $c_b = 1/\Gamma(b)$, then for $|y| \to \infty$ and $y \in \overline{\Sigma_\theta}$, we get $E_{\mu,\nu}(-y) = \frac{c_b}{y} + O(|y|^{-2})$.

Define
$$ R_b(y) = E_{\mu,\nu}(-y) - \frac{c_b}{1+y}, $$
then we have
$$ R_b(y) = \left( E_{\mu,\nu}(-y) - \frac{c_b}{y} \right) + \frac{c_b}{y(1+y)}. $$
Moreover, since $|\arg y| \le \theta < \pi/2$, set $y = |y|e^{i\phi}$ with $\phi = \arg y$,
$|1+y| \ge \cos\theta \, (1+|y|)$. It follows from  $|y| \to \infty$ that
$$\frac{c_b}{|y| \cdot |1+y|} \le \frac{c_b}{|y| \cdot \cos\theta(1+|y|)}\le \frac{c_b}{\cos\theta} \cdot \frac{1}{|y|^2}.$$
Combining with the $O(|y|^{-2})$, we deduce
$|R_b(y)| \le C|y|^{-2}$ for all sufficiently large $y \in \overline{\Sigma_\theta}$. On every bounded subset of $\overline{\Sigma_\theta}$, the function $R_b$ is bounded because $E_{\mu,\nu}$ is entire and $-1 \notin \overline{\Sigma_\theta}$. Hence, for $y \in \overline{\Sigma_\theta}$, we get
$$ |R_b(y)| \le \frac{C}{1+|y|^2}. $$
Let $y = zt^\mu$, we deduce
$$\psi_{\nu,A(s)}(t) = t^{\nu-1} \frac{1}{2\pi i} \int_{\mathcal{C}_{\theta,\delta_0}} \left[ \frac{c_b}{1+zt^\mu} + R_b(zt^\mu) \right] (z - A(s))^{-1} dz.$$
It follows from Dunford integral that
$$\frac{1}{2\pi i} \int_{\mathcal{C}_{\theta,\delta_0}} \frac{1}{1+zt^\mu} (z - A(s))^{-1} dz= (I + t^\mu A(s))^{-1}.$$
Then we have
\begin{align*}
\psi_{\nu,A(s)}(t) =
c_b t^{\nu-1} (I + t^\mu A(s))^{-1} + t^{\nu-1} \mathcal{R}_b(t,s),
\end{align*}
where the integral operator is defined as
$\mathcal{R}_b(t,s) = \frac{1}{2\pi i} \int_{\mathcal{C}_{\theta,\delta_0}} R_b(zt^\mu) (z - A(s))^{-1} dz$.
By \eqref{e2.1}, we have $\|A(s)(z-A(s))^{-1}\|_{L(X_0)} <C$. From
the estimate of $R_b$, we deduce
$$\int_{\mathcal{C}_{\theta, \delta_0}} |R_b(zt^\mu)| \|A(s)(z - A(s))^{-1}\|_{L(X_0)} dz < \infty.$$
Then we have
\begin{align*}
&A(s)\psi_{\nu,A(s)}(t) \\
=& c_b t^{\nu-1} A(s)(I + t^\mu A(s))^{-1} + t^{\nu-1} \frac{1}{2\pi i} \int_{\mathcal{C}_{\theta,\delta_0}} R_b(zt^\mu) A(s)(z - A(s))^{-1} dz= I_1+I_2.
\end{align*}
Together with $\|t^\mu A(s)(I + t^\mu A(s))^{-1}\|_{L(X_0)} \le C$, we get
$$\|I_1 \|_{L(X_0)} = c_b t^{\nu-\mu-1} \|t^\mu A(s)(I + t^\mu A(s))^{-1}\|_{L(X_0)} \le C t^{\nu-\mu-1}.$$
For $\mathcal{C}_2$, $z = \delta_0 e^{i\varphi}$ with $\varphi \in [-\theta, \theta]$ and $|dz| = \delta_0 d\varphi$, the variable $|z| = \delta_0$ is constant. Then
$$\left\|t^{\nu-1} \frac{1}{2\pi i} \int_{\mathcal{C}_2} R_b(zt^\mu) A(s)(z - A(s))^{-1} dz \right\|_{L(X_0)} \le t^{\nu-1} \frac{1}{2\pi} \int_{-\theta}^{\theta} C \cdot C \cdot \delta_0 d\varphi \le C' t^{\nu-1}.$$
For $\mathcal{C}_1$ and $\mathcal{C}_3$,  $z = r e^{\pm i\theta}$ for $r \in [\delta_0, \infty)$ and $|dz| = dr$, then we have
\begin{align*}
&\left\|t^{\nu-1} \frac{1}{2\pi i} \int_{\mathcal{C}_1 \cup \mathcal{C}_3} R_b(zt^\mu) A(s)(z - A(s))^{-1} dz \right\|_{L(X_0)}\\
\le& C t^{\nu-1} \int_{\delta_0}^\infty \frac{dr}{1 + t^{2\mu}r^2}
\le C t^{\nu-1}t^{-\mu} \int_{\delta_0 t^\mu}^\infty \frac{d\rho}{1 + \rho^2}
\le C t^{\nu-\mu-1}\int_{0}^\infty \frac{d\rho}{1 + \rho^2}
\le C t^{\nu-\mu-1},
\end{align*}
where $\int_{0}^\infty \frac{d\rho}{1 + \rho^2} = \frac{\pi}{2}$.
Hence, we deduce $
\|I_2 \|_{L(X_0)}\leq  C t^{\nu-\mu-1}= C t^{\lambda(1-\mu)-1}.  $
Finally, we deduce \eqref{e3.2}.
For $\lambda>0$, $b-1=\nu-\mu-1\in(-1,0)$ and $\Gamma(b-1)$ is finite.
By the proof of \eqref{e3.2}, we deduce
$$\|A(s)\psi_{\nu-1,A(s)}(t)\|_{L(X_0)}
\le Ct^{\nu-\mu-2}.$$
It follows from the property of Mittag-Leffler function that
$$\frac{d}{d\tau}
\left[
\tau^{\nu-1}
E_{\mu,\nu}(-z\tau^\mu)
\right]
=
\tau^{\nu-2}
E_{\mu,\nu-1}(-z\tau^\mu).$$
Then we have $\frac{d}{d\tau}\psi_{\nu,A(s)}(\tau)=
\psi_{\nu-1,A(s)}(\tau)$. Hence,
\begin{align*}
A(s)\left(\psi_{\nu,A(s)}(t_2) - \psi_{\nu,A(s)}(t_1)\right) = \int_{t_1}^{t_2} \frac{d}{d\tau} \left[ A(s)\psi_{\nu,A(s)}(\tau) \right] d\tau=\int_{t_1}^{t_2}A(s)\psi_{\nu-1,A(s)}(\tau)d\tau.
\end{align*}
Then, we obtain
$$\|A(s)\left(\psi_{\nu,A(s)}(t_2) - \psi_{\nu,A(s)}(t_1)\right)\|_{L(X_0)}  \le C \int_{t_1}^{t_2} \tau^{\nu-\mu-2} d\tau\le C t_1^{\nu-\mu-2}(t_2 - t_1).$$

Next, we prove the case of $\lambda=0$. Then we have $\nu=\mu$ and $\psi_{\nu,A(s)}=\psi_{\mu,A(s)}$. It follows from Lemma \ref{lem3.1} (iii) that $$\|A(s)\psi_{\nu,A(s)}(t)\|_{L(X_0)}
=\|A(s)\psi_{\mu,A(s)}(t)\|_{L(X_0)}
\le Ct^{-1}.$$ Hence, \eqref{e3.2} holds.
By  Lemma \ref{lem3.1} (iv), we get
$$\|A(s)[\psi_{\nu,A(s)}(t_2)-\psi_{\nu,A(s)}(t_1)]\|_{L(X_0)}
\le Ct_1^{\nu-\mu-2}(t_2-t_1).$$
The proof is complete.
\end{proof}

Define operator families
\begin{align}\label{e3.3}
\tilde{\mathcal{Q}}_1(t,s) = -(A(t) - A(s)) \psi_{\nu,A(s)}(t-s),
\end{align}
\begin{align}\label{e3.4}
\tilde{\mathcal{Q}}_2(t,s) = -(A(t) - A(s)) \psi_{\mu,A(s)}(t-s).
\end{align}

For $i \in\{1,2\}$, we consider the operator valued Volterra type integral equation
\begin{align}\label{e3.5}
\mathcal Q_i(t,s)=
\widetilde{\mathcal Q}_i(t,s)
+\int_s^t
\widetilde{\mathcal Q}_2(t,r)
\mathcal Q_i(r,s)\,dr.
\end{align}
We solve the integral equation \eqref{e3.5} for $\mathcal Q_i(t,s)$ by successive approximations. For $m\geq1$, set $\widetilde{\mathcal Q}_{i,1}(t,s)=\widetilde{\mathcal Q}_{i}(t,s)$, we define
$$\widetilde{\mathcal Q}_{i,m+1}(t,s)=\int_s^t \widetilde{\mathcal Q}_2(t,r)\widetilde{\mathcal Q}_{i,m}(r,s)dr,\quad \mathcal Q_i(t,s) =\sum_{m=1}^\infty \widetilde{\mathcal Q}_{i,m}(t,s). $$

\begin{lemma}[\cite{He2022}]\label{lem3.4}
Let Assumption \ref{assump2.1} hold. The operator $\tilde{\mathcal{Q}}_2 (t, s)$ and $\mathcal Q_2 (t, s)$ are continuous in the uniform operator topology for $0\leq s\leq t \leq T$.  $\mathcal Q_2 (t, s)$ is the unique solutions to integral equation \eqref{e3.5}
and
$$
\| \tilde{\mathcal{Q}}_2 (t, s)\|_{L(X_0) }\leq C (t-s)^{\vartheta-1},\quad \| \mathcal Q_2(t, s)\|_{L(X_0) }\leq C (t-s)^{\vartheta-1}.
$$
Moreover, we have
\begin{align*}
&\|\tilde{\mathcal{Q}}_2(t,s) -  \tilde{\mathcal{Q}}_2(\sigma,s)   \|_{L(X_0)}\leq C (t-\sigma)^{\beta}(\sigma-s)^{\vartheta-\beta-1},\\
&\|\mathcal Q_2(t,s) -  \mathcal Q_2(\sigma,s)\|_{L(X_0)}\leq C (t-\sigma)^{\beta}
(\sigma-s)^{\vartheta-\beta-1}
\end{align*}
for all $0\leq s<\sigma\leq t\leq T$ with  $0<\beta<\vartheta \leq1$.
\end{lemma}

\begin{lemma}\label{lem3.5}
Let Assumption \ref{assump2.1} hold. $\mathcal Q_1 (t, s)$ is the unique solution to integral equation \eqref{e3.5}. The following estimates hold for all $0 \le s < t \le T$
\begin{align}\label{e3.6}
\|\widetilde{\mathcal{Q}}_1(t, s)\|_{L(X_0)} \le C(t - s)^{\vartheta + \nu - \mu - 1},
\end{align}
\begin{align}\label{e3.7}
\|\mathcal{Q}_1(t,s)\|_{L(X_0)} \le C(t-s)^{\vartheta + \nu - \mu - 1}.
\end{align}
\end{lemma}
\begin{proof}
According to \eqref{e2.4} and \eqref{e3.2}, we have
\begin{align*} \|\widetilde{\mathcal{Q}}_1(t,s)\| &= \|(A(t) - A(s))\psi_{\nu,A(s)}(t - s)\| \\ &\le \|(A(t) - A(s))A(s)^{-1}\| \, \|A(s)\psi_{\nu,A(s)}(t - s)\| \\ &\le C(t - s)^{\vartheta+\nu-\mu-1} .
\end{align*}
Then we prove \eqref{e3.6}. Let $\gamma_1 := \vartheta + \nu - \mu$. Since $\vartheta > 0$ and $\nu - \mu = \lambda(1-\mu) \ge 0$, we have $\gamma_1 > 0$.
Then there exists a constant $M>0$ such that
$$\|\widetilde{\mathcal{Q}}_1(t,s)\| \le \frac{M}{\Gamma(\gamma_1)}(t - s)^{\gamma_1-1},\quad\|\widetilde{\mathcal{Q}}_2(t,s)\|_{L(X_0)} \le \frac{M}{\Gamma(\vartheta)}(t-s)^{\vartheta - 1}. $$
Since
$$\widetilde{\mathcal{Q}}_{1,1}(t,s) = \widetilde{\mathcal{Q}}_1(t,s),\quad\widetilde{\mathcal{Q}}_{1,m+1}(t,s) = \int_s^t \widetilde{\mathcal{Q}}_2(t,r) \widetilde{\mathcal{Q}}_{1,m}(r,s) dr.$$
By induction, for $m \ge 1$, we have
$$\|\widetilde{\mathcal{Q}}_{1,m}(t,s)\| \le \frac{M^m}{\Gamma((m-1)\vartheta+\gamma_1)}(t - s)^{(m-1)\vartheta+\gamma_1-1}. $$
By $\mathcal{Q}_1(t,s) = \sum_{m=1}^{\infty} \widetilde{\mathcal{Q}}_{1,m}(t,s)$, we deduce
\begin{align*}
\|\mathcal{Q}_1(t,s)\|_{L(X_0)} &\le \sum_{m=1}^{\infty} \frac{M^m (t-s)^{(m-1)\vartheta + \gamma_1 - 1}}{\Gamma((m-1)\vartheta + \gamma_1)}= M(t-s)^{\gamma_1 - 1} \sum_{k=0}^{\infty} \frac{(M(t-s)^{\vartheta})^k}{\Gamma(k\vartheta + \gamma_1)}\\
&= M(t-s)^{\gamma_1 - 1} E_{\vartheta, \gamma_1}(M(t-s)^{\vartheta}).
\end{align*}
Since $0 < t-s \le T$, there is a constant $C > 0$ such that
$$\|\mathcal{Q}_1(t,s)\|_{L(X_0)} \le C (t-s)^{\gamma_1 - 1}.$$
Hence, \eqref{e3.7} holds.
By the iterative definition, for $N \ge 2$,
we obtain
\begin{align*}
S_{1,N}(t,s) &= \sum_{m=1}^N \widetilde{\mathcal{Q}}_{1,m}(t,s)= \widetilde{\mathcal{Q}}_{1,1}(t,s) + \sum_{m=2}^N \widetilde{\mathcal{Q}}_{1,m}(t,s)\\
&= \widetilde{\mathcal{Q}}_1(t,s) + \sum_{m=1}^{N-1} \int_s^t \widetilde{\mathcal{Q}}_2(t,r) \widetilde{\mathcal{Q}}_{1,m}(r,s) dr\\
&= \widetilde{\mathcal{Q}}_1(t,s) + \int_s^t \widetilde{\mathcal{Q}}_2(t,r) S_{1,N-1}(r,s) dr.
\end{align*}
Then
$$\|S_{1,N}(r,s)\|_{L(X_0)} \le \sum_{m=1}^{\infty} \frac{M^m (r-s)^{(m-1)\vartheta + \gamma_1 - 1}}{\Gamma((m-1)\vartheta + \gamma_1)} \le C_T(r-s)^{\gamma_1 - 1}.$$
It follows from $$\int_s^t (t-r)^{\vartheta - 1}(r-s)^{\gamma_1 - 1} dr = B(\vartheta, \gamma_1)(t-s)^{\vartheta+\gamma_1-1} < \infty,$$
that
$$\|\widetilde{\mathcal{Q}}_2(t,r) S_{1,N-1}(r,s)\|_{L(X_0)} \le C_T C(t-r)^{\vartheta - 1}(r-s)^{\gamma_1 - 1},$$
which is independent of $N$. As $N \to \infty$, $S_{1,N}(t,s) \longrightarrow \mathcal{Q}_1(t,s)$ holds in $L(X_0)$. By the Lebesgue dominated convergence theorem, we obtain
$$\lim_{N \to \infty} \int_s^t \widetilde{\mathcal{Q}}_2(t,r) S_{1,N-1}(r,s) dr = \int_s^t \widetilde{\mathcal{Q}}_2(t,r) \mathcal{Q}_1(r,s) dr.$$
Therefore, it yields
$$\mathcal{Q}_1(t,s) = \widetilde{\mathcal{Q}}_1(t,s) + \int_s^t \widetilde{\mathcal{Q}}_2(t,r) \mathcal{Q}_1(r,s) dr.$$
Then we prove that $\mathcal{Q}_1(t,s)$  is a solution to the integral equation \eqref{e3.5}.

Next, we show the uniqueness. Let $\mathcal{Q}_1(t,s)$ and $\widehat{\mathcal{Q}}_1(t,s)$ be two solutions for the integral equation. Then for $C_D>0$, their difference
$D(t,s) = \mathcal{Q}_1(t,s) - \widehat{\mathcal{Q}}_1(t,s)$ satisfies
$$
\|D(t,s)\|_{L(X_0)}\leq C_D(t-s)^{\gamma_1-1}
$$
and the homogeneous Volterra equation
$D(t,s) = \int_s^t \widetilde{\mathcal{Q}}_2(t,r) D(r,s) dr$. Hence, we have
$$\|D(t,s)\|_{L(X_0)} \le \frac{C M^n}{\Gamma(n\vartheta + \gamma_1)} (t-s)^{n\vartheta + \gamma_1 - 1} \to 0 \quad \text{as} \,\, n \to \infty.$$
Then $D(t,s) \equiv 0$. Hence, we have $\mathcal{Q}_1(t,s) \equiv \widehat{\mathcal{Q}}_1(t,s)$. The proof is complete.
\end{proof}

\begin{lemma}\label{lem3.6}
Let Assumption \ref{assump2.1} hold and set
$\gamma_1 = \vartheta + \nu - \mu$. For $$\max\{0, \gamma_1 - 1\} < \beta < \vartheta,$$
there exists a constant $C_\beta > 0$ such that
\begin{align}\label{e3.8}
\|\widetilde{\mathcal{Q}}_1(t,s) - \widetilde{\mathcal{Q}}_1(\sigma,s)\|_{L(X_0)} \le C_\beta (t-\sigma)^\beta (\sigma-s)^{\gamma_1-\beta-1}
\end{align}
\begin{align}\label{e3.9}
\|\mathcal{Q}_1(t,s) - \mathcal{Q}_1(\sigma,s)\|_{L(X_0)} \le C_\beta (t-\sigma)^\beta (\sigma-s)^{\gamma_1-\beta-1}
\end{align}
for all $0 \le s < \sigma < t \le T$. If $\gamma_1 \le 1$, the above estimates hold for $0 < \beta < \vartheta$.
\end{lemma}
\begin{proof}
According to the definition $\widetilde{\mathcal{Q}}_1(t,s)$, we obtain
\begin{align*}
&\widetilde{\mathcal{Q}}_1(t,s) - \widetilde{\mathcal{Q}}_1(\sigma,s) \\
=& -[(A(t) - A(s))\psi_{\nu,A(s)}(t-s) - (A(\sigma) - A(s))\psi_{\nu,A(s)}(\sigma-s)] \\
=&-(A(t) - A(\sigma))\psi_{\nu,A(s)}(t-s)-(A(\sigma) - A(s))[\psi_{\nu,A(s)}(t-s) - \psi_{\nu,A(s)}(\sigma-s)]\\
=:&I_1+I_2
\end{align*}
Combining \eqref{e2.4} and \eqref{e3.2}, we deduce
\begin{align*}
\|I_1\|_{L(X_0)} = \|(A(t) - A(\sigma))A(s)^{-1} A(s)\psi_{\nu,A(s)}(t-s)\|_{L(X_0)}\le C(t-\sigma)^\vartheta (t-s)^{\nu-\mu-1}.
\end{align*}
If $t-\sigma \le \sigma-s$, then
$(t-\sigma)^\vartheta (t-s)^{\nu-\mu-1} \le C (t-\sigma)^\beta (\sigma-s)^{\gamma_1-\beta-1}.$
If $t-\sigma > \sigma-s$, then
$(t-\sigma)^\vartheta (t-s)^{\nu-\mu-1} \le C (t-\sigma)^{\gamma_1-1}.$
Since $\beta > \gamma_1 - 1$, we have $\gamma_1 - \beta - 1 < 0$. Thus, for $t-\sigma > \sigma-s$, we obtain
$(t-\sigma)^{\gamma_1-\beta-1} \le (\sigma-s)^{\gamma_1-\beta-1}$.
It follows that
\begin{align}\label{e3.10}
\|I_1\|_{L(X_0)}\leq C(t-\sigma)^\beta (\sigma-s)^{\gamma_1-\beta-1}.
\end{align}
By Lemma \ref{lem3.3}, we have
$$\|A(s) \left[ \psi_{\nu,A(s)}(t-s) - \psi_{\nu,A(s)}(\sigma-s) \right] \|_{L(X_0)} \le C (t-\sigma) (\sigma-s)^{\nu-\mu-2}.$$
On the other hand, we get
\begin{align*}
&\|A(s) \left[ \psi_{\nu,A(s)}(t-s) - \psi_{\nu,A(s)}(\sigma-s) \right] \|_{L(X_0)}\\
\le& \|A(s)\psi_{\nu,A(s)}(t-s)\|_{L(X_0)} + \|A(s)\psi_{\nu,A(s)}(\sigma-s)\|_{L(X_0)}
\le  C (\sigma-s)^{\nu-\mu-1}.
\end{align*}
Interpolating the two estimates above with parameter $\beta \in (0,1)$, we obtain
\begin{align*}
\nonumber&\|A(s) \left[ \psi_{\nu,A(s)}(t-s) - \psi_{\nu,A(s)}(\sigma-s) \right] \|_{L(X_0)}\\
\nonumber\leq&\left[ C (t-\sigma)(\sigma-s)^{\nu-\mu-2} \right]^\beta \left[ C (\sigma-s)^{\nu-\mu-1} \right]^{1-\beta}
\leq C_\beta (t-\sigma)^\beta (\sigma-s)^{\nu-\mu-\beta-1}.
\end{align*}
Together with \eqref{e2.4}, we deduce
\begin{align}\label{e3.11}
\nonumber\|I_2\|_{L(X_0)}&\leq \|(A(\sigma) - A(s))A(s)^{-1}\|_{L(X_0)} \|A(s) \left[ \psi_{\nu,A(s)}(t-s) - \psi_{\nu,A(s)}(\sigma-s) \right] \|_{L(X_0)} \\
&\leq C (\sigma-s)^\vartheta (t-\sigma)^\beta (\sigma-s)^{\nu-\mu-\beta-1}=C (t-\sigma)^\beta (\sigma-s)^{\gamma_1-\beta-1}.
\end{align}
It follows from \eqref{e3.10} and \eqref{e3.11} that \eqref{e3.8} holds.

From the Volterra equation
$$\mathcal{Q}_1(t,s) = \widetilde{\mathcal{Q}}_1(t,s) + \int_s^t \widetilde{\mathcal{Q}}_2(t,r) \mathcal{Q}_1(r,s) dr,$$
we obtain
\begin{align*}
\mathcal{Q}_1(t,s) -\mathcal{Q}_1(\sigma,s)
&= \left(\widetilde{\mathcal{Q}}_1(t,s) - \widetilde{\mathcal{Q}}_1(\sigma,s)\right)+ \int_\sigma^t \widetilde{\mathcal{Q}}_2(t,r) \mathcal{Q}_1(r,s) dr\\
&\quad+ \int_s^\sigma \left[ \widetilde{\mathcal{Q}}_2(t,r) - \widetilde{\mathcal{Q}}_2(\sigma,r) \right] \mathcal{Q}_1(r,s) dr=: J_1 + J_2 + J_3.
\end{align*}
Since $\beta > \gamma_1 - 1$, it follows from Lemma \ref{lem3.4} and \ref{lem3.5} that
\begin{align*}
\nonumber\|J_2\|_{L(X_0)} \le C \int_\sigma^t (t-r)^{\vartheta-1} (r-s)^{\gamma_1-1} dr
\end{align*}
If $t-\sigma \le \sigma-s$, then $\sigma-s \le r - s \le (\sigma-s) + (t-\sigma) \le 2(\sigma-s)$. Hence,$$\|J_2\|_{L(X_0)} \le C (t-\sigma)^\vartheta (\sigma-s)^{\gamma_1-1}\le C_T (t-\sigma)^\beta (\sigma-s)^{\gamma_1-\beta-1}.$$
If $t-\sigma > \sigma-s$, then
\begin{align*}
\|J_2\|_{L(X_0)} \le C \int_s^t (t - r)^{\vartheta-1}(r - s)^{\gamma_1-1} \, dr 
\le C (t-\sigma)^{\vartheta+\gamma_1-1}.
\end{align*}
Since $t-\sigma \le T$, $\sigma-s < t-\sigma$, and $\gamma_1 - \beta - 1 < 0$,
\begin{align*}
(t-\sigma)^{\vartheta+\gamma_1-1}
&\le T^\vartheta (t-\sigma)^{\gamma_1-1} = T^\vartheta (t-\sigma)^\beta (t-\sigma)^{\gamma_1-\beta-1} \le T^\vartheta (t-\sigma)^\beta (\sigma-s)^{\gamma_1-\beta-1}.
\end{align*}
Thus, we get
\begin{align}\label{e3.12}
\|J_2\|_{L(X_0)} \le C_{\beta,T} (t-\sigma)^\beta (\sigma-s)^{\gamma_1-\beta-1}.
\end{align}
It follows from Lemma \ref{lem3.4} that
\begin{align}\label{e3.13}
\nonumber\|J_3\|_{L(X_0)} &\le C_\beta (t-\sigma)^\beta \int_s^\sigma (\sigma-r)^{\vartheta-\beta-1} (r-s)^{\gamma_1-1} dr\\
&\le C_{\beta, T} (t-\sigma)^\beta (\sigma-s)^{\gamma_1-\beta-1}.
\end{align}
Combining \eqref{e3.8}, \eqref{e3.12} and \eqref{e3.13}, we deduce \eqref{e3.9}. The proof is complete.
\end{proof}

Define the solution operator families
$$H_{\mu, \nu}(t,s) = \psi_{\nu,A(s)}(t-s) + \int_s^t \psi_{\mu,A(r)}(t-r) \mathcal{Q}_1(r, s) dr$$
$$\mathcal{P}_{\mu}(t,s) = \psi_{\mu,A(s)}(t-s) + \int_s^t \psi_{\mu,A(r)}(t-r) \mathcal{Q}_2(r, s) dr.$$

\begin{lemma}\label{lem3.7}
For $0 \le s < t \le T$, the operator families $H_{\mu,\nu}$ and $\mathcal{P}_\mu$ are strongly continuous. There exists $C > 0$ such that
\begin{align}\label{e3.14}
\|H_{\mu,\nu}(t,s)\|_{L(X_0)} \le C(t-s)^{\nu-1},
\end{align}
\begin{align}\label{e3.15}
\|\mathcal{P}_\mu(t,s)\|_{L(X_0)} \le C(t-s)^{\mu-1}.
\end{align}
Moreover, for every $x \in X_0$, we deduce
\begin{align}\label{e3.16}
\lim_{h \to 0} h^{1-\nu} H_{\mu,\nu}(s+h, s)x = \frac{x}{\Gamma(\nu)},
\end{align}
\begin{align}\label{e3.17}
\lim_{t \to s} I_{s+}^{1-\nu} (H_{\mu,\nu}(\cdot, s)x)(t) = x.
\end{align}
\end{lemma}
\begin{proof}
According to the results of \cite{He2022}, we obtain \eqref{e3.15} and $\mathcal P_\mu(t,s)$ is strongly continuous in $X_0$.
Combining Lemma \ref{lem3.1}(ii), \ref{lem3.2} (i) and \eqref{e3.7}, we deduce
\begin{align*}
\|H_{\mu,\nu}(t,s)\|_{L(X_0)}
&\leq \|\psi_{\nu,A(s)}(t-s)\|_{L(X_0)} + \int_s^t \|\psi_{\mu,A(r)}(t-r)\|_{L(X_0)} \|\mathcal{Q}_1(r,s)\|_{L(X_0)} dr\\
&\leq C(t-s)^{\nu-1} + C \int_s^t (t-r)^{\mu-1} (r-s)^{\vartheta+\nu-\mu-1} dr\leq C(t-s)^{\nu-1}.
\end{align*}
For any fixed $s \in [0, T)$, let $s < t_1 < t_2 \le T$. Since the inverse operator $A(s)^{-1} \in L(X_0)$ is uniformly bounded, it follows from Lemma \ref{lem3.3} that
\begin{align*}
&\|\psi_{\nu,A(s)}(t_2 - s) - \psi_{\nu,A(s)}(t_1 - s)\|_{L(X_0)}\\
\le& \|A(s)^{-1}\|_{L(X_0)} \|A(s) \left[ \psi_{\nu,A(s)}(t_2 - s) - \psi_{\nu,A(s)}(t_1 - s) \right] \|_{L(X_0)}\le C(t_1 - s)^{\nu-\mu-2}(t_2 - t_1).
\end{align*}
Therefore, as $t_2 \to t_1$,  we deduce that $$\psi_{\nu,A(s)}(t_2 - s) \to \psi_{\nu,A(s)}(t_1 - s)$$
holds in $L(X_0)$. Then $\psi_{\nu,A(s)}(t - s)$ is strongly continuous.

Denote the integral term by $\mathscr{W}(t,s) := \int_s^t \psi_{\mu,A(r)}(t-r)\mathcal{Q}_1(r,s) dr$.
Then we deduce
\begin{align*}
\mathscr{W}(t_2,s) - \mathscr{W}(t_1,s) \
=& \int_{t_1}^{t_2} \psi_{\mu,A(r)}(t_2-r)\mathcal{Q}_1(r,s) dr \\
&+ \int_s^{t_1} \left[ \psi_{\mu,A(r)}(t_2-r) - \psi_{\mu,A(r)}(t_1-r) \right] \mathcal{Q}_1(r,s) dr=:K_1+K_2.
\end{align*}
It follows from Lemma \ref{lem3.1}(ii) and  \eqref{e3.7} that
\begin{align*}
\|K_1\|_{L(X_0)} \le C \int_{t_1}^{t_2} (t_2-r)^{\mu-1}(r-s)^{\vartheta + \nu - \mu-1} dr\le C \int_{t_1}^{t_2} (t_2-r)^{\mu-1} dr \to 0 \quad\text{as}\,\, t_2 \to t_1.
\end{align*}
For fixed $r\in(s,t_1)$, by the uniform boundedness of $A(r)^{-1}$ and Lemma~\ref{lem3.1}(vi), we obtain
\begin{align*}
&\|\psi_{\mu,A(r)}(t_2-r)
-\psi_{\mu,A(r)}(t_1-r)\|_{L(X_0)}\\
\leq&\|A(r)^{-1}\|_{L(X_0)}
\left\|
A(r)\left[
\psi_{\mu,A(r)}(t_2-r)
-\psi_{\mu,A(r)}(t_1-r)
\right]
\right\|_{L(X_0)}\\
\leq& C\left((t_1-r)^{-1}-(t_2-r)^{-1}\right)
\to 0
\quad\text{as}\,\, t_2\to t_1.
\end{align*}
It follows from Lemma \ref{lem3.1}(ii) and \eqref{e3.7} that
\begin{align*}
\|[\psi_{\mu,A(r)}(t_2-r)-\psi_{\mu,A(r)}(t_1-r)]
\mathcal Q_1(r,s)\|_{L(X_0)}\leq C(t_1-r)^{\mu-1}(r-s)^{\vartheta + \nu - \mu-1},
\end{align*}
and
\begin{align*}
\|K_2\|_{L(X_0)} \le C \int^{t_1}_{s} (t_1-r)^{\mu-1} (r-s)^{\vartheta + \nu - \mu-1} dr.
\end{align*}
Then the function $ (t_1-r)^{\mu-1} (r-s)^{\gamma_1-1}$ is independent of $t_2$ and is integrable on $(s, t_1)$. Therefore, we have
$\|K_2\|_{L(X_0)} \to 0$ as $t_2 \to t_1$ by Lebesgue dominated convergence theorem. Finally, the operator $H_{\mu,\nu}$ is strongly continuous.

For any $x \in X_0$, we have
$$h^{1-\nu} H_{\mu,\nu}(s+h, s)x = h^{1-\nu} \psi_{\nu,A(s)}(h)x + h^{1-\nu} \left[ H_{\mu,\nu}(s+h, s) - \psi_{\nu,A(s)}(h) \right]x.$$
By Lemma \ref{lem3.2}(iii), we get
$\lim_{h \to 0} h^{1-\nu} \psi_{\nu,A(s)}(h)x = \frac{x}{\Gamma(\nu)}$. Combining with
\begin{align}\label{e3.18}
\|\mathscr{W}(t,s) \|_{L(X_0)}=\left\|\int_s^t \psi_{\mu,A(r)}(t-r) \mathcal{Q}_1(r, s) dr\right\|_{L(X_0)} \le C(t-s)^{\nu+\vartheta-1},
\end{align}
we deduce
\begin{align*}
\left\| h^{1-\nu} \mathscr{W}(s+h, s) x \right\|_{X_0} &\le h^{1-\nu} \|\mathscr{W}(s+h, s)\|_{L(X_0)} \|x\|_{X_0}\le h^{1-\nu} \cdot C h^{\nu+\vartheta-1} \|x\|_{X_0} = C h^\vartheta \|x\|_{X_0}.
\end{align*}
It follows that $C h^\vartheta \|x\|_{X_0} \to 0$ as $h \to 0$. Hence, \eqref{e3.16} holds.

For $\nu = 1$, by \eqref{e3.16}, we have
\begin{align*}
\lim_{t \to s} I_{s+}^{1-\nu}(H_{\mu,\nu}(\cdot, s)x)(t) = \lim_{t \to s} H_{\mu,1}(t, s)x = x.
\end{align*}
For $0<\nu < 1$, note that
$$I_{s+}^{1-\nu} (H_{\mu,\nu}(\cdot, s)x)(t) = I_{s+}^{1-\nu} (\psi_{\nu,A(s)}(\cdot-s)x)(t) + I_{s+}^{1-\nu} (\mathscr{W}(\cdot, s)x)(t).$$
By \eqref{e3.18}, we have
\begin{align*}
\left\| I_{s+}^{1-\nu} (\mathscr{W}(\cdot, s)x)(t) \right\|_{X_0}
&\le \frac{1}{\Gamma(1-\nu)} \int_s^t (t-r)^{-\nu} \|\mathscr{W}(r, s)x\|_{X_0} dr\\
&\le \frac{C\|x\|_{X_0}}{\Gamma(1-\nu)} \int_s^t (t-r)^{-\nu} (r-s)^{\nu+\vartheta-1} dr\le C (t-s)^\vartheta \|x\|_{X_0} \to 0 \quad\text{as} \,\, t\to s.
\end{align*}
Combining with Lemma \ref{lem3.2}(iii), we have \eqref{e3.17}. This completes the proof.
\end{proof}
\begin{lemma}\label{lem3.8}
Let Assumption \ref{assump2.1} hold. For $a \in [0, T]$ and $\tau > 0$, define
$$K_a(\tau) := A(a)\psi_{\mu,A(a)}(\tau).$$
Then, for $a, c \in [0, T]$,
$$\left\| \frac{d}{d\tau} (K_a(\tau) - K_c(\tau)) \right\|_{L(X_0)} \le C|a - c|^\vartheta \tau^{-2}.$$
Moreover, for every $\eta \in (0, 1)$ and $\tau, h > 0$ with $\tau + h \le T$, we get
\begin{align*}
\|[K_a(\tau + h) - K_c(\tau + h)] - [K_a(\tau) - K_c(\tau)]\|_{L(X_0)} \le C_\eta |a - c|^\vartheta h^\eta \tau^{-1-\eta}.
\end{align*}
\end{lemma}
\begin{proof}
It follows from definition of $\psi_{\mu,A(a)}$ that
$$K_a(\tau) = \frac{1}{2\pi i} \int_{\mathcal{C}} \tau^{\mu-1} E_{\mu,\mu}(-z\tau^\mu) A(a) R(z; A(a)) \, dz,$$
where $R(z; A(a)) = (z - A(a))^{-1}$.
Since $A(a) R(z; A(a)) = z R(z; A(a)) - I$, we get
$$K_a(\tau) - K_c(\tau) = \frac{1}{2\pi i} \int_{\mathcal{C}} z\tau^{\mu-1} E_{\mu,\mu}(-z\tau^\mu) \Delta R_{a,c}(z) \, dz,$$
where $\Delta R_{a,c}(z) = R(z; A(a)) - R(z; A(c))$. Then
\begin{align*}
\Delta R_{a,c}(z) &= R(z; A(a))[A(a) - A(c)]R(z; A(c)) = R(z; A(a))[A(a)A(c)^{-1} - I]A(c)R(z; A(c)).
\end{align*}
For $z \in \mathcal{C}$, it follows from \eqref{e2.1} and \eqref{e2.4} that
\begin{align*}
\|\Delta R_{a,c}(z)\|_{L(X_0)} &\le \|R(z; A(a))\|_{L(X_0)} \cdot \|A(a)A(c)^{-1} - I\|_{L(X_0)} \cdot \|A(c) R(z; A(c))\|_{L(X_0)}\\
&\le  \frac{C}{|z|}  \cdot (C|a - c|^\vartheta) \cdot C\le C|z|^{-1} |a - c|^\vartheta.
\end{align*}
For fixed $z$, set $g(\tau, z) = z\tau^{\mu-1} E_{\mu,\mu}(-z\tau^\mu)$,
then
$$\frac{\partial}{\partial \tau} g(\tau, z) = (\mu-1) z \tau^{\mu-2} E_{\mu,\mu}(-z\tau^\mu) - \mu z^2 \tau^{2\mu-2} E_{\mu,\mu}'(-z\tau^\mu),$$
where $E_{\mu,\mu}'$ denotes the derivative with respect to the complex argument.
By the asymptotic expansion of the Mittag-Leffler function in the sector containing $-\mathcal{C}$, we get $$E_{\mu,\mu}(-\xi) = O(|\xi|^{-2}), \quad E_{\mu,\mu}'(-\xi) = O(|\xi|^{-3})$$
as $|\xi| \to \infty$. Near the origin, both functions are analytic and bounded. Then $$G(\xi) = \frac{|(\mu - 1)\xi E_{\mu,\mu}(-\xi) - \mu\xi^2 E_{\mu,\mu}'(-\xi)|}{|\xi|}$$is bounded near $\xi = 0$ and satisfies $G(\xi) \le C|\xi|^{-2}$for large $|\xi|$. Then
$\sup_{\tau>0}\int_{\tau^\mu\mathcal C}G(\xi)\,|d\xi|<\infty$.
Hence, $G$ is integrable along every scaled contour $\tau^\mu \mathcal{C}$ uniformly for $\tau > 0$.
Note that
\begin{align*}
\frac{d}{d\tau}(K_a(\tau) - K_c(\tau))
=& \frac{1}{2\pi i} \int_{\mathcal{C}} \frac{\partial g}{\partial \tau}(\tau, z) \Delta R_{a,c}(z)dz \\
=&\frac{1}{2\pi i} \int_{\tau^\mu \mathcal{C}} \tau^{-2} \left[ (\mu-1) \xi E_{\mu,\mu}(-\xi) - \mu \xi^2 E_{\mu,\mu}'(-\xi) \right] \Delta R_{a,c}(\xi \tau^{-\mu}) \tau^{-\mu} d\xi.
\end{align*}
Then we have
\begin{align*}
\left\| \frac{d}{d\tau} (K_a(\tau) - K_c(\tau)) \right\|_{L(X_0)} \le& \frac{1}{2\pi} \int_{\tau^\mu \mathcal{C}} \tau^{-2} \left| (\mu-1) \xi E_{\mu,\mu}(-\xi) - \mu \xi^2 E_{\mu,\mu}'(-\xi) \right| \cdot C_1 |a - c|^\vartheta |\xi|^{-1} |d\xi|\\
=& \frac{C}{2\pi} |a - c|^\vartheta \tau^{-2} \int_{\tau^\mu \mathcal{C}} \frac{\left| (\mu-1) \xi E_{\mu,\mu}(-\xi) - \mu \xi^2 E_{\mu,\mu}'(-\xi) \right|}{|\xi|} |d\xi|\\
\le& |a - c|^\vartheta \tau^{-2}.
\end{align*}
It follows from Lemma \ref{lem3.2}(ii) that $\|K_a(\tau) - K_c(\tau)\|_{L(X_0)} \le C|a - c|^\vartheta \tau^{-1}$.
Then we have
\begin{align*}
&\|[K_a(\tau + h) - K_c(\tau + h)] - [K_a(\tau) - K_c(\tau)]\|_{L(X_0)} \\
\le& \int_\tau^{\tau+h} \left\| \frac{d}{dr}(K_a(r) - K_c(r)) \right\|_{L(X_0)} dr \le C|a - c|^\vartheta h \tau^{-2}. \end{align*}
Therefore, we deduce
\begin{align*}
\|[K_a(\tau + h) - K_c(\tau + h)] - [K_a(\tau) - K_c(\tau)]\|_{L(X_0)} \le C|a - c|^\vartheta \min\{\tau^{-1}, h\tau^{-2}\}.
\end{align*}
For every $\eta \in (0, 1)$,
$$\min\{\tau^{-1}, h\tau^{-2}\} \le (\tau^{-1})^{1-\eta}(h\tau^{-2})^\eta = h^\eta \tau^{-1-\eta}.$$
Then we obtain the desired result. This completes the proof.
\end{proof}
\begin{lemma}\label{lem3.9}
Let Assumption \ref{assump2.1} hold and suppose $\vartheta>b=\nu-\mu$. Fix $s\in[0,T)$. Let $\phi:(s,T]\to X_0$ be strongly measurable. Assume that for $s<r\leq T$, there is $\delta > b$ and $\alpha \in (b, 1]$ such that
$$\|\phi(r,s)\|_{X_0}\leq C(r-s)^{\delta-1},$$
and for $\varepsilon > 0$, $t, r \in [s + \varepsilon, T]$,
$$\|\phi(t, s) - \phi(r, s)\|_{X_0} \le C_\varepsilon|t - r|^\alpha.$$
Assume
$$h_s(t):=\int_s^t
\frac{\|\phi(t,s)-\phi(r,s)\|_{X_0}}{(t-r)^{b+1}}\,dr$$
is finite for $t>s$ and  $h_s\in L^1(s,T)$.
Then the integral
$$J_\phi(t, s) := \lim_{\rho \to 0} \int_s^{t-\rho} A(r)\psi_{\mu,A(r)}(t - r)\phi(r, s) \, dr$$
exists in $X_0$ for $t > s$.
Define $q_\phi(t, s) := \phi(t, s) - J_\phi(t, s)$.
Then, for $\varepsilon > 0$ and $b < \eta < \min\{\alpha, \vartheta\}$,
we have
$$q_\phi(\cdot, s) \in C^\eta([s + \varepsilon, T]; X_0).$$
\end{lemma}
\begin{proof}
Let $K_r(\tau) := A(r)\psi_{\mu,A(r)}(\tau)$. We first prove that $J_\phi(t, s)$ is well defined. For $\rho > 0$,
$$J_{\phi,\rho}(t, s) := \int_s^{t-\rho} K_r(t - r)\phi(r, s) \, dr.$$
Then we obtain
\begin{align*}
J_{\phi,\rho}(t, s) &= \int_s^{t-\rho} K_r(t - r)[\phi(r, s) - \phi(t, s)] \, dr \\ &\quad + \int_s^{t-\rho} [K_r(t - r) - K_t(t - r)]\phi(t, s) \, dr  + \int_s^{t-\rho} K_t(t - r)\phi(t, s)  dr.
\end{align*}
Since $b \ge 0$, we deduce
\begin{align*}
\int_s^t \|K_r(t - r)\| \|\phi(r, s) - \phi(t, s)\|_{X_0} \, dr
\le C \int_s^t \frac{\|\phi(r, s) - \phi(t, s)\|_{X_0}}{t - r} \, dr \le C T^b h_s(t) < \infty.
\end{align*}
By Lemma \ref{lem3.2}(ii), we obtain
$\|K_r(t - r) - K_t(t - r)\|_{L(X_0)} \le  C(t - r)^{\vartheta-1}$.
Hence,
$$\int_s^t \|K_r(t - r) - K_t(t - r)\| \|\phi(t, s)\|_{X_0} \, dr < \infty.$$
Then
$$\int_s^{t-\rho} K_t(t - r) \, dr = \int_\rho^{t-s} K_t(\tau) \, d\tau = S_{\mu,A(t)}(\rho) - S_{\mu,A(t)}(t - s),$$
and
$$\lim_{\rho \to 0} \int_s^{t-\rho} K_t(t - r) \, dr = I - S_{\mu,A(t)}(t - s)$$
in the strong operator topology. It follows that $J_\phi(t, s)$ exists.

We next prove the local H\"{o}lder regularity. Fix $\varepsilon > 0$, let $s + \varepsilon \le \sigma < t \le T$, $ h := t - \sigma$,
and suppose first that $0 < h \le 1$. We have
\begin{align*}
q_\phi(t, s) - q_\phi(\sigma, s)
= &\phi(t, s) - \phi(\sigma, s) - \int_\sigma^t K_r(t - r)\phi(r, s) dr-\int_s^\sigma [K_r(t - r) - K_r(\sigma - r)]\phi(r, s) dr\\
=:&\phi(t, s) - \phi(\sigma, s) - \Delta_1 - \Delta_2.
\end{align*}
By Lemma \ref{lem3.1} (viii), we have
$\int_\sigma^t K_t(t - r) dr = \int_0^h K_t(\tau) d\tau=I - S_{\mu,A(t)}(h)$.
Similarly, we deduce
$$\int_s^\sigma K_\sigma(t - r) dr = \int_{h}^{t - s} K_\sigma(\tau) d\tau = S_{\mu,A(\sigma)}(h) - S_{\mu,A(\sigma)}(t - s)$$
$$\int_s^\sigma K_\sigma(\sigma - r) dr = \int_{0}^{\sigma - s} K_\sigma(\tau) d\tau =I-S_{\mu,A(\sigma)}(\sigma - s).$$
For $\Delta_1$, we obtain
\begin{align*}
\Delta_1= &\int_\sigma^t K_r(t - r)[\phi(r, s) - \phi(t, s)]dr + \int_\sigma^t [K_r(t - r) - K_t(t - r)]\phi(t, s) dr\\
&+\int_\sigma^t K_t(t - r)\phi(t,s) dr
=:I_{11}+I_{12}+\left(I - S_{\mu,A(t)}(h)\right)\phi(t,s).
\end{align*}

For $\Delta_2$, we obtain
\begin{align*}
\Delta_2=&\int_s^\sigma [K_r(t - r) - K_r(\sigma - r)][\phi(r, s) - \phi(\sigma, s)] dr\\
&+ \int_s^\sigma \left( [K_r(t - r) - K_\sigma(t - r)] - [K_r(\sigma - r) - K_\sigma(\sigma - r)] \right) \phi(\sigma, s) dr\\
&+\left( \int_s^\sigma K_\sigma(t - r) dr - \int_s^\sigma K_\sigma(\sigma - r) dr \right) \phi(\sigma, s)\\
=:&I_{21}+I_{22}+[S_{\mu,A(\sigma)}(h) - S_{\mu,A(\sigma)}(t - s) - I + S_{\mu,A(\sigma)}(\sigma - s)]\phi(\sigma, s).
\end{align*}
Hence, it follows that
\begin{align*}
q_\phi(t, s) - q_\phi(\sigma, s) = - I_{11} - I_{12} - I_{21} - I_{22} + M_1 + M_2,
\end{align*}
where $M_1 := S_{\mu,A(t)}(h)\phi(t, s) - S_{\mu,A(\sigma)}(h)\phi(\sigma, s)$ and $M_2 := [S_{\mu,A(\sigma)}(t - s) - S_{\mu,A(\sigma)}(\sigma - s)]\phi(\sigma, s)$.
By Lemma \ref{lem3.1} (iii) and the H\"{o}lder continuity of $\phi(\cdot, s)$ on $[s + \varepsilon, T]$, we deduce
\begin{align*}
\|I_{11}\|_{X_0} \le C \int_\sigma^t (t - r)^{-1} \|\phi(r, s) - \phi(t, s)\|_{X_0} dr\le C_\varepsilon \int_\sigma^t (t - r)^{\alpha - 1} dr \le C_\varepsilon h^\alpha \le C_{\varepsilon, \eta} h^\eta.
\end{align*}
From Lemma \ref{lem3.2} (ii), we obtain
$$\|I_{12}\|_{X_0} \le C\|\phi(t, s)\|_{X_0} \int_\sigma^t (t - r)^{\vartheta - 1} dr \le C_\varepsilon h^\vartheta \le C_{\varepsilon, \eta} h^\eta.$$
Set $a := s + \frac{\varepsilon}{2}$. Since $\sigma\ge s+\varepsilon$, we have $s<a<\sigma$ and
$\sigma-r\ge\varepsilon/2$ for $r\in[s,a]$.
\begin{align*}
I_{21}=& \int_s^a [K_r(t - r) - K_r(\sigma - r)][\phi(r, s) - \phi(\sigma, s)] dr\\
&+ \int_a^\sigma [K_r(t - r) - K_r(\sigma - r)][\phi(r, s) - \phi(\sigma, s)] dr
=:J_1+J_2.
\end{align*}
For $r \in (s, a)$, we have $\sigma - r \ge \frac{\varepsilon}{2}$. Combining with $\|\phi(r,s)\|_{X_0}
\leq C(r-s)^{\delta-1}$, we obtain
$$\|J_1\|_{X_0} \le C_\varepsilon h \int_s^a (1 + (r - s)^{\delta - 1}) dr \le C_{\varepsilon,T} h \le C_{\varepsilon,T, \eta} h^\eta.$$
It follows from Lemma \ref{lem3.1}(vi) that
$\|K_r(\tau + h) - K_r(\tau)\|_{L(X_0)} \le C_\eta h^\eta \tau^{-1-\eta}$
for $0 < \eta < 1$ and $h, \tau > 0$. Then for $\eta < \alpha$, we have
\begin{align*}
\|J_2\|_{X_0} \le C_\eta h^\eta \int_a^\sigma (\sigma - r)^{-1 - \eta} \|\phi(r, s)- \phi(\sigma, s)\|_{X_0} dr\le C_{\varepsilon, \eta} h^\eta \int_a^\sigma (\sigma - r)^{\alpha - \eta - 1} dr \le C_{\varepsilon,T, \eta} h^\eta,
\end{align*}
Therefore, we have $\|I_{21}\|_{X_0} \le C_{\varepsilon, T,\eta} h^\eta$.

For $I_{22}$, it follows from Lemma \ref{lem3.8} that
\begin{align*}
&\|[K_r(t - r) - K_\sigma(t - r)] - [K_r(\sigma - r) - K_\sigma(\sigma - r)]\|_{L(X_0)} \\
 \le& C_\eta h^\eta|\sigma - r|^\vartheta(\sigma - r)^{-1-\eta} = C_\eta h^\eta(\sigma - r)^{\vartheta-\eta-1}.
\end{align*}
Since $\eta < \vartheta$, we obtain
$$\|I_{22}\|_{X_0} \le C_\varepsilon h^\eta \int_s^\sigma (\sigma - r)^{\vartheta - \eta - 1} dr = \frac{C_\varepsilon}{\vartheta - \eta} h^\eta (\sigma - s)^{\vartheta - \eta} \le C_{\varepsilon, T, \eta} h^\eta.$$
Note that
\begin{align*}
M_1 :=& S_{\mu,A(t)}(h)\phi(t, s) - S_{\mu,A(\sigma)}(h)\phi(\sigma, s)\\
=& S_{\mu,A(t)}(h)[\phi(t, s) - \phi(\sigma, s)] + [S_{\mu,A(t)}(h) - S_{\mu,A(\sigma)}(h)]\phi(\sigma, s).
\end{align*}
Combining Lemma \ref{lem3.1} (vii) and (ix), we get
$$\|M_1\|_{X_0} \le C\|\phi(t, s) - \phi(\sigma, s)\|_{X_0} + C|t - \sigma|^\vartheta \|\phi(\sigma, s)\|_{X_0} \le C_\varepsilon(h^\alpha + h^\vartheta) \le C_{\varepsilon, \eta} h^\eta.$$
Since $\sigma - s \ge \varepsilon$ and  Lemma \ref{lem3.1} (viii), we have
$$\|M_2\|_{X_0} \le \|\phi(\sigma, s)\|_{X_0} \int_{\sigma - s}^{t - s} \|K_\sigma(\tau)\|_{L(X_0)} d\tau \le C_\varepsilon \int_{\sigma - s}^{t - s} \tau^{-1} d\tau \le C_\varepsilon h \le C_{\varepsilon, \eta} h^\eta.$$
Then we conclude that
$$\|q_\phi(t, s) - q_\phi(\sigma, s)\|_{X_0} \le C_{\varepsilon, \eta} |t - \sigma|^\eta.$$
The case $h > 1$ follows from the boundedness of $q_\phi$ on $[s + \varepsilon, T]$. Hence,
$$q_\phi(\cdot, s) \in C^\eta([s + \varepsilon, T]; X_0).$$
This completes the proof.
\end{proof}

\begin{lemma}\label{lem3.10}
Under the assumptions of Lemma \ref{lem3.9}, define
$$w_\phi(t,s):=\int_s^t
\psi_{\mu,A(r)}(t-r)\phi(r,s)\,dr,
\quad t\in(s,T].$$
Then $w_\phi(t,s)$ is well defined and $\lim_{t\to s}
I_{s+}^{1-\nu}w_\phi(t,s)=0$. Moreover, for $t>s$, we have
$$I_{s+}^{1-\nu}w_\phi(\cdot,s)
\in W_{\mathrm{loc}}^{1,1}((s,T);X_0)$$
and $D_{s+}^{\lambda,\mu} w_\phi(t, s) = q_\phi(t, s)$, where $q_\phi$ is the function defined in Lemma \ref{lem3.9}, that is,
\begin{align}\label{e3.19}
D_{s+}^{\lambda,\mu}w_\phi(t,s)
=\phi(t,s)-\int_s^t
A(r)\psi_{\mu,A(r)}(t-r)\phi(r,s)\,dr.
\end{align}
\end{lemma}
\begin{proof}
According to Lemma \ref{lem3.1} (ii), we have
\begin{align}\label{e3.20}
\|w_\phi(t,s)\|_{X_0}
\leq
C\int_s^t
(t-r)^{\mu-1}(r-s)^{\delta-1}\,dr \leq C(t-s)^{\mu+\delta-1}.
\end{align}
Since $\delta>0$, the integral is finite. This proves that $w_\phi(t,s)$ is well defined.

We first consider $0 < \nu < 1$. It follows from \eqref{e3.20} that
\begin{align*}
\left\|
I_{s+}^{1-\nu}w_\phi(t,s)
\right\|_{X_0}
\leq
\frac{C}{\Gamma(1-\nu)}
\int_s^t
(t-\tau)^{-\nu}
(\tau-s)^{\mu+\delta-1}\,d\tau  \leq C(t-s)^{\mu+\delta-\nu}.
\end{align*}
Since $\mu+\delta-\nu>0$, we obtain
$$\lim_{t\to s}
I_{s+}^{1-\nu}w_\phi(t,s)=0.$$
If $\nu = 1$, then $I_{s+}^{1-\nu} = I$. Since $\mu + \delta - 1 = \delta - (1 - \mu) > 0$, combining with \eqref{e3.20}, we have
$$\|w_\phi(t, s)\|_{X_0} \le C(t - s)^{\mu+\delta-1} \to 0 \quad \text{as}\,\, t\to s.$$
Since $I_{0+}^{1-\mu}
\psi_{\mu,A(r)}(\tau)=S_{\mu,A(r)}(\tau)$.
It follows from Fubini's theorem that
\begin{align*}
\nonumber I_{s+}^{1-\mu} w_\phi(t,s)
\nonumber&= \frac{1}{\Gamma(1-\mu)} \int_s^t (t-\tau)^{-\mu} w_\phi(\tau,s) \, d\tau \\
\nonumber&= \int_s^t \left( \frac{1}{\Gamma(1-\mu)} \int_r^t (t-\tau)^{-\mu} \psi_{\mu,A(r)}(\tau-r) \, d\tau \right) \phi(r,s) \, dr \\
&= \int_s^t S_{\mu,A(r)}(t-r)\phi(r,s) \, dr.
\end{align*}
Moreover, we have
$\|I_{s+}^{1-\mu}w_\phi(t, s)\|_{X_0} \le CB(1 - \mu, \mu + \delta)(t - s)^\delta$ and $\lim_{t \to s} I_{s+}^{1-\mu}w_\phi(t, s) = 0$.
Since
$$\int_s^t \frac{\|\phi(t, s) - \phi(r, s)\|_{X_0}}{t - r} \, dr \le T^b h_s(t)$$
and $h_s \in L^1(s, T)$,
according to \cite[Lemma 4]{He2022}, we get
\begin{align}\label{e3.21}
{}^{RL}D_{s+}^{\mu}w_\phi(t,s)&=\frac{d}{dt}I_{s+}^{1-\mu}w_\phi(t,s)=\phi(t,s)-\int_s^t
A(r)\psi_{\mu,A(r)}(t-r)\phi(r,s)\,dr=q_\phi(t, s).
\end{align}
By Lemma \ref{lem3.9}, we have
\begin{align*}
\|q_\phi(t, s)\|_{X_0} &\le \|\phi(t, s)\|_{X_0} + \|J_\phi(t, s)\|_{X_0} \le C \left[ h_s(t) + (1 + (t - s)^\vartheta) \|\phi(t, s)\|_{X_0} \right].
\end{align*}
Since $h_s \in L^1(s, T)$, $\delta > 0$, and $\|\phi(t, s)\|_{X_0} \le C(t - s)^{\delta-1}$, we obtain $q_\phi(\cdot, s) \in L^1(s, T; X_0)$.
Since $\lim_{t \to s} I_{s+}^{1-\mu} w_\phi(t, s) = 0$, we deduce
$w_\phi(\cdot, s) = I_{s+}^\mu q_\phi(\cdot, s)$.
Then we have
$$I_{s+}^{1-\nu} w_\phi = I_{s+}^{1-\nu} I_{s+}^\mu q_\phi = I_{s+}^{1-b} q_\phi.$$
Next, we prove $ I_{s+}^{1-b} q_\phi \in W_{\mathrm{loc}}^{1,1}(s, T; X_0)$. If $b = 0$, then $I_{s+}^{1-b} q_\phi = I_{s+}^1 q_\phi \in W^{1,1}(s, T; X_0)$.
For $0 < b < 1$, fix $\varepsilon > 0$, according to Lemma \ref{lem3.9}, for $b < \eta < \min\{\alpha, \vartheta\}$, we have
$$q_\phi(\cdot, s) \in C^\eta([s + \varepsilon, T]; X_0).$$
Fix $\varepsilon>0$ and let $t\in[s+2\varepsilon,T]$. Since $q_\phi(\cdot,s)\in L^1(s,T;X_0)$ and $q_\phi(\cdot,s)\in C^\eta([s+\varepsilon,T];X_0)$ for some $\eta>b$, it follows from the Marchaud representation of the left-sided Riemann–Liouville derivative that
\begin{align*}
\frac{d}{dt} I_{s+}^{1-b} q_\phi(t, s) = \frac{q_\phi(t, s)}{\Gamma(1 - b)(t - s)^b}  + \frac{b}{\Gamma(1 - b)} \int_s^t \frac{q_\phi(t, s) - q_\phi(r, s)}{(t - r)^{b+1}} \, dr.
\end{align*}
Let $q(t) := q_\phi(t, s)$. For $t \in [s + 2\varepsilon, T]$, since $(t - s)^{-b} \le (2\varepsilon)^{-b}$ and  $q_\phi(t, s)$ is continuous and bounded, we have $\left\| \frac{q(t)}{\Gamma(1-b)(t - s)^b} \right\|_{X_0} \le C_\varepsilon$.
Therefore, this term belongs to $L^1$ on any interval separated from $s$. Let
$$\int_s^t \frac{\|q(t) - q(r)\|_{X_0}}{(t - r)^{b+1}} dr = \int_s^{s+\varepsilon} \frac{\|q(t) - q(r)\|_{X_0}}{(t - r)^{b+1}} dr + \int_{s+\varepsilon}^t \frac{\|q(t) - q(r)\|_{X_0}}{(t - r)^{b+1}} dr =: J_{1}(t) + J_{2}(t).$$

For the part away from $r=t$, when $r \in (s, s+\varepsilon)$ and $t \ge s+2\varepsilon$, we have $t - r \ge \varepsilon$. Thus,
$$J_{1}(t) \le \varepsilon^{-b-1} \int_s^{s+\varepsilon} (\|q(t)\|_{X_0} + \|q(r)\|_{X_0}) dr \le \varepsilon^{-b} \|q(t)\|_{X_0} + \varepsilon^{-b-1} \int_s^{s+\varepsilon} \|q(r)\|_{X_0} dr.$$
Since $q \in L^1(s, T; X_0)$ and $q(t)$ is H\"{o}lder continuous on $[s + 2\varepsilon, T]$, $J_{1}(t)$ is bounded.

For the part near $r=t$, when $r \in [s+\varepsilon, t]$, both $r$ and $t$ are bounded away from the initial point $s$. From Lemma \ref{lem3.9}, we obtain
$$J_{2}(t) \le C_\varepsilon \int_{s+\varepsilon}^t \frac{(t - r)^\eta}{(t - r)^{b+1}} dr = C_\varepsilon \int_{s+\varepsilon}^t (t - r)^{\eta - b - 1} dr = \frac{C_\varepsilon}{\eta - b} (t - s - \varepsilon)^{\eta - b}.$$
Since $\eta > b$, $J_{2}(t)$ is bounded.
Combining the above estimates, $\frac{d}{dt} I_{s+}^{1-b} q(t)$ exists and belongs to $L^1$ on $[s + 2\varepsilon, T]$. Since $\varepsilon > 0$ is arbitrary, we have
$$I_{s+}^{1-b} q_\phi(\cdot, s) \in W_{\text{loc}}^{1,1}((s, T]; X_0).$$
Finally, since
$\lim_{t \to s} I_{s+}^{1-b} q_\phi(t, s) = \lim_{t \to s} I_{s+}^{1-\nu} w_\phi(t, s) = 0$. For $b > 0$, we get
$$I_{s+}^b \frac{d}{dt} I_{s+}^{1-b} q_\phi(t, s) = q_\phi(t, s).$$
Hence, we obtin
\begin{align*}
D_{s+}^{\lambda,\mu} w_\phi(t, s) = I_{s+}^b \frac{d}{dt} I_{s+}^{1-\nu} w_\phi(t, s) = I_{s+}^b \frac{d}{dt} I_{s+}^{1-b} q_\phi(t, s) = q_\phi(t, s).
\end{align*}
For $b = 0$, we have the same conclusion. It follows from \eqref{e3.21} that
\begin{align*}
D_{s+}^{\lambda,\mu} w_\phi(t, s)= \phi(t, s) - \int_s^t A(r)\psi_{\mu,A(r)}(t - r)\phi(r, s) \, dr.
\end{align*}
This completes the proof.
\end{proof}

\begin{lemma}\label{lem3.11}
Let Assumption \ref{assump2.1} hold, let $b := \nu - \mu = \lambda(1 - \mu)$, and assume that $b < \min\{\rho, \vartheta\}$, $f \in C^\rho([0, T]; X_0)$, $u_0 \in X_0$ with $\rho \in (0, 1]$. Define
$$\phi_1(r) := \mathcal{Q}_1(r, 0)u_0, \quad \phi_2(r) := f(r), \quad \phi_3(r) := g(r) := \int_0^r \mathcal{Q}_2(r, \xi)f(\xi) \, d\xi.$$
Then each function $\phi_i(i = 1, 2, 3)$ satisfies the hypotheses of Lemma \ref{lem3.9} with $s = 0$. In particular, for each $i$ and $0 < r \le T$, there exist $\delta_i > b$ and $\alpha_i > b$ such that$$\|\phi_i(r)\|_{X_0} \le Cr^{\delta_i-1} ,$$
$\phi_i$ is locally $\alpha_i$-H\"{o}lder continuous on $(0, T]$ and
$h_i(t) := \int_0^t \frac{\|\phi_i(t) - \phi_i(r)\|_{X_0}}{(t - r)^{b+1}} \, dr$ is finite for every $t > 0$ and belongs to $L^1(0, T)$.
\end{lemma}
\begin{proof}
Set $\gamma_1 := \vartheta + \nu - \mu = \vartheta + b$.

\textbf{Step 1: } Consider the function $\phi_1(r) = \mathcal{Q}_1(r, 0)u_0$. By Lemma \ref{lem3.5},
$$\|\mathcal{Q}_1(r, 0)\|_{L(X_0)} \le Cr^{\gamma_1-1}.$$
Then we get
$\|\phi_1(r)\|_{X_0} \le Cr^{\gamma_1-1}\|u_0\|_{X_0}.$
For $\delta_1 = \gamma_1 = \vartheta + b > b$, then $b < \vartheta$ and $\gamma_1 - 1 = \vartheta + b - 1 < \vartheta$. Choose $\beta$ such that $\max\{b, \gamma_1 - 1, 0\} < \beta < \vartheta$. By Lemma \ref{lem3.6}, for $0 < r < t \le T$,
\begin{align*}
\|\phi_1(t) - \phi_1(r)\|_{X_0} &\le \|\mathcal{Q}_1(t, 0) - \mathcal{Q}_1(r, 0)\|_{L(X_0)}\|u_0\|_{X_0} \le C_\beta(t - r)^\beta r^{\gamma_1-\beta-1}\|u_0\|_{X_0}.
\end{align*}
For every $\varepsilon > 0$, $r^{\gamma_1-\beta-1}$ is bounded on $[\varepsilon, T]$. Hence, $\phi_1 \in C^\beta([\varepsilon, T]; X_0)$.
Moreover,
\begin{align*}
h_1(t) &\le C_\beta\|u_0\|_{X_0} \int_0^t (t - r)^{\beta-b-1} r^{\gamma_1-\beta-1} \, dr \\ &= C_\beta B(\beta - b, \gamma_1 - \beta) t^{\gamma_1-b-1} \|u_0\|_{X_0}\le C_\beta t^{\vartheta-1} \|u_0\|_{X_0}.
\end{align*}
Since $\vartheta > 0$, we deduce $h_1 \in L^1(0, T)$. Thus, $\phi_1$ satisfies all the hypotheses of Lemma \ref{lem3.9}.

\textbf{Step 2:} Consider the function $\phi_2 = f$.

Since $f \in C^\rho([0, T]; X_0)$, we have $\|\phi_2(r)\|_{X_0} = \|f(r)\|_{X_0} \le \|f\|_{C([0, T]; X_0)}$ and $\delta_2 = 1 > b$. Combining with
$$\|f(t) - f(r)\|_{X_0} \le [f]_{C^\rho([0, T]; X_0)} |t - r|^\rho,$$
we deduce
$$h_2(t) \le [f]_{C^\rho} \int_0^t (t - r)^{\rho - b - 1} dr = \frac{[f]_{C^\rho}}{\rho - b} t^{\rho - b}.$$
Since $\rho > b$, $h_2 \in L^1(0, T)$.
Then $\phi_2$ satisfies the hypotheses of Lemma \ref{lem3.9} with $\alpha_2 = \rho$ and $\delta_2 = 1$.

\textbf{Step 3:} Consider the function $\phi_3(r)$.

By Lemma \ref{lem3.4}, we deduce
$\|\mathcal{Q}_2(r, \xi)\|_{L(X_0)} \le C(r - \xi)^{\vartheta - 1}$.
Hence, we obtain
$$\|g(r)\|_{X_0} \le C\|f\|_{C([0, T]; X_0)} \int_0^r (r - \xi)^{\vartheta - 1} d\xi \le C r^\vartheta \|f\|_{C([0, T]; X_0)}$$
and $\delta_3 = 1 + \vartheta > b$.
Choose any $b < \beta < \vartheta$, for $0 < r < t \le T$, let
$$g(t) - g(r) = \int_r^t \mathcal{Q}_2(t, \xi)f(\xi) d\xi + \int_0^r [\mathcal{Q}_2(t, \xi) - \mathcal{Q}_2(r, \xi)]f(\xi) d\xi.$$
It follows from Lemma \ref{lem3.4} that
$$\left\| \int_r^t \mathcal{Q}_2(t, \xi)f(\xi) d\xi \right\|_{X_0} \le C\|f\|_{C([0, T]; X_0)} \int_r^t (t - \xi)^{\vartheta - 1} d\xi \le C(t - r)^\vartheta \|f\|_{C([0, T]; X_0)}$$
and
\begin{align*}
\left\| \int_0^r [\mathcal{Q}_2(t, \xi) - \mathcal{Q}_2(r, \xi)]f(\xi) d\xi \right\|_{X_0} &\le C_\beta (t - r)^\beta \|f\|_{C([0, T]; X_0)} \int_0^r (r - \xi)^{\vartheta - \beta - 1} d\xi\\
&\le C_\beta (t - r)^\beta r^{\vartheta - \beta} \|f\|_{C([0, T]; X_0)}.
\end{align*}
Hence, $$\|g(t) - g(r)\|_{X_0} \le C\|f\|_{C([0, T]; X_0)} \left[ (t - r)^\vartheta + (t - r)^\beta r^{\vartheta - \beta} \right].$$
In particular, $g$ is locally $\beta$-H\"{o}lder continuous on $(0, T]$.
Then we deduce
\begin{align*}
h_3(t) &\le C\|f\|_{C([0, T]; X_0)}\left( \int_0^t (t - r)^{\vartheta - b - 1} dr +  \int_0^t (t - r)^{\beta - b - 1} r^{\vartheta - \beta} dr\right)
\le C_\beta t^{\vartheta - b} \|f\|_{C([0, T]; X_0)}.
\end{align*}
Since $\vartheta > b$, $h_3 \in L^1(0, T)$. Then $\phi_3$ satisfies the hypotheses of Lemma \ref{lem3.9}. The proof is complete.
\end{proof}

\begin{lemma}\label{lem3.12}
Let Assumption \ref{assump2.1} hold. Fix $s \in [0, T)$, and let $\phi : (s, T] \to X_0$ satisfy the hypotheses of Lemma \ref{lem3.9}. For $t \in (s, T]$, define
$$w_\phi(t, s) := \int_s^t \psi_{\mu, A(r)}(t - r)\phi(r, s) dr.$$
Then for $t > s$, we get $w_\phi(t, s) \in D(A(t))$ and $$ A(t)w_\phi(t, s) =B_\phi(t, s) + J_\phi(t, s),$$
where $B_\phi(t, s) := \int_s^t [A(t) - A(r)]\psi_{\mu,A(r)}(t - r)\phi(r, s) dr$ is absolutely convergent and  $J_\phi(t, s)$ is defined in Lemma \ref{lem3.9}.
Furthermore,
$$A(\cdot)w_\phi(\cdot, s) \in C((s, T]; X_0).$$
\end{lemma}
\begin{proof}
For $\rho > 0$, define
$$w_{\phi, \rho}(t, s) := \int_s^{t - \rho} \psi_{\mu, A(r)}(t - r)\phi(r, s) dr.$$
For any $x \in X_0$ and $r \le t - \rho$, we get $\psi_{\mu,A(r)}(t - r)x \in D(A(r))$. Since the domains are independent of time, $D(A(r)) = D(A(t))$.
Note that
$$A(t)\psi_{\mu, A(r)}(t - r)= [A(t) - A(r)]\psi_{\mu, A(r)}(t - r) + A(r)\psi_{\mu, A(r)}(t - r).$$
Since $A(t)$ is closed and both terms are Bochner integrable on $[s,t-\rho]$, it follows from the closed-operator theorem that $w_{\phi,\rho}(t, s) \in D(A(t))$ and
\begin{align*}
A(t)w_{\phi,\rho}(t, s) = \int_s^{t-\rho} [A(t) - A(r)]\psi_{\mu,A(r)}(t - r)\phi(r, s) \, dr  + \int_s^{t-\rho} A(r)\psi_{\mu,A(r)}(t - r)\phi(r, s) \, dr.
\end{align*}
It follows from Lemma \ref{lem3.1}(vi) that
$$\|[A(t) - A(r)]\psi_{\mu,A(r)}(t - r)\|_{L(X_0)}\le C(t - r)^{\vartheta-1}.$$
Together with $\|\phi(r, s)\|_{X_0} \le C(r - s)^{\delta-1}$ for $\delta > b \ge 0$, we obtain
\begin{align*}
\|B_\phi(t, s)\|_{X_0}  \le C \int_s^t (t - r)^{\vartheta-1} (r - s)^{\delta-1} dr = CB(\vartheta, \delta)(t - s)^{\vartheta+\delta-1} < \infty. \end{align*}
Thus, $B_\phi(t, s)$ is an absolutely convergent Bochner integral. By Lemma \ref{lem3.9}, $$\lim_{\rho \to 0} \int_s^{t-\rho} A(r)\psi_{\mu,A(r)}(t - r)\phi(r, s) \, dr = J_\phi(t, s)$$
exists in $X_0$. Since $$\|\psi_{\mu,A(r)}(t - r)\phi(r, s)\|_{X_0} \le C(t - r)^{\mu-1}(r - s)^{\delta-1}$$ is integrable, $w_{\phi,\rho}(t, s) \to w_\phi(t, s)$ in $X_0$.
Thus, we have $A(t)w_{\phi,\rho}(t, s) \to B_\phi(t, s) + J_\phi(t, s)$ in $X_0$. Since $A(t)$ is closed, it follows that $w_\phi(t, s) \in D(A(t))$ and $$A(t)w_\phi(t, s) = B_\phi(t, s) + J_\phi(t, s).$$

It remains to prove continuity. It follows from Lemma \ref{lem3.9} that $q_\phi(t, s) = \phi(t, s) - J_\phi(t, s)$ is locally H\"{o}lder continuous on $(s, T]$. Since $\phi(\cdot, s)$ is also locally H\"{o}lder continuous, $$J_\phi(t, s) = \phi(t, s) - q_\phi(t, s) \in C((s, T]; X_0).$$

Next, we prove that $B_\phi(\cdot, s)$ is continuous. For $s < r < t$, set$$L(t, r) := [A(t) - A(r)]\psi_{\mu,A(r)}(t - r).$$
Then $\|L(t, r)\|_{L(X_0)} \le C(t - r)^{\vartheta-1}$. Let $s < \sigma < t \le T$, then
$$ B_\phi(t, s) - B_\phi(\sigma, s) = \int_\sigma^t L(t, r)\phi(r, s) dr + \int_s^\sigma [L(t, r) - L(\sigma, r)]\phi(r, s) \, dr.$$
Since $\phi$ is bounded on $[\sigma,T]$,
$$ \left\| \int_\sigma^t L(t, r)\phi(r, s) dr \right\|_{X_0} \le C_\sigma \int_\sigma^t (t - r)^{\vartheta-1} dr \le C_\sigma(t - \sigma)^\vartheta \to 0, $$
as $t \to \sigma$. Fix $d > 0$ sufficiently small, then
\begin{align*}
&\left\| \int_s^\sigma [L(t,r) - L(\sigma,r)]\phi(r,s) dr \right\|_{X_0}\\
\le& \int_s^{\sigma-d} \|[L(t,r) - L(\sigma,r)]\phi(r,s)\|_{X_0} dr + \int_{\sigma-d}^\sigma \|[L(t,r) - L(\sigma,r)]\phi(r,s)\|_{X_0} dr
=: I_1(t) + I_2(t).
\end{align*}
On $[s, \sigma - d]$, for $t \ge \sigma$, $t - r \ge \sigma - r \ge d > 0$. By the strong continuity of  $L(t,r)$, we deduce
$$L(t, r)\phi(r, s) \to L(\sigma, r)\phi(r, s)$$
as $t \to \sigma$. According to the Lebesgue dominated convergence theorem, we obtain
$$\lim_{t \to \sigma} I_1(t) = \lim_{t \to \sigma} \int_s^{\sigma-d} \| [L(t,r) - L(\sigma,r)]\phi(r,s) \|_{X_0} dr = 0.$$
On $[\sigma - d, \sigma]$, since $\sup_{r\in[\sigma-d,\sigma]}
(r-s)^{\delta-1}<\infty$, we have
$\|\phi(r, s)\|_{X_0} \le C(r - s)^{\delta - 1} \le C$. By triangle inequality, we deduce
\begin{align*}
I_2(t) &\le \int_{\sigma-d}^\sigma \|L(t,r)\|_{L(X_0)}\|\phi(r,s)\|_{X_0} dr + \int_{\sigma-d}^\sigma \|L(\sigma,r)\|_{L(X_0)}\|\phi(r,s)\|_{X_0} dr\\
&\le C \int_{\sigma-d}^\sigma (t - r)^{\vartheta - 1} dr+ C \int_{\sigma-d}^\sigma (\sigma - r)^{\vartheta - 1} dr\le C\int_{\sigma-d}^\sigma (\sigma - r)^{\vartheta - 1} dr \leq Cd^\vartheta.
\end{align*}
Combining the estimates for $I_1(t)$ and $I_2(t)$, we get
$$\limsup_{t \to \sigma} \left\| \int_s^\sigma [L(t,r) - L(\sigma,r)]\phi(r,s) dr \right\|_{X_0} \le \lim_{t \to \sigma} I_1(t) + \limsup_{t \to \sigma} I_2(t) \le C d^\vartheta.$$

To prove the left continuity, we consider $s < t < \sigma \le T$. By the similar method, we have the same result. Letting $t\to\sigma$ and then letting $d \to 0$, we obtain $B_\phi(t, s) \to B_\phi(\sigma, s)$.
Hence, we have
$$B_\phi(\cdot, s) \in C((s, T]; X_0).$$
Since both $B_\phi(\cdot, s)$ and $J_\phi(\cdot, s)$ are continuous,
$$A(\cdot)w_\phi(\cdot, s) = B_\phi(\cdot, s) + J_\phi(\cdot, s) \in C((s, T]; X_0).$$
This completes the proof.
\end{proof}

\begin{definition}\label{def3.1}
Let $u_0 \in X_0$ and $f \in C([0, T]; X_0)$. A function $u : (0, T] \to X_0$ is called a classical solution of problem \eqref{e3.1} if the following conditions hold.
\begin{itemize}
\item[{\rm (i)}]  $u \in C((0, T]; X_0)$, $u(t) \in D(A(t))$ for $t \in (0, T]$ and $A(\cdot)u(\cdot) \in C((0, T]; X_0) \cap L^1(0, t; X_0)$.
\item [{\rm (ii)}]  The generalized derivative $ D_{0+}^{\lambda,\mu}u \in C((0, T]; X_0)$.
\item [{\rm (iii)}]  $I_{0+}^{1-\nu}u \in C([0, T]; X_0) \cap W_{\text{loc}}^{1,1}((0, T); X_0)$ and $\lim_{t \to 0} I_{0+}^{1-\nu}u(t) = u_0$. For $\nu = 1$, then $\lambda=1$ and $u(0) = u_0$.
\end{itemize}
\end{definition}

According to Lemma \ref{lem3.12}, we deduce $A(t)w_\phi(t, s) = B_\phi(t, s) + J_\phi(t, s)$ and
$$\|B_\phi(t, s)\|_{X_0} \le C \int_s^t (t - r)^{\vartheta - 1}(r - s)^{\delta - 1} dr \le C(t - s)^{\vartheta + \delta - 1}.$$
It follows from Lemma \ref{lem3.9} that
$$\|J_\phi(t, s)\|_{X_0} \le C \left[ h_s(t) + (1 + (t - s)^\vartheta)\|\phi(t, s)\|_{X_0} \right].$$
Since $h_s \in L^1(s, T)$, $\delta > 0$, and $\|\phi(t, s)\|_{X_0} \le C(t - s)^{\delta - 1}$, $B_\phi(\cdot, s)$ and $J_\phi(\cdot, s)$ belong to $L^1(s, T; X_0)$. Hence, we have
$$A(\cdot)w_\phi(\cdot, s) \in C((s, T]; X_0) \cap L^1(s, T; X_0).$$

\begin{theorem}\label{the3.1}
Let Assumption \ref{assump2.1} hold and set $b := \nu - \mu = \lambda(1 - \mu)$. Suppose that
$0 \le b < \min\{\rho, \vartheta\}$, $u_0 \in D(A(0))$, $f \in C^\rho([0, T]; X_0)$ with $\rho \in (0, 1]$. Then for $t \in (0, T]$, problem \eqref{e3.1} admits a unique classical solution given by
\begin{align}
u(t) = H_{\mu, \nu}(t, 0)u_0 + \int_0^t \mathcal{P}_\mu(t, s)f(s) ds.
\end{align}
\end{theorem}
\begin{proof}
Let $$u(t) =u_{I}(t) + u_f(t),$$
where $u_{I}(t) = H_{\mu, \nu}(t, 0)u_0$ and $u_f(t) = \int_0^t \mathcal{P}_\mu(t, s)f(s) ds$. We first consider the initial value component. By the definition of $H_{\mu, \nu}$,$$u_{I}(t) = \psi_{\nu, A(0)}(t)u_0 + \int_0^t \psi_{\mu, A(r)}(t - r)\mathcal{Q}_1(r, 0)u_0 dr.$$
Set $\phi_1(r) := \mathcal{Q}_1(r, 0)u_0$.
It follows from Lemma \ref{lem3.11} that $\phi_1$ satisfies the hypotheses required in Lemmas \ref{lem3.9}, \ref{lem3.10}, and \ref{lem3.12}. Therefore, by Lemma \ref{lem3.10}, we have
\begin{align*}
D_{0+}^{\lambda,\mu} \int_0^t \psi_{\mu, A(r)}(t - r)\phi_1(r) dr = \mathcal{Q}_1(t, 0)u_0 - \int_0^t A(r)\psi_{\mu, A(r)}(t - r)\mathcal{Q}_1(r, 0)u_0 dr.
\end{align*}
Since $u_0 \in D(A(0))$,
\begin{align*}
D_{0+}^{\lambda,\mu} (\psi_{\nu,A(0)}(t)u_0) &= I_{0+}^{\lambda(1-\mu)} \frac{d}{dt} I_{0+}^{1-\nu} (\psi_{\nu,A(0)}(t)u_0)\\
&=I_{0+}^{\lambda(1-\mu)}\frac{d}{dt} \left( S_{\mu,A(0)}(t) u_0 \right) = -A(0)\psi_{\nu,A(0)}(t)u_0.
\end{align*}
According to Lemma \ref{lem3.12}, we deduce
$\int_0^t \psi_{\mu, A(r)}(t - r)\mathcal{Q}_1(r, 0)u_0 dr \in D(A(t))$.
Then we obtain $u_{I}(t) \in D(A(t))$ and
\begin{align*}
&D_{0+}^{\lambda,\mu}u_{I}(t) + A(t)u_{I}(t) \\
={}&\mathcal Q_1(t,0)u_0+
\bigl(A(t)-A(0)\bigr)
\psi_{\nu,A(0)}(t)u_0+
\int_0^t
\bigl(A(t)-A(r)\bigr)
\psi_{\mu,A(r)}(t-r)
\mathcal Q_1(r,0)u_0\,dr\\
=&
\left(
\mathcal Q_1(t,0)
-\widetilde{\mathcal Q}_1(t,0)
-\int_0^t
\widetilde{\mathcal Q}_2(t,r)
\mathcal Q_1(r,0)\,dr
\right)u_0.
\end{align*}
It follows from \eqref{e3.5} that
$$D_{0+}^{\lambda,\mu}u_{I}(t) + A(t)u_{I}(t) = 0.$$

By Lemmas \ref{lem3.1} (ii), \ref{lem3.4}, and the definition of $\mathcal{P}_\mu$, we deduce $\mathcal{P}_\mu$ is absolutely integrable. Combining with Fubini's theorem, we obtain
$$u_f(t) = \int_0^t \psi_{\mu, A(s)}(t - s)f(s) ds + \int_0^t \psi_{\mu, A(r)}(t - r)g(r) dr,$$
where $g(r) := \int_0^r \mathcal{Q}_2(r, s)f(s) ds$. Similarly, by Lemma \ref{lem3.11}, both $f$ and $g$ satisfy the hypotheses required in Lemmas \ref{lem3.9}, \ref{lem3.10}, and \ref{lem3.12}. It follows from Lemma \ref{lem3.10} that
$$D_{0+}^{\lambda,\mu} \int_0^t \psi_{\mu, A(s)}(t - s)f(s) ds = f(t) - \int_0^t A(s)\psi_{\mu, A(s)}(t - s)f(s) ds,$$
$$D_{0+}^{\lambda,\mu} \int_0^t \psi_{\mu, A(r)}(t - r)g(r) dr = g(t) - \int_0^t A(r)\psi_{\mu, A(r)}(t - r)g(r) dr.$$
Moreover, by Lemma \ref{lem3.12}, we have $u_f(t) \in D(A(t))$. Then by Fubini's theorem, we deduce
\begin{align*}
&D_{0+}^{\lambda,\mu}u_f(t)+A(t)u_f(t)\\
=&f(t)
+\int_0^t\mathcal Q_2(t,s)f(s)\,ds-
\int_0^t
\widetilde{\mathcal Q}_2(t,s)f(s)\,ds-
\int_0^t
\widetilde{\mathcal Q}_2(t,r)
\int_0^r
\mathcal Q_2(r,s)f(s)\,ds\,dr\\
=&f(t)
+\int_0^t\mathcal Q_2(t,s)f(s)\,ds-
\int_0^t
\widetilde{\mathcal Q}_2(t,s)f(s)\,ds-\int_0^t
\int_s^t
\widetilde{\mathcal Q}_2(t,r)
\mathcal Q_2(r,s)\,drf(s)\,ds\\
=&f(t)-\int_0^t
\left(
\widetilde{\mathcal Q}_2(t,s)
-\mathcal Q_2(t,s)+
\int_s^t
\widetilde{\mathcal Q}_2(t,r)
\mathcal Q_2(r,s)\,dr
\right)f(s)\,ds.
\end{align*}
From \eqref{e3.5}, we obtain
$$D_{0+}^{\lambda,\mu} u_f(t) + A(t)u_f(t) = f(t).$$
Therefore,
$$D_{0+}^{\lambda,\mu} u(t) + A(t)u(t) = f(t), \qquad t \in (0, T].$$
It remains to verify the initial condition. By Lemma \ref{lem3.7},$$\lim_{t \to 0} I_{0+}^{1-\nu} (H_{\mu, \nu}(\cdot, 0)u_0)(t) = u_0.$$
Then we obtain
$$\|u_f(t)\|_{X_0} \le C \|f\|_{C([0, T]; X_0)} \int_0^t (t - s)^{\mu - 1} ds \le C t^\mu \|f\|_{C([0, T]; X_0)}.$$
If $0 < \nu < 1$, then
\begin{align*}
\|I_{0+}^{1-\nu} u_f(t)\|_{X_0} &\le \frac{C}{\Gamma(1-\nu)} \int_0^t (t - r)^{-\nu} r^\mu dr \, \|f\|_{C([0, T]; X_0)} \\ &\le C t^{\mu + 1 - \nu} \|f\|_{C([0, T]; X_0)} = C t^{1 - b} \|f\|_{C([0, T]; X_0)} \to 0.
\end{align*}
If $\nu = 1$, then $I_{0+}^{1-\nu}$ is the identity operator and
$\|u_f(t)\|_{X_0} \le C t^\mu \|f\|_{C([0, T]; X_0)} \to 0$.
Hence, we have
$$\lim_{t \to 0} I_{0+}^{1-\nu} u(t) = u_0.$$
It follows from Lemmas \ref{lem3.7} and \ref{lem3.9}-\ref{lem3.12} that $u$ satisfies the spatial regularity in Definition \ref{def3.1}.
For $x\in D(A(0))$ and $z\in\rho(A(0))$, we have
$A(0)(z-A(0))^{-1}x=(z-A(0))^{-1}A(0)x$.
By the definition of
$\psi_{\nu,A(0)}(t)$ and the closedness of $A(0)$, we have $\psi_{\nu,A(0)}(t)x\in D(A(0))$,
\begin{align}\label{e3.23}
A(0)\psi_{\nu,A(0)}(t)x
=\psi_{\nu,A(0)}(t)A(0)x.
\end{align}
Since $u_0\in D(A(0))$, it follows that
\begin{align*}
\|A(t)\psi_{\nu,A(0)}(t)u_0\|_{X_0}
&=\|A(t)A(0)^{-1}A(0)\psi_{\nu,A(0)}(t)u_0\|_{X_0}=\|A(t)A(0)^{-1}\psi_{\nu,A(0)}(t)A(0)u_0\|_{X_0}\\
&\leq \|A(t)A(0)^{-1}\|_{L(X_0)}
\|\psi_{\nu,A(0)}(t)A(0)u_0\|_{X_0}\leq Ct^{\nu-1}\|A(0)u_0\|_{X_0}.
\end{align*}
Hence, we deduce $A(\cdot)\psi_{\nu, A(0)}(\cdot)u_0 \in L^1(0, t; X_0)$. Furthermore, by Lemma \ref{lem3.12}, the integral terms corresponding to $\phi_1$, $f$, and $g$ in the solution $u(t)$ are $L^1$-integrable. Then
$$A(\cdot)u(\cdot) \in C((0, T]; X_0) \cap L^1(0, t; X_0).$$ Therefore, $u$ is a classical solution.

We finally prove uniqueness. Let $u_1$ and $u_2$ be two classical solutions with the same initial data, and put $z := u_1 - u_2$.
Then$$D_{0+}^{\lambda,\mu} z(r) + A(r)z(r) = 0, \quad  I_{0+}^{1-\nu} z(r) = 0.$$
Fix $t \in (0, T]$. On $(0, t]$, we get
$$D_{0+}^{\lambda,\mu} z(r) + A(t)z(r) = [A(t) - A(r)]z(r).$$
Let $F_t(r) := [A(t) - A(r)]z(r)$. It follows from \eqref{e2.4} that
$$\|F_t(r)\|_{X_0} \le C |t - r|^\vartheta \|A(r)z(r)\|_{X_0}.$$
Since  $(t - r)^\vartheta$ is uniformly bounded on $(0, t]$ and $A(\cdot)z(\cdot) \in L^1(0, t; X_0)$, $F_t \in L^1(0, t; X_0)$. By the method of variation of constant to autonomous generalized fractional equations, we have
$$z(t) = \int_0^t \psi_{\mu, A(t)}(t - r)[A(t) - A(r)]z(r) dr.$$
Let $K(t) := \|A(t)z(t)\|_{X_0}$. Combining Lemma \ref{lem3.1}(iii) and \eqref{e2.4}, we get
$$K(t) \le \int_0^t C(t - r)^{-1} (t - r)^\vartheta \|A(r)z(r)\|_{X_0} dr=C \int_0^t (t - r)^{\vartheta - 1} K(r) dr.$$
Since $K \in L^1(0, T)$, by the method of successive iterations, we deduce the $n$-th iteration
$$K(t) \le \frac{(C \Gamma(\vartheta))^n}{\Gamma(n\vartheta)} \int_0^t (t - r)^{n\vartheta - 1} K(r) dr.$$
Choose an integer $n$ sufficiently large such that $n\vartheta \ge 1$, we have
$$K(t) \le \frac{(C \Gamma(\vartheta))^n T^{n\vartheta - 1}}{\Gamma(n\vartheta)} \int_0^t K(r) dr \le \frac{(C \Gamma(\vartheta))^n T^{n\vartheta - 1}}{\Gamma(n\vartheta)} \|K\|_{L^1(0, T)} \to 0 \quad \text{as}\,\,n \to \infty.$$
Hence, we have $K(t) = 0$ for $t \in (0, T]$. Then we deduce $A(t)z(t) = 0$ and $z(t) = 0$ for $t \in (0, T]$.
Thus, $u_1(t) = u_2(t)$ on $(0, T]$ and we establish the uniqueness of the classical solution. The proof is complete.
\end{proof}

\section{Stochastic fractional evolution equations}\label{sec4}
In this section, we establish the existence and uniqueness of mild solutions for the non-autonomous fractional stochastic evolution equation \eqref{e1.1}.
Consider
\begin{equation}\label{e4.1}
\left\{
\begin{aligned}
&D_{0+}^{\lambda,\mu}u(t)+A(t)u(t)=F(t,u(t))+ G(t, u(t))dW_H(t), & t \in (0, T],  \\
&I_{0+}^{1-\nu}u(0)=u_0,
\end{aligned}\right.
\end{equation}
where $\nu=\mu+\lambda(1-\mu)$, $W_H$ is a cylindrical Brownian motion on a filtered probability space $(\Omega,\mathcal F,(\mathcal F_t)_{t\in[0,T]},\mathbb P)$, and $H$ is a separable Hilbert space.

The following estimates are the fractional-domain versions of Lemmas~\ref{lem3.4} and~\ref{lem3.7}.

\begin{lemma}\label{lem4.1}
Let Assumptions~\ref{assump2.1} and~\ref{assump2.2} hold, and let
$\alpha\in[0,1)$. Then for all $0\leq s<t\leq T$, we have
\begin{align}\label{e4.2}
\|\mathcal P_\mu(t,s)\|_{L(X_0,X_\alpha)}\leq C(t-s)^{\mu(1-\alpha)-1}.
\end{align}
Furthermore, for $\mu\alpha\leq\vartheta$, we have
\begin{align}\label{e4.3}
\|H_{\mu,\nu}(t,s)\|_{L(X_\alpha)}
\leq C(t-s)^{\nu-1}.
\end{align}
\end{lemma}
\begin{proof}
Let $r,s\in[0,T]$ and $\tau>0$. By \eqref{e2.4}, we obtain
$$\|A(r)A(s)^{-1}\|_{L(X_0)}
\leq 1+\|(A(r)-A(s))A(s)^{-1}\|_{L(X_0)}
\leq 1+CT^\vartheta.$$
For $x\in X_0$, we have
$T_{A(s)}(\tau)x\in D(A(s))=D(A(r))$.
For $0<\alpha<1$, it follows from the moment inequality and \eqref{e2.2} that
\begin{align*}
\|A(r)^\alpha T_{A(s)}(\tau)x\|_{X_0}
&\leq C\|A(r)T_{A(s)}(\tau)x\|_{X_0}^\alpha
\|T_{A(s)}(\tau)x\|_{X_0}^{1-\alpha}\\
&\leq C\|A(r)A(s)^{-1}\|_{L(X_0)}^\alpha
\|A(s)T_{A(s)}(\tau)x\|_{X_0}^\alpha
\|T_{A(s)}(\tau)x\|_{X_0}^{1-\alpha}\\
&\leq C\tau^{-\alpha}\|x\|_{X_0}.
\end{align*}
For $\alpha=0$, the same result holds by \eqref{e2.2}. From Lemmas \ref{lem2.3} (iv), \ref{lem2.4}, and \ref{lem3.1}(i), we get
\begin{align}\label{e4.4}
\|\psi_{\mu,A(s)}(t)\|_{L(X_0,X_\alpha)}
\leq C\|A(r)^{\alpha}\psi_{\mu,A(s)}(t)\|_{L(X_0)}
\leq Ct^{\mu(1-\alpha)-1}.
\end{align}
 By Lemma \ref{lem3.4}, we obtain
\begin{align*}
\|\int_s^t A(t)^{\alpha}\psi_{\mu,A(\tau)}(t-\tau)\mathcal{Q}_2 (\tau,s) d\tau \|_{L(X_0)}&\leq C \int_s^t (t -\tau)^{\mu(1-\alpha)-1} (\tau - s)^{\vartheta-1} d\tau\\
&\leq C (t - s)^{\mu(1-\alpha) + \vartheta - 1}.
\end{align*}
Hence, we obtain
\begin{align*}
\|\mathcal P_\mu(t,s)\|_{L(X_0,X_\alpha)}\leq C(t-s)^{\mu(1-\alpha)-1}.
\end{align*}
This proves \eqref{e4.2}.
Let $x\in X_\alpha$. By Lemma \ref{lem2.4}, we obtain $X_\alpha=D(A(s)^\alpha)$.  For
$z\in\rho(A(s))$, together with the  definition of $A(s)^{-\alpha}$, we have
$$(z-A(s))^{-1}A(s)^{-\alpha}
=A(s)^{-\alpha}(z-A(s))^{-1}.$$
For $x\in D(A(s)^\alpha)$, put $y=A(s)^\alpha x$, then $x=A(s)^{-\alpha}y$ and
$$(z-A(s))^{-1}x
=A(s)^{-\alpha}(z-A(s))^{-1}y\in D(A(s)^\alpha).$$
Hence, we deduce
\begin{equation*}
A(s)^\alpha(z-A(s))^{-1}x
=(z-A(s))^{-1}A(s)^\alpha x.
\end{equation*}
Combining with
$$\psi_{\nu,A(s)}(t-s)x
=\frac{(t-s)^{\nu-1}}{2\pi i}
\int_{\mathcal C}E_{\mu,\nu}(-z(t-s)^\mu)
(z-A(s))^{-1}x\,dz,$$
and the closedness of $A(s)^\alpha$, we get
$\psi_{\nu,A(s)}(t-s)x\in D(A(s)^\alpha)$
and
$$A(s)^\alpha\psi_{\nu,A(s)}(t-s)x
=\psi_{\nu,A(s)}(t-s)A(s)^\alpha x.$$
Combining with Lemma \ref{lem2.4} and Lemma \ref{lem3.2}(i), we obtain
\begin{align*}
\|\psi_{\nu,A(s)}(t-s)x\|_{X_\alpha}
&\leq
C\|A(s)^\alpha\psi_{\nu,A(s)}(t-s)x\|_{X_0}=
C\|\psi_{\nu,A(s)}(t-s)A(s)^\alpha x\|_{X_0}\\
&\leq
C\|\psi_{\nu,A(s)}(t-s)\|_{ L(X_0)}
\|A(s)^\alpha x\|_{X_0}\leq
C(t-s)^{\nu-1}\|x\|_{X_\alpha}.
\end{align*}
Hence, it yields
$$\|\psi_{\nu,A(s)}(t-s)\|_{L(X_\alpha)}
\leq
C(t-s)^{\nu-1}.$$
Since $X_\alpha\hookrightarrow X_0$, it follows from \eqref{e3.7} and \eqref{e4.4} that
\begin{align*}
\left\|
\int_s^t
\psi_{\mu,A(r)}(t-r)\mathcal Q_1(r,s)x\,dr
\right\|_{X_\alpha}
\leq&
\int_s^t
\|\psi_{\mu,A(r)}(t-r)\|_{L(X_0,X_\alpha)}
\|\mathcal Q_1(r,s)x\|_{X_0}\,dr\\
\leq&
C\int_s^t
(t-r)^{\mu(1-\alpha)-1}
(r-s)^{\vartheta+\nu-\mu-1}
\|x\|_{X_0}\,dr\\
\leq&
C\|x\|_{X_\alpha}
\int_s^t
(t-r)^{\mu(1-\alpha)-1}
(r-s)^{\vartheta+\nu-\mu-1}\,dr\\
\leq&C(t-s)^{\vartheta+\nu-\mu\alpha-1}\|x\|_{X_\alpha}.
\end{align*}
Since $\mu\alpha\leq\vartheta$, we deduce
$$(t-s)^{\vartheta+\nu-\mu\alpha-1}
=(t-s)^{\nu-1}
(t-s)^{\vartheta-\mu\alpha}
\leq T^{\vartheta-\mu\alpha}
(t-s)^{\nu-1}.$$
Then $\left\|\int_s^t
\psi_{\mu,A(r)}(t-r)\mathcal Q_1(r,s)\,dr
\right\|_{ L(X_\alpha)}
\leq
C_T(t-s)^{\nu-1}$.
Hence,
$$\|H_{\mu,\nu}(t,s)\|_{ L(X_\alpha)}
\leq C(t-s)^{\nu-1}.$$
Then \eqref{e4.3} holds. The proof is complete.
\end{proof}

\begin{lemma}\label{lem4.2}
Let Assumptions~\ref{assump2.1} and~\ref{assump2.2} hold, and let
$\alpha\in[0,1)$. Then, for $r\in[0,T]$ and $t>0$,
\begin{equation}\label{e4.5}
\|\partial_t\psi_{\mu,A(r)}(t)\|_{L(X_0,X_\alpha)}
\leq Ct^{\mu(1-\alpha)-2}.
\end{equation}
Moreover, for every $\eta\in[0,1]$ and $0<t_1<t_2$,
\begin{equation}\label{e4.6}
\|\psi_{\mu,A(r)}(t_2)-\psi_{\mu,A(r)}(t_1)\|_{L(X_0,X_\alpha)}
\leq C_\eta (t_2-t_1)^\eta t_1^{\mu(1-\alpha)-\eta-1}.
\end{equation}
The constants are independent of $r$, $t$, $t_1$, and $t_2$.
Moreover, if
$0<\eta<\min\{\mu(1-\alpha),\vartheta\}$,
then for $0\leq s<t_1<t_2\leq T$, we get
\begin{align}\label{e4.7}
\|\mathcal P_\mu(t_2,s)-\mathcal P_\mu(t_1,s)\|_{L(X_0,X_\alpha)}
\leq C_\eta (t_2-t_1)^\eta(t_1-s)^{\mu(1-\alpha)-\eta-1}.
\end{align}
\end{lemma}
\begin{proof}
Let $t_0>0$ be arbitrary and
$I_{t_0}:=[t_0/2,3t_0/2]$. Since
$\frac{d}{dy}T_{A(r)}(y)=-A(r)T_{A(r)}(y)$, we get
$$\frac{d}{dt}T_{A(r)}(t^\mu\xi)
 =-\mu t^{\mu-1}\xi A(r)T_{A(r)}(t^\mu\xi).$$
It follows from Lemma \ref{lem2.4} that
\begin{align*}
\left\|\mu\xi M_\mu(\xi)
 \frac{d}{dt}T_{A(r)}(t^\mu\xi)\right\|_{L(X_0,X_\alpha)}&\leq Ct^{\mu-1}\xi^2M_\mu(\xi)
 \|A(r)^{\alpha+1}T_{A(r)}(t^\mu\xi)\|_{L(X_0)}\leq C_{t_0}\xi^{1-\alpha}M_\mu(\xi),
\end{align*}
where in the last inequality we used \eqref{e2.3} and
$$\|A(r)^{\alpha+1}T_{A(r)}(y)\|_{L(X_0)}
\leq \|A(r)^\alpha T_{A(r)}(y/2)\|_{L(X_0)}
       \|A(r)T_{A(r)}(y/2)\|_{L(X_0)}
\leq Cy^{-\alpha-1}.$$
By Lemma \ref{lem2.3}(iv), we have $\int_0^\infty\xi^{1-\alpha}M_\mu(\xi)\,d\xi<\infty.$ Hence, from  the Lebesgue dominated convergence theorem, we get
$$\frac{d}{dt}
\int_0^\infty
\mu\xi M_\mu(\xi)
T_{A(r)}(t^\mu\xi)\,d\xi
=\int_0^\infty
\mu\xi M_\mu(\xi)
\frac{d}{dt}T_{A(r)}(t^\mu\xi)\,d\xi.$$
Since $t_0>0$ is arbitrary, for every $t>0$, we deduce
\begin{align*}
\partial_t\psi_{\mu,A(r)}(t)
={}&(\mu-1)t^{\mu-2}\int_0^\infty
 \mu\xi M_\mu(\xi)T_{A(r)}(t^\mu\xi)\,d\xi-\mu t^{2\mu-2}\int_0^\infty
 \mu\xi^2M_\mu(\xi)A(r)T_{A(r)}(t^\mu\xi)\,d\xi.
\end{align*}
Therefore, we obtain
\begin{align*}
\|\partial_t\psi_{\mu,A(r)}(t)\|_{L(X_0,X_\alpha)}
\leq& Ct^{\mu-2}\int_0^\infty
 \xi M_\mu(\xi)(t^\mu\xi)^{-\alpha}\,d\xi+Ct^{2\mu-2}\int_0^\infty
 \xi^2M_\mu(\xi)(t^\mu\xi)^{-\alpha-1}\,d\xi\\
=&Ct^{\mu(1-\alpha)-2}
 \int_0^\infty\xi^{1-\alpha}M_\mu(\xi)\,d\xi\leq Ct^{\mu(1-\alpha)-2}.
\end{align*}
Then we prove \eqref{e4.5}. By
\eqref{e4.4} and the triangle inequality, we obtain
\begin{align*}
\|\psi_{\mu,A(r)}(t_2)-\psi_{\mu,A(r)}(t_1)\|_{L(X_0,X_\alpha)}\leq C\bigl(t_2^{\mu(1-\alpha)-1}+t_1^{\mu(1-\alpha)-1}\bigr)
\leq Ct_1^{\mu(1-\alpha)-1}.                  \end{align*}
Combining with \eqref{e4.5}, we have
\begin{align*}
\|\psi_{\mu,A(r)}(t_2)-\psi_{\mu,A(r)}(t_1)\|_{L(X_0,X_\alpha)}
\leq\int_{t_1}^{t_2}
\|\partial_\tau\psi_{\mu,A(r)}(\tau)\|_{L(X_0,X_\alpha)} \,d\tau
\leq C(t_2-t_1)t_1^{\mu(1-\alpha)-2}.              \end{align*}
According to the inequality $\min\{a,b\}\leq a^{1-\eta}b^\eta$, we obtain
\begin{align*}
\|\psi_{\mu,A(r)}(t_2)-\psi_{\mu,A(r)}(t_1)\|_{L(X_0,X_\alpha)}&\leq C_\eta
\bigl(t_1^{\mu(1-\alpha)-1}\bigr)^{1-\eta}
\bigl((t_2-t_1)t_1^{\mu(1-\alpha)-2}\bigr)^\eta\\
&=C_\eta(t_2-t_1)^\eta t_1^{\mu(1-\alpha)-\eta-1}.
\end{align*}
Then \eqref{e4.6} holds.

Next, we prove \eqref{e4.7}. Since
\begin{align*}
&\mathcal P_\mu(t_2,s)-\mathcal P_\mu(t_1,s)\\
={}&\psi_{\mu,A(s)}(t_2-s)-\psi_{\mu,A(s)}(t_1-s)+\int_s^{t_1}[\psi_{\mu,A(r)}(t_2-r)
-\psi_{\mu,A(r)}(t_1-r)]\mathcal Q_2(r,s)\,dr\\
&+\int_{t_1}^{t_2}\psi_{\mu,A(r)}(t_2-r)\mathcal Q_2(r,s)\,dr
=:{}I_1+I_2+I_3.
\end{align*}
By \eqref{e4.6}, we have
$$\|I_1\|_{L(X_0,X_\alpha)}
\leq C_\eta(t_2-t_1)^\eta(t_1-s)^{\mu(1-\alpha)-\eta-1}.$$
For $t_2-t_1\leq t_1-s$, then
\begin{align*}
\|I_2\|
&\leq C_\eta(t_2-t_1)^\eta\int_s^{t_1}
(t_1-r)^{\mu(1-\alpha)-\eta-1}(r-s)^{\vartheta-1}\,dr\\
&=C_\eta B(\mu(1-\alpha)-\eta,\vartheta)
(t_2-t_1)^\eta(t_1-s)^{\mu(1-\alpha)+\vartheta-\eta-1}\leq C_{\eta,T}(t_2-t_1)^\eta(t_1-s)^{\mu(1-\alpha)-\eta-1},
\end{align*}
and
\begin{align*}
\|I_3\|
&\leq C\int_{t_1}^{t_2}(t_2-r)^{\mu(1-\alpha)-1}(r-s)^{\vartheta-1}\,dr
\leq C(t_2-t_1)^{\mu(1-\alpha)}(t_1-s)^{\vartheta-1}\\
&\leq C_T(t_2-t_1)^\eta(t_1-s)^{\mu(1-\alpha)-\eta-1}.
\end{align*}
If $t_2-t_1>t_1-s$, it follows from \eqref{e4.2} that
\begin{align*}
\|\mathcal P_\mu(t_2,s)-\mathcal P_\mu(t_1,s)\|
&\leq C\bigl((t_2-s)^{\mu(1-\alpha)-1}+(t_1-s)^{\mu(1-\alpha)-1}\bigr)\\
&\leq C(t_1-s)^{\mu(1-\alpha)-1}\leq C(t_2-t_1)^\eta(t_1-s)^{\mu(1-\alpha)-\eta-1}.
\end{align*}
Thus, \eqref{e4.7} holds. The proof is complete.
\end{proof}

Fix $\kappa\in[0,1)$, recall that $X_\kappa=[X_0,X_1]_\kappa=D(A(t)^\kappa)$ for $t\in[0,T]$. Thus, $X_\kappa$ is the
fixed space in which the spatial regularity of the solution will be measured. And
we have the following assumptions about the nonlinear terms $F$ and $G$.
\begin{assumption}\label{assump4.1}
The maps
$$F:[0,T]\times\Omega\times X_\kappa\to X_0,
\quad
G:[0,T]\times\Omega\times X_\kappa\to\gamma(H,X_0)$$
are measurable with respect to the progressive $\sigma$-field on
$[0,T]\times\Omega$ and the Borel $\sigma$-field of $X_\kappa$. There exist
constants $L_F,L_G,C_F,C_G\geq0$ such that for all $x,y\in X_\kappa$ and
almost all $(t,\omega)$,
\begin{align}\label{e4.8}
\|F(t,\omega,x)-F(t,\omega,y)\|_{X_0}
\leq L_F\|x-y\|_{X_\kappa},
\end{align}
\begin{align}\label{e4.9}
\|G(t,\omega,x)-G(t,\omega,y)\|_{\gamma(H,X_0)}
\leq L_G\|x-y\|_{X_\kappa},
\end{align}
\begin{align}\label{e4.10}
\|F(t,\omega,x)\|_{X_0}
\leq C_F(1+\|x\|_{X_\kappa}),
\end{align}
\begin{align}\label{e4.11}
\|G(t,\omega,x)\|_{\gamma(H,X_0)}
\leq C_G(1+\|x\|_{X_\kappa}).
\end{align}
\end{assumption}
Let $p\in[2,\infty)$. We denote by $\mathcal H_{p,\kappa}^\nu(T)$ the
space of all progressively measurable processes
$$u:(0,T]\times\Omega\to X_\kappa$$
for which the map $t\mapsto t^{1-\nu}u(t)$
admits an extension belonging to
$C([0,T];L^p(\Omega;X_\kappa))$. This is a Banach space equipped with the norm
\begin{equation} \label{e4.12}
\|u\|_{\mathcal H_{p,\kappa}^\nu}
:=\sup_{0<t\leq T}t^{1-\nu}
\|u(t)\|_{L^p(\Omega;X_\kappa)}.
\end{equation}

\begin{definition}\label{def4.1}
Let $u_0\in L^p(\Omega,\mathcal F_0;X_\kappa)$. A process $u\in\mathcal H_{p,\kappa}^\nu(T)$ is called a mild solution of \eqref{e4.1} if for every
$t\in(0,T]$, almost surely,
\begin{equation}\label{e4.13}
u(t)={}H_{\mu,\nu}(t,0)u_0
+\int_0^t\mathcal P_\mu(t,s)F(s,u(s))\,ds+\int_0^t\mathcal P_\mu(t,s)G(s,u(s))\,dW_H(s),
\end{equation}
where $$H_{\mu, \nu}(t,0) = \psi_{\nu,A(0)}(t-0) + \int_0^t \psi_{\mu,A(r)}(t-r) \mathcal{Q}_1(r, 0) dr$$
$$\mathcal{P}_{\mu}(t,s) = \psi_{\mu,A(s)}(t-s) + \int_s^t \psi_{\mu,A(r)}(t-r) \mathcal{Q}_2(r, s) dr.$$
\end{definition}
For any $u \in {\mathcal{H}^\nu_{p, \kappa}}$, we define the operator $\mathcal{T}$ by
\begin{align*}
\mathcal{T}(u)(t)= {}H_{\mu,\nu}(t,0)u_0+\int_0^t\mathcal P_\mu(t,s)F(s,u(s))\,ds+\int_0^t\mathcal P_\mu(t,s)G(s,u(s))\,dW_H(s).
\end{align*}

\begin{lemma}\label{lem4.3}
Let Assumptions~\ref{assump2.1}, \ref{assump2.2}, and~\ref{assump4.1} hold. Let $u_0\in L^p(\Omega,\mathcal F_0;X_\kappa)$. Let $p\in[2,\infty)$ and
\begin{equation}\label{e4.14}
a:=\mu(1-\kappa)>\frac12,\quad \mu\kappa<\vartheta.
\end{equation}
Then for $u\in\mathcal H_{p,\kappa}^\nu(T)$, $\mathcal T$ is well defined and $\mathcal T$ maps
$\mathcal H_{p,\kappa}^\nu(T)$ into itself.
\end{lemma}
\begin{proof}
Let $u\in\mathcal H_{p,\kappa}^\nu(T)$ and put
$$M=\|u\|_{\mathcal H_{p,\kappa}^\nu}=
\sup_{0<t\leq T}
t^{1-\nu}
\|u(t)\|_{L^p(\Omega;X_\kappa)}.$$
Then for $s\in(0,T]$, we have $s^{1-\nu}
\|u(s)\|_{L^p(\Omega;X_\kappa)}
\leq M$. Set
$$(\mathcal Tu)(t)=I_1(t)+I_2(t)+I_3(t),$$
where $I_1(t)=H_{\mu,\nu}(t,0)u_0$ and
$$I_2(t)=\int_0^t\mathcal P_\mu(t,s)F(s,u(s))\,ds,\quad I_3(t)=\int_0^t\mathcal P_\mu(t,s)G(s,u(s))\,dW_H(s).$$
We first prove that these terms are well defined. By \eqref{e4.3},
$I_1(t)\in L^p(\Omega;X_\kappa)$ and
\begin{equation}\label{e4.15}
t^{1-\nu}\|I_1(t)\|_{L^p(\Omega;X_\kappa)}
\leq C\|u_0\|_{L^p(\Omega;X_\kappa)}.
\end{equation}
Since $u\in\mathcal H_{p,\kappa}^\nu(T)$,
\begin{equation}\label{e4.16}
\|u(s)\|_{L^p(\Omega;X_\kappa)}\leq Ms^{\nu-1}, \qquad 0<s\leq T.
\end{equation}
It follows from the measurability assumptions on $F$, \eqref{e4.2},
\eqref{e4.10}, and \eqref{e4.16} that
\begin{align*}
\int_0^t
\|\mathcal P_\mu(t,s)F(s,u(s))\|_{L^p(\Omega;X_\kappa)}\,ds\leq C\int_0^t(t-s)^{a-1}(1+Ms^{\nu-1})\,ds<\infty.
\end{align*}
Thus $I_2(t)$ is a well-defined $L^p(\Omega;X_\kappa)$-valued Bochner
integral.
Moreover, by Minkowski’s integral inequality, we have
\begin{equation}\label{e4.17}
\begin{aligned}
t^{1-\nu}\|I_2(t)\|_{L^p(\Omega;X_\kappa)}
&=t^{1-\nu}\left\|\int_0^t\mathcal P_\mu(t,s)F(s,u(s))\,ds\right\|_{L^p(\Omega;X_\kappa)}\\
&\leq t^{1-\nu}\int_0^t\|\mathcal P_\mu(t,s)F(s,u(s))\|_{L^p(\Omega;X_\kappa)}
\,ds\\
&\leq Ct^{1-\nu}\int_0^t
(t-s)^{a-1}(1+Ms^{\nu-1})\,ds\\
&\leq C\left(t^{a+1-\nu}+Mt^a\right).
\end{aligned}
\end{equation}

For $0<s<t$, put
$$\Phi_t(s):=\mathcal P_\mu(t,s)G(s,u(s)) .$$
For fixed $t\in(0,T]$, extend $\Phi_t$ by zero to the whole interval
$[0,T]$ and set
$$\widetilde\Phi_t(s):=\mathbf 1_{[0,t]}(s)\Phi_t(s),
\quad 0\leq s\leq T.$$
By the property of $\gamma$-radonifying operators, together with \eqref{e4.2} and \eqref{e4.11}, we obtain
\begin{align*}
\|\Phi_t(s)\|
 _{L^p(\Omega;\gamma(H,X_\kappa))}
\le&\|\mathcal P_\mu(t,s)\|_{L(X_0,X_\kappa)}
\|G(s,u(s))\|_{L^p(\Omega;\gamma(H,X_0))}\\
\le&C(t-s)^{a-1}
\left(1+\|u(s)\|_{L^p(\Omega;X_\kappa)}\right)\\
\le&C(t-s)^{a-1}\left(1+Ms^{\nu-1}\right).
\end{align*}

Since $\mathcal P_\mu(t,s)$ is deterministic and strongly measurable in
$s$, both $\Phi_t$ on $[0,t]$ and its zero extension $\widetilde\Phi_t$ on $[0,T]$ are progressively measurable. From $(1+x)^2\leq2(1+x^2)$, we deduce
\begin{align*}
\int_0^t\|\Phi_t(s)\|_{L^p(\Omega;\gamma(H,X_\kappa))}^2\,ds&\leq C\int_0^t(t-s)^{2a-2}(1+Ms^{\nu-1})^2\,ds\\
&\leq 2C\int_0^t(t-s)^{2a-2}\,ds
+2CM^2\int_0^t(t-s)^{2a-2}s^{2\nu-2}\,ds\\
&\leq \frac{2Ct^{2a-1}}{2a-1}
+2CM^2B(2a-1,2\nu-1)t^{2a+2\nu-3}<\infty.
\end{align*}
The integrals are finite since $2a-1>0$ by \eqref{e4.14} and $2\nu-1>0$ since $\nu\geq\mu>1/2$.

By Lemma~\ref{lem2.4}, $X_\kappa$ is a UMD space with type $2$. Since $p\geq2$, it follows from Minkowski's integral inequality that
\begin{align*}
\|\widetilde\Phi_t\|_{L^p(\Omega;L^2(0,T;\gamma(H,X_\kappa)))}&=\|\Phi_t\|_{L^p(\Omega;L^2(0,t;\gamma(H,X_\kappa)))}=\left\|\left(\int_0^t
\|\Phi_t(s)\|_{\gamma(H,X_\kappa)}^2\,ds\right)^{1/2}
\right\|_{L^p(\Omega)}\\
&\leq\left(\int_0^t
\|\Phi_t(s)\|_{L^p(\Omega;\gamma(H,X_\kappa))}^2\,ds
\right)^{1/2}<\infty.
\end{align*}
Then $\widetilde\Phi_t$ satisfies all conditions of
Lemma~\ref{lem2.1} on $[0,T]$. Since
$$\int_0^T\widetilde\Phi_t(s)\,dW_H(s)
=\int_0^t\Phi_t(s)\,dW_H(s)=I_3(t),$$
we have
\begin{align*}
\|I_3(t)\|_{L^p(\Omega;X_\kappa)}
&\leq\left(\mathbb E\sup_{0\leq\tau\leq T}
\left\|\int_0^\tau\widetilde\Phi_t(s)\,dW_H(s)\right\|_{X_\kappa}^p
\right)^{1/p}\leq C\|\widetilde\Phi_t\|_%
{L^p(\Omega;L^2(0,T;\gamma(H,X_\kappa)))}\\
&\leq C\left(\int_0^t
\|\Phi_t(s)\|_{L^p(\Omega;\gamma(H,X_\kappa))}^2\,ds\right)^{1/2}.
\end{align*}
Thus, $I_3(t)$ is a well-defined $X_\kappa$-valued stochastic integral.
Moreover,
\begin{align}\label{e4.18}
t^{1-\nu}\|I_3(t)\|_{L^p(\Omega;X_\kappa)}
leq Ct^{1-\nu}
\left(t^{2a-1}+M^2t^{2a+2\nu-3}\right)^{1/2}
\leq C\bigl(t^{a+1/2-\nu}+Mt^{a-1/2}\bigr).
\end{align}
Hence, $I_1(t), I_2(t)$ and $I_3(t)$ are $\mathcal F_t$-measurable. Thus, $\mathcal Tu$ is adapted.

It follows from \eqref{e4.15}, \eqref{e4.17} and \eqref{e4.18} that
$$\sup_{0<t\leq T}t^{1-\nu}
\|(\mathcal Tu)(t)\|_{L^p(\Omega;X_\kappa)}<\infty.$$

It remains to prove weighted continuity. Fix
$0<\varepsilon\leq t_1<t_2\leq T$, then
\begin{align*}
I_2(t_2)-I_2(t_1)
=&\int_{t_1}^{t_2}\mathcal P_\mu(t_2,s)F(s,u(s))\,ds+\int_0^{t_1}[\mathcal P_\mu(t_2,s)-\mathcal P_\mu(t_1,s)]
F(s,u(s))\,ds\\
=&J_{21}+J_{22}.
\end{align*}
Choose $0<\eta<\min\{a,\vartheta\}$. By \eqref{e4.2}, \eqref{e4.7},
and \eqref{e4.16}, we obtain
\begin{align*}
\|J_{21}\|_{L^p(\Omega;X_\kappa)}
&\leq C\int_{t_1}^{t_2}(t_2-s)^{a-1}
(1+Ms^{\nu-1})\,ds\\
&\leq C_{\varepsilon,M}\int_{t_1}^{t_2}(t_2-s)^{a-1}\,ds
\leq C_{\varepsilon,M}(t_2-t_1)^a,
\end{align*}
where we used $s^{\nu-1}\leq\varepsilon^{\nu-1}$ for
$s\in[t_1,t_2]$. For $a-\eta>0$ and $\nu>0$, it follows from \eqref{e4.7} that
\begin{align*}
\|J_{22}\|_{L^p(\Omega;X_\kappa)}
&\leq C(t_2-t_1)^\eta
\int_0^{t_1}(t_1-s)^{a-\eta-1}(1+Ms^{\nu-1})\,ds\\
&\leq C_{\varepsilon,T,M}(t_2-t_1)^\eta.
\end{align*}
Therefore, $I_2$ is continuous from $(0,T]$ into
$L^p(\Omega;X_\kappa)$. Similarly,
\begin{align*}
I_3(t_2)-I_3(t_1)
=&\int_{t_1}^{t_2}\mathcal P_\mu(t_2,s)G(s,u(s))\,dW_H(s)+\int_0^{t_1}[\mathcal P_\mu(t_2,s)-\mathcal P_\mu(t_1,s)]
G(s,u(s))\,dW_H(s)\\
=&J_{31}+J_{32}.
\end{align*}
Choose $0<\eta<\min\{a-\frac12,\vartheta\}$.
According to Lemma~\ref{lem2.1} and Minkowski's integral inequality and \eqref{e4.2}, we obtain
\begin{align*}
\|J_{31}\|_{L^p(\Omega;X_\kappa)}
&\leq C\left(\int_{t_1}^{t_2}(t_2-s)^{2a-2}
(1+Ms^{\nu-1})^2\,ds\right)^{1/2}\\
&\leq C_{\varepsilon,M}
\left(\int_{t_1}^{t_2}(t_2-s)^{2a-2}\,ds\right)^{1/2}\\
&\leq C_{\varepsilon,M}(t_2-t_1)^{a-1/2},
\end{align*}
where the integral is finite since $a>1/2$. Similarly, by \eqref{e4.7} and $(1+x)^2\leq2(1+x^2)$,
\begin{align*}
\|J_{32}\|_{L^p(\Omega;X_\kappa)}
&\leq C(t_2-t_1)^\eta
\left(\int_0^{t_1}(t_1-s)^{2a-2\eta-2}
(1+Ms^{\nu-1})^2\,ds\right)^{1/2}\\
&\leq C(t_2-t_1)^\eta
\left[\frac{2t_1^{2a-2\eta-1}}{2a-2\eta-1}+2M^2B(2a-2\eta-1,2\nu-1)
t_1^{2a-2\eta+2\nu-3}\right]^{1/2}\\
&\leq C_{\varepsilon,T,M}(t_2-t_1)^\eta,
\end{align*}
The beta integral is finite since $\eta<a-1/2$ and $\nu>1/2$, and it is uniformly bounded
for $t_1\in[\varepsilon,T]$.
Hence, $I_3$ is continuous from $(0,T]$ into $L^p(\Omega;X_\kappa)$.

Next, we prove the continuity of $I_1$ in $L^p(\Omega;X_\kappa)$. Fix $x\in X_\kappa$ and $t_0\in(0,T]$. By Lemma~\ref{lem2.4}, we deduce
$X_\kappa=D(A(0)^\kappa)$ and
$$
\|y\|_{X_\kappa}\leq C\|A(0)^\kappa y\|_{X_0},
$$
for $y\in X_\kappa$. According to the proof of Lemma~\ref{lem4.1}, we have
$A(0)^\kappa\psi_{\nu,A(0)}(t)x
=\psi_{\nu,A(0)}(t)A(0)^\kappa x$.
It yields
\begin{align*}
\|[\psi_{\nu,A(0)}(t)-\psi_{\nu,A(0)}(t_0)]x\|_{X_\kappa}
&\leq C\left\|A(0)^\kappa
[\psi_{\nu,A(0)}(t)-\psi_{\nu,A(0)}(t_0)]x\right\|_{X_0}\\
&=C\left\|[\psi_{\nu,A(0)}(t)-\psi_{\nu,A(0)}(t_0)]
A(0)^\kappa x\right\|_{X_0}.
\end{align*}
Since $A(0)^\kappa x\in X_0$, we deduce $$\left\|[\psi_{\nu,A(0)}(t)-\psi_{\nu,A(0)}(t_0)]
A(0)^\kappa x\right\|_{X_0} \to 0$$ as $t\to t_0$ for $(0,T]$ by the strong continuity of
$\psi_{\nu,A(0)}$ in $X_0$. Hence, $t\mapsto\psi_{\nu,A(0)}(t)x$ is continuous in $X_\kappa$ on $(0,T]$.
Let
$$H(t)x:=\int_0^t\psi_{\mu,A(r)}(t-r)\mathcal Q_1(r,0)x\,dr,$$
and $\gamma_1:=\vartheta+\nu-\mu>0$.
Fix $\varepsilon\leq t_1<t_2\leq T$, set $h=t_2-t_1$ and $0<\eta<a$. Then
\begin{align*}
H(t_2)x-H(t_1)x
={}&\int_0^{t_1}
\bigl[\psi_{\mu,A(r)}(t_2-r)-\psi_{\mu,A(r)}(t_1-r)\bigr]
\mathcal Q_1(r,0)x\,dr\\
&+\int_{t_1}^{t_2}
\psi_{\mu,A(r)}(t_2-r)\mathcal Q_1(r,0)x\,dr
=: H_1+H_2.
\end{align*}
It follows from \eqref{e3.7} and \eqref{e4.6} that
\begin{align*}
\|H_1\|_{X_\kappa}
\leq C h^\eta\|x\|_{X_0}
\int_0^{t_1}(t_1-r)^{a-\eta-1}r^{\gamma_1-1}\,dr\leq C_{\varepsilon,T}h^\eta\|x\|_{X_\kappa},
\end{align*}
By \eqref{e3.7} and \eqref{e4.4}, we get
\begin{align*}
\|H_2\|_{X_\kappa}
\leq C\|x\|_{X_0}
\int_{t_1}^{t_2}(t_2-r)^{a-1}r^{\gamma_1-1}\,dr\leq C_{\varepsilon,T}\|x\|_{X_0}
\int_{t_1}^{t_2}(t_2-r)^{a-1}\,dr\leq C_{\varepsilon,T}h^a\|x\|_{X_\kappa},
\end{align*}
where we use the continuous embedding $X_\kappa\hookrightarrow
X_0$.
Hence, we obtain $$
\|H(t_2)x-H(t_1)x\|_{X_\kappa}
\leq C_{\varepsilon,T}(h^\eta+h^a)\|x\|_{X_\kappa}
\to 0
$$
as $t_2\to t_1$. Then $t\mapsto H_{\mu,\nu}(t,0)x$ is continuous
in $X_\kappa$ for $x\in X_\kappa$.
Moreover, by \eqref{e4.3}, we get
$$\sup_{t\in[\varepsilon,T]}\|H_{\mu,\nu}(t,0)\|_{L(X_\kappa)}<\infty.$$
For $u_0\in L^p(\Omega;X_\kappa)$, let $x=u_0(\omega)$. By the Lebesgue dominated convergence theorem, we obtain
$$\mathbb E
\|H_{\mu,\nu}(t,0)u_0
-H_{\mu,\nu}(t_0,0)u_0\|_{X_\kappa}^{p}
\to 0$$
as $t \to t_0$. Hence, we have
$$I_1\in C((0,T];L^p(\Omega;X_\kappa)).$$

It remains to show that the weighted process $t^{1-\nu}(\mathcal Tu)(t)$ admits a continuous extension to $t=0$. Since $\nu=\mu+\lambda(1-\mu)\leq1$ and \eqref{e4.14}, we have
$a>1/2$, $a+1/2-\nu\geq a-1/2>0$, and $a+1-\nu\geq a>0$. According to \eqref{e4.17} and \eqref{e4.18}, we deduce
\begin{align*}
t^{1-\nu}\|I_2(t)\|_{L^p(\Omega;X_\kappa)}
&\leq C\left(t^{a+1-\nu}+Mt^a\right)
\to 0,
\end{align*}
\begin{align*}
t^{1-\nu}\|I_3(t)\|_{L^p(\Omega;X_\kappa)}
&\leq C\left(t^{a+1/2-\nu}+Mt^{a-1/2}\right)
\to 0.
\end{align*}
Hence, it yields
\begin{equation}\label{e4.19}
t^{1-\nu}I_2(t)\to 0,
\quad
t^{1-\nu}I_3(t)\to 0
\quad\text{in }L^p(\Omega;X_\kappa).
\end{equation}
For the initial term, for every
$x\in X_\kappa$, we have
\begin{align*}
&t^{1-\nu}H_{\mu,\nu}(t,0)x-\frac{x}{\Gamma(\nu)}\\
=&t^{1-\nu}\psi_{\nu,A(0)}(t)x-\frac{x}{\Gamma(\nu)}+t^{1-\nu}\int_0^t
\psi_{\mu,A(r)}(t-r)\mathcal Q_1(r,0)x\,dr.
\end{align*}
By Lemma~\ref{lem2.4}, Lemma~\ref{lem3.2}(iii), and $A(0)^\kappa\psi_{\nu,A(0)}(t)x=\psi_{\nu,A(0)}(t)A(0)^\kappa x$, we deduce
\begin{align*}
\left\|t^{1-\nu}\psi_{\nu,A(0)}(t)x
-\frac{x}{\Gamma(\nu)}\right\|_{X_\kappa}\leq& C\left\|A(0)^\kappa
\left(t^{1-\nu}\psi_{\nu,A(0)}(t)x
-\frac{x}{\Gamma(\nu)}\right)\right\|_{X_0}\\
=&C\left\|t^{1-\nu}\psi_{\nu,A(0)}(t)A(0)^\kappa x
-\frac{A(0)^\kappa x}{\Gamma(\nu)}\right\|_{X_0}
\to 0.
\end{align*}
Hence, we deduce
\begin{equation}\label{e4.20}
t^{1-\nu}\psi_{\nu,A(0)}(t)x
\to \frac{x}{\Gamma(\nu)}
\quad\text{in }X_\kappa.
\end{equation}
According to the proof of Lemma~\ref{lem4.1},  we have
$$\left\|
\int_0^t
\psi_{\mu,A(r)}(t-r)\mathcal Q_1(r,0)x\,dr
\right\|_{X_\kappa}
\le
Ct^{\vartheta+\nu-\mu\kappa-1}
\|x\|_{X_\kappa}.$$
Then
\begin{align*}
\left\|t^{1-\nu}\int_0^t
\psi_{\mu,A(r)}(t-r)\mathcal Q_1(r,0)x\,dr
\right\|_{X_\kappa}\leq Ct^{\vartheta-\mu\kappa}\|x\|_{X_\kappa}
\to 0,
\end{align*}
where the last limit uses $\mu\kappa<\vartheta$ in \eqref{e4.14}. Therefore, we obtain
\begin{align}\label{e4.21}
t^{1-\nu}H_{\mu,\nu}(t,0)x
\to \frac{x}{\Gamma(\nu)}
\quad\text{in }X_\kappa.
\end{align}
Together with \eqref{e4.3}, for $0<t\leq T$, we deduce
$$\|t^{1-\nu}H_{\mu,\nu}(t,0)x\|_{X_\kappa}
\leq C\|x\|_{X_\kappa}.$$
Since $u_0\in L^p(\Omega;X_\kappa)$,
$u_0(\omega)\in X_\kappa$ for almost every $\omega\in\Omega$. Then by \eqref{e4.21} with $x=u_0(\omega)$, we get
$$t^{1-\nu}H_{\mu,\nu}(t,0)u_0(\omega)
\to \frac{u_0(\omega)}{\Gamma(\nu)}
\quad\text{in }X_\kappa$$
for almost every $\omega\in\Omega$. From the triangle inequality, we have
\begin{align*}
\left\|t^{1-\nu}H_{\mu,\nu}(t,0)u_0(\omega)
-\frac{u_0(\omega)}{\Gamma(\nu)}\right\|_{X_\kappa}\leq
\left(C+\frac{1}{\Gamma(\nu)}\right)
\|u_0(\omega)\|_{X_\kappa},
\end{align*}
\begin{align*}
\left\|t^{1-\nu}H_{\mu,\nu}(t,0)u_0(\omega)
-\frac{u_0(\omega)}{\Gamma(\nu)}\right\|_{X_\kappa}^p\leq
\left(C+\frac{1}{\Gamma(\nu)}\right)^p
\|u_0(\omega)\|_{X_\kappa}^p.
\end{align*}
The function on the right hand side is integrable over $\Omega$ since
$u_0\in L^p(\Omega;X_\kappa)$. Hence, it follows from the Lebesgue dominated convergence theorem that
$$\mathbb E\left\|t^{1-\nu}H_{\mu,\nu}(t,0)u_0
-\frac{u_0}{\Gamma(\nu)}\right\|_{X_\kappa}^p
\to 0.$$
Then
\begin{equation}\label{e4.22}
t^{1-\nu}I_1(t)
=t^{1-\nu}H_{\mu,\nu}(t,0)u_0
\to \frac{u_0}{\Gamma(\nu)}
\quad\text{in }L^p(\Omega;X_\kappa).
\end{equation}
Since
$$t^{1-\nu}(\mathcal Tu)(t)
=t^{1-\nu}I_1(t)+t^{1-\nu}I_2(t)+t^{1-\nu}I_3(t),$$
it follows from the triangle inequality, \eqref{e4.19}, and \eqref{e4.22} that
\begin{align*}
&\left\|t^{1-\nu}(\mathcal Tu)(t)
-\frac{u_0}{\Gamma(\nu)}\right\|_{L^p(\Omega;X_\kappa)}\\
\leq&\left\|t^{1-\nu}I_1(t)
-\frac{u_0}{\Gamma(\nu)}\right\|_{L^p(\Omega;X_\kappa)}+t^{1-\nu}\|I_2(t)\|_{L^p(\Omega;X_\kappa)}
+t^{1-\nu}\|I_3(t)\|_{L^p(\Omega;X_\kappa)}
\to 0.
\end{align*}
For $0<t\leq T$, define the weighted process
$v(t):=t^{1-\nu}(\mathcal Tu)(t),$
and set $v(0):=\frac{u_0}{\Gamma(\nu)}$. Together with the continuity on
$(0,T]$, we deduce
$$v\in C([0,T];L^p(\Omega;X_\kappa)).$$

Next, we verify the initial condition in \eqref{e4.1}.  If $0<\nu<1$, we deduce
\begin{align*}
\bigl[I_{0+}^{1-\nu}(\mathcal Tu)\bigr](t)
&=\frac{1}{\Gamma(1-\nu)}
\int_0^t(t-s)^{-\nu}(\mathcal Tu)(s)\,ds=\frac{1}{\Gamma(1-\nu)}
\int_0^1(1-r)^{-\nu}r^{\nu-1}v(tr)\,dr.
\end{align*}
For $r\in(0,1]$, by the continuity of $v$ at zero, we have
$v(tr)\to v(0)=\frac{u_0}{\Gamma(\nu)}$ as $t\to 0$ in $L^p(\Omega;X_\kappa)$. Since $v$ is bounded on $[0,T]$ and $(1-r)^{-\nu}r^{\nu-1}$ belongs to $L^1(0,1)$, we
get $I_{0+}^{1-\nu}(\mathcal Tu)(0)
=\frac{B(\nu,1-\nu)}
{\Gamma(1-\nu)\Gamma(\nu)}u_0=u_0$ by
the Lebesgue dominated convergence theorem. If $\nu=1$, then $I_{0+}^{1-\nu}=I_{0+}^0$ is the identity operator,
and $I_{0+}^{0}(\mathcal Tu)(0)=v(0)=u_0$ in $L^p(\Omega;X_\kappa)$. Hence, it follows that $\mathcal Tu$ satisfies
$I_{0+}^{1-\nu}(\mathcal Tu)(0)=u_0$.

Moreover, $\mathcal Tu$ is adapted and $L^p(\Omega;X_\kappa)$-continuous
on $(0,T]$. Therefore, it admits a progressively measurable version. Together with the weighted continuity, we deduce
$\mathcal Tu\in\mathcal H_{p,\kappa}^{\nu}(T).$
The proof is complete.
\end{proof}

\begin{theorem}\label{the4.1}
Let Assumptions~\ref{assump2.1}, \ref{assump2.2}, ~\ref{assump4.1}, and \eqref{e4.14} hold. For $p\in[2,\infty)$, let $u_0\in L^p(\Omega,\mathcal F_0;X_\kappa)$. Then equation \eqref{e4.1} has a unique mild solution $u\in\mathcal H_{p,\kappa}^\nu(T)$. Moreover,
\begin{equation}\label{e4.23}
\|u\|_{\mathcal H_{p,\kappa}^\nu}
\leq C_T\left(1+\|u_0\|_{L^p(\Omega;X_\kappa)}\right),
\end{equation}
and
\begin{equation}\label{e4.24}
\lim_{t\to 0}t^{1-\nu}u(t)
=\frac{u_0}{\Gamma(\nu)}
\quad\text{in }L^p(\Omega;X_\kappa).
\end{equation}
Moreover, let $\delta\in[0,1-\kappa)$ satisfy
$\mu(\kappa+\delta)\leq\vartheta$ and put
$a_\delta=\mu(1-\kappa-\delta)=a-\mu\delta$. For every $\varepsilon\in(0,T)$ and
\begin{equation}\label{e4.25}
0<k^\ast<\min\left\{a_\delta-\frac12,\vartheta\right\},
\end{equation}
the regular part
$v(t):=u(t)-H_{\mu,\nu}(t,0)u_0,$
has the local regularity
\begin{equation}\label{e4.26}
v\in C^{k^\ast}([\varepsilon,T];L^p(\Omega;X_{\kappa+\delta})).
\end{equation}
In addition, if
\begin{equation}\label{e4.27}
0<k^\ast<\min\left\{a_\delta+\nu-\frac32,\vartheta\right\},
\end{equation}
then $v$ extends by $v(0)=0$ and belongs to
$C^{k^\ast}([0,T];L^p(\Omega;X_{\kappa+\delta}))$, and
\begin{equation}\label{e4.28}
\|v\|_{C^{k^\ast}([0,T];L^p(\Omega;X_{\kappa+\delta}))}
\leq C_T\left(1+\|u_0\|_{L^p(\Omega;X_\kappa)}\right).
\end{equation}
\end{theorem}
\begin{proof}
Let $U,V\in\mathcal H_{p,\kappa}^\nu(T)$. Then
\begin{align*}
(\mathcal TU)(t)-(\mathcal TV)(t)
={}&\int_0^t\mathcal P_\mu(t,s)
[F(s,U(s))-F(s,V(s))]ds\\
&+\int_0^t\mathcal P_\mu(t,s)
[G(s,U(s))-G(s,V(s))]dW_H(s)
=: H_2(t)+H_3(t).
\end{align*}
For $\zeta>0$, \eqref{e4.12} is equivalent to the following norm
\begin{equation*}
\|u\|_\zeta
:=\sup_{0<t\leq T}e^{-\zeta t}t^{1-\nu}
\|u(t)\|_{L^p(\Omega;X_\kappa)}.
\end{equation*}
It follows that
$\|\mathcal TU-\mathcal TV\|_\zeta
\leq\|H_2\|_\zeta+\|H_3\|_\zeta.$
Set $D=\|U-V\|_\zeta$, then we have
$$\|U(s)-V(s)\|_{L^p(\Omega;X_\kappa)}
\leq e^{\zeta s}s^{\nu-1}D.$$
Using \eqref{e4.2} and \eqref{e4.8}, for $0<t\leq T$, we obtain
\begin{align}\label{e4.29}
&e^{-\zeta t}t^{1-\nu}\|H_2(t)\|_{L^p(\Omega;X_\kappa)}\nonumber\\
\leq& CL_FD\,t^{1-\nu}
\int_0^t(t-s)^{a-1}e^{-\zeta(t-s)}s^{\nu-1}\,ds\nonumber\\
=&CL_FD \left(t^a\int_0^{1/2}
(1-r)^{a-1}e^{-\zeta t(1-r)}r^{\nu-1}\,dr+
t^a\int_{1/2}^1
(1-r)^{a-1}e^{-\zeta t(1-r)}r^{\nu-1}\,dr\right)\nonumber\\
\leq &CL_FD\left(
Ct^ae^{-\zeta t/2}
+2^{1-\nu}\int_0^{t/2}k^{a-1}e^{-\zeta k}\,dk\right)\leq  CL_F\zeta^{-a}D,
\end{align}
where $k=t(1-r)$. Hence, we get
$\|H_2\|_\zeta\leq C L_F\zeta^{-a}D.$
Similarly, for the stochastic term and for $0<t\leq T$, it follows from the BDG inequality that
\begin{align}\label{e4.30}
e^{-2\zeta t}t^{2-2\nu}
\|H_3(t)\|_{L^p(\Omega;X_\kappa)}^2\nonumber
\leq& CL_G^2D^2\,t^{2-2\nu}
\int_0^t(t-s)^{2a-2}e^{-2\zeta(t-s)}s^{2\nu-2}\,ds\nonumber\\
\leq & CL_G^2D^2\left(
Ct^{2a-1}e^{-\zeta t}
+2^{2-2\nu}\int_0^{t/2}k^{2a-2}e^{-2\zeta k}\,dk\right)\nonumber\\
\leq &CL_G^2\zeta^{-(2a-1)}D^2.
\end{align}
Then we deduce
$\|H_3\|_\zeta
\leq CL_G\zeta^{-(a-1/2)}D.$
It yields
$$\|\mathcal TU-\mathcal TV\|_\zeta
\leq C\bigl(\zeta^{-a}+\zeta^{-(a-1/2)}\bigr)\|U-V\|_\zeta.$$
Since $a>1/2$, we choose $\zeta$ so large that
$$q_\zeta:=C\bigl(\zeta^{-a}+\zeta^{-(a-1/2)}\bigr)<1.$$
Thus, $\mathcal T$ is a contraction on
$\mathcal H_{p,\kappa}^\nu(T)$ and it follows  from the Banach fixed point theorem that problem \eqref{e4.1} has a unique mild solution $u$.

According to $u=\mathcal Tu$,
\begin{align*}
\|u\|_\zeta\leq \|\mathcal Tu-\mathcal T0\|_\zeta+\|\mathcal T0\|_\zeta\leq q_\zeta\|u\|_\zeta
+C_T\left(1+\|u_0\|_{L^p(\Omega;X_\kappa)}\right).
\end{align*}
Hence, we deduce
$$\|u\|_\zeta
\leq \frac{C_T}{1-q_\zeta}
\left(1+\|u_0\|_{L^p(\Omega;X_\kappa)}\right).$$
Fix $\zeta>0$, we have $
\|u\|_\zeta
\leq \|u\|_{\mathcal H_{p,\kappa}^{\nu}}
\leq e^{\zeta T}\|u\|_\zeta.
$ for $0\leq t\leq T$.
Then we prove \eqref{e4.23}. It follows from \eqref{e4.19} and \eqref{e4.22} that \eqref{e4.24} holds.

It remains to prove the regularity results. Let $M:=\|u\|_{\mathcal H_{p,\kappa}^\nu}$. Choose $k^\ast<\eta<\min\{a_\delta-\tfrac12,\vartheta\}$.
For the deterministic and stochastic convolutions of \eqref{e4.13}, set $v=v_F+v_G$, where
$$v_F(t)
:=
\int_0^t
\mathcal P_\mu(t,s)F(s,u(s))\,ds,\quad v_G(t)
:=\int_0^t
\mathcal P_\mu(t,s)G(s,u(s))\,dW_H(s).$$
Let $\varepsilon\leq t_1<t_2\leq T$ and $h=t_2-t_1$. It follows from \eqref{e4.2}, \eqref{e4.7} and \eqref{e4.10} that
\begin{align*}
&\|v_F(t_2)-v_F(t_1)\|_{L^p(X_{\Omega;\kappa+\delta})}\\
\leq&\int_{t_1}^{t_2}
\| \mathcal P_\mu(t_2,s)F(s,u(s))\|
_{L^p(\Omega;X_{\kappa+\delta})}\,ds+
\int_0^{t_1}\|
[\mathcal P_\mu(t_2,s)-\mathcal P_\mu(t_1,s)]
F(s,u(s))
\|_{L^p(\Omega;X_{\kappa+\delta})}\,ds\\
\leq& C\int_{t_1}^{t_2}(t_2-s)^{a_\delta-1}
(1+s^{\nu-1}M)\,ds+C_\varepsilon h^\eta
\int_0^{t_1}(t_1-s)^{a_\delta-\eta-1}
(1+s^{\nu-1}M)\,ds
\leq  C_{\varepsilon,T}h^\eta(1+M).
\end{align*}
For $v_G$, by the BDG inequality, \eqref{e4.2}, \eqref{e4.7} and \eqref{e4.11}, we have
\begin{align*}
&\|v_G(t_2)-v_G(t_1)\|_{L^p(\Omega;X_{\kappa+\delta})}\\
\leq&C\left(
\int_{t_1}^{t_2}\|\mathcal P_\mu(t_2,s)G(s,u(s))\|_{L^p(\Omega;\gamma(H,X_{\kappa+\delta}))}^{2}\,ds
\right)^{1/2}\\
&+C\left(\int_0^{t_1}\|[\mathcal P_\mu(t_2,s)-\mathcal P_\mu(t_1,s)]G(s,u(s))\|_{L^p(\Omega;\gamma(H,X_{\kappa+\delta}))}^{2}\,ds
\right)^{1/2}\\
\leq&C\left(
\int_{t_1}^{t_2}
(t_2-s)^{2a_\delta-2}
(1+Ms^{\nu-1})^{2}\,ds
\right)^{1/2}+C_\varepsilon h^\eta
\left(
\int_0^{t_1}
(t_1-s)^{2a_\delta-2\eta-2}
(1+Ms^{\nu-1})^{2}\,ds
\right)^{1/2}\\
\leq&
C_{\varepsilon,T}h^\eta(1+M),
\end{align*}
where $a_\delta-\eta>1/2$. Hence, we deduce \eqref{e4.26}.

Now, we show the regularity on $[0,T]$.
By \eqref{e4.27}, for $0<k^\ast<\min\left\{a_\delta+\nu-\frac32,\vartheta\right\}$.
Set $ \rho:=a_\delta+\nu-\frac32$, then
 $k^\ast<\rho$. Since $\nu\leq1$,
$k^\ast<a_\delta-1/2$. Hence, choose
$k^\ast<\eta<\min\{a_\delta-\tfrac12,\vartheta\}$.
By \eqref{e4.2}, let $\alpha=\kappa+\delta$, for $a_\delta=\mu(1-\kappa-\delta)=a-\mu\delta$, we obtain
\begin{align}\label{e4.31}
\|\mathcal P_\mu(t,s)\|_{L(X_0,X_{\kappa+\delta})}
\le C(t-s)^{a_\delta-1}.
\end{align}
It follows from \eqref{e4.10}, \eqref{e4.11} and \eqref{e4.31} that
\begin{align*}
\|v_F(t)\|_{L^p(\Omega;X_{\kappa+\delta})}
&\le
\int_0^t
\|\mathcal P_\mu(t,s)F(s,u(s))\|
 _{L^p(\Omega;X_{\kappa+\delta})}\,ds\\
&\le C\int_0^t(t-s)^{a_\delta-1}
(1+Ms^{\nu-1})\,ds
\le C\left(t^{a_\delta}+Mt^{a_\delta+\nu-1}
\right)
\end{align*}
and
\begin{align*}
\|v_G(t)\|_{L^p(\Omega;X_{\kappa+\delta})}
&=\left\|\int_0^t\mathcal P_\mu(t,s)G(s,u(s))\,dW_H(s)
\right\|_{L^p(\Omega;X_{\kappa+\delta})}\\
&\le C\left(\int_0^t\|\mathcal P_\mu(t,s)G(s,u(s))\|
_{L^p(\Omega;\gamma(H,X_{\kappa+\delta}))}^{2}
\,ds\right)^{1/2}\\
&\le C\left(\int_0^t(t-s)^{2a_\delta-2}
(1+Ms^{\nu-1})^2\,ds\right)^{1/2}
\le C\left(t^{a_\delta-\frac12}
+Mt^{a_\delta+\nu-\frac32}\right).
\end{align*}
Since $\rho\leq a_\delta-1/2$ and
$\rho<a_\delta+\nu-1$, for $0<t\leq T$, it follows that
\begin{equation}\label{e4.32}
\|v(t)\|_{L^p(X_{\Omega;\kappa+\delta})}
\leq C_T(1+M)t^\rho.
\end{equation}
In particular, $v$ extends continuously to $t=0$ by  $v(0)=0$.

We now prove the H\"{o}lder estimate on $[0,T]$. Let $0\le t_1<t_2\le T$ and set $h=t_2-t_1$. We distinguish two cases. If $t_1=0$ or $h\geq t_1$, then
$t_2\leq2h$, and \eqref{e4.32} yields
\begin{align}\label{e4.33}
\|v(t_2)-v(t_1)\|_{L^p(X_{\Omega;\kappa+\delta})}
\nonumber\le&\|v(t_2)\|_{L^p(\Omega;X_{\kappa+\delta})}
+\|v(t_1)\|_{L^p(\Omega;X_{\kappa+\delta})}\\
\leq& C_T(1+M)(t_2^\rho+t_1^\rho)
\leq C_T(1+M)h^\rho\leq C_T(1+M)h^{k^\ast}.
\end{align}
For $0<h<t_1$, we have
\begin{align*}
v_F(t_2)-v_F(t_1)
={}&
\int_{t_1}^{t_2}
\mathcal P_\mu(t_2,s)F(s,u(s))\,ds+
\int_0^{t_1}
[\mathcal P_\mu(t_2,s)-\mathcal P_\mu(t_1,s)]
F(s,u(s))\,ds\\
=:{}&H_1+H_2.
\end{align*}
From \eqref{e4.10}, we have
\begin{align*}
\|F(s,\cdot,u(s,\cdot))\|_{L^p(\Omega;X_0)}
\le C_F\left(1+\|u(s)\|_{L^p(\Omega;X_\kappa)}\right)\le C_F\left(1+Ms^{\nu-1}\right).
\end{align*}
By \eqref{e4.31}, we get
\begin{align*}
\|J_1\|_{L^p(\Omega;X_{\kappa+\delta})}
\le
C(1+Mt_1^{\nu-1})
\int_{t_1}^{t_2}(t_2-s)^{a_\delta-1}\,ds\le
Ch^{a_\delta}(1+Mt_1^{\nu-1}).
\end{align*}
According to \eqref{e4.7}, then
\begin{align*}
\|\mathcal P_\mu(t_2,s)-\mathcal P_\mu(t_1,s)\|
 _{L(X_0,X_{\kappa+\delta})}\le
C_\eta h^\eta
(t_1-s)^{a_\delta-\eta-1}.
\end{align*}
Hence, we deduce
\begin{align*}
\|J_2\|_{L^p(\Omega;X_{\kappa+\delta})}
&\le
C_\eta h^\eta
\int_0^{t_1}
(t_1-s)^{a_\delta-\eta-1}
(1+Ms^{\nu-1})\,ds\\
&=
C_\eta h^\eta
\int_0^{t_1}
(t_1-s)^{a_\delta-\eta-1}\,ds+
C_\eta Mh^\eta
\int_0^{t_1}
(t_1-s)^{a_\delta-\eta-1}s^{\nu-1}\,ds\\
&\le
C_\eta h^\eta
\left(
t_1^{a_\delta-\eta}+
Mt_1^{a_\delta-\eta+\nu-1}
\right).
\end{align*}
Therefore, we deduce
\begin{align}\label{e4.34}
\|v_F(t_2)-v_F(t_1)\|_{L^p(\Omega;X_{\kappa+\delta})}
\le Ch^{a_\delta}(1+Mt_1^{\nu-1})+
C_\eta h^\eta\left(
t_1^{a_\delta-\eta}+
Mt_1^{a_\delta-\eta+\nu-1}
\right).
\end{align}
Since $h<t_1$, $\eta>k^\ast$, and
$a_\delta+\nu-1>k^\ast$, combining with \eqref{e4.34}, we deduce
\begin{align}\label{e4.35}
\nonumber\|v_F(t_2)-v_F(t_1)\|_{L^p(\Omega;X_{\kappa+\delta})}\le &C h^{a_\delta}+CMh^{a_\delta}t_1^{\nu-1}
+C_\eta h^\eta t_1^{a_\delta-\eta}
+C_\eta Mh^\eta t_1^{a_\delta-\eta+\nu-1}\\
\le& C_T(1+M)h^{k^\ast}.
\end{align}
Similarly, we have
\begin{align*}
v_G(t_2)-v_G(t_1)
=&\int_{t_1}^{t_2}
\mathcal P_\mu(t_2,s)G(s,u(s))\,dW_H(s)+\int_0^{t_1}\bigl[\mathcal P_\mu(t_2,s)-\mathcal P_\mu(t_1,s)
\bigr]G(s,u(s))\,dW_H(s)\\
=:&K_1+K_2.
\end{align*}
From \eqref{e4.11}, we get
\begin{align*}
\|G(s,u(s))\|_{L^p(\Omega;\gamma(H,X_0))}
\le C_G\left(
1+\|u(s)\|_{L^p(\Omega;X_\kappa)}
\right)\le C_G\left(1+Ms^{\nu-1}\right).
\end{align*}
Then we have
\begin{align*}
\|K_1\|_{L^p(\Omega;X_{\kappa+\delta})}
&\le C\left(\int_{t_1}^{t_2}
\|\mathcal P_\mu(t_2,s)G(s,u(s))\|
_{L^p(\Omega;\gamma(H,X_{\kappa+\delta}))}^{2}\,ds
\right)^{1/2}\\
&\le C\left(\int_{t_1}^{t_2}(t_2-s)^{2a_\delta-2}
(1+Ms^{\nu-1})^2\,ds\right)^{1/2}\\
&\le C(1+Mt_1^{\nu-1})\left(
\int_{t_1}^{t_2}(t_2-s)^{2a_\delta-2}\,ds
\right)^{1/2}\le
C h^{a_\delta-\frac12}
(1+Mt_1^{\nu-1}),
\end{align*}
and
\begin{align*}
\|K_2\|_{L^p(\Omega;X_{\kappa+\delta})}
\le& C\left(\int_0^{t_1}\left\|
[\mathcal P_\mu(t_2,s)-\mathcal P_\mu(t_1,s)]
G(s,u(s))\right\|_{L^p(\Omega;\gamma(H,X_{\kappa+\delta}))}^{2}\,ds\right)^{1/2}\\
\le& C_\eta h^\eta\left(\int_0^{t_1}
(t_1-s)^{2a_\delta-2\eta-2}
(1+Ms^{\nu-1})^2\,ds\right)^{1/2}\\
\le& C_\eta h^\eta\left(
\int_0^{t_1}(t_1-s)^{2a_\delta-2\eta-2}\,ds
\right)^{1/2}+C_\eta Mh^\eta
\left(\int_0^{t_1}
(t_1-s)^{2a_\delta-2\eta-2}s^{2\nu-2}\,ds
\right)^{1/2}\\
\le& C_\eta h^\eta
\left(t_1^{a_\delta-\eta-\frac12}
+Mt_1^{a_\delta-\eta+\nu-\frac32}
\right).
\end{align*}
Hence, for $h<t_1$, $\eta>k^\ast$, $a_\delta-1/2>k^\ast$ and $\rho>k^\ast$, we obtain
\begin{align}\label{e4.36}
\nonumber\|v_G(t_2)-v_G(t_1)\|_{L^p(\Omega;X_{\kappa+\delta})}\le& Ch^{a_\delta-\frac12}
+CMh^{a_\delta-\frac12}t_1^{\nu-1}
+C_\eta h^\eta t_1^{a_\delta-\eta-\frac12}
+C_\eta Mh^\eta t_1^{a_\delta-\eta+\nu-\frac32}\\
\le&C_T h^{k^\ast}+C_TMh^{k^\ast}+
C_T h^{k^\ast}+C_TMh^{k^\ast}
\le C_T(1+M)h^{k^\ast}.
\end{align}
It follows from \eqref{e4.33}, \eqref{e4.35} and \eqref{e4.36} that
$$v\in C^{k^\ast}([0,T];L^p(\Omega;X_{\kappa+\delta})).$$
Hence, according to \eqref{e4.23}, we deduce \eqref{e4.28}. The proof is complete.
\end{proof}

\begin{remark}
The decomposition
$
u(t)=H_{\mu,\nu}(t,0)u_0+v(t)
$
separates the initial singularity from the regular part of the solution. For $\nu<1$, the initial component is singular as $t\to 0$. However, the regular part satisfies $t^{1-\nu}v(t)\to 0$ in $L^p(\Omega;X_\kappa)$. Moreover, by the additional condition in \eqref{e4.27}, $v$ extends
H\"{o}lder continuously to $t=0$ with $v(0)=0$.
\end{remark}

\begin{lemma}\label{lem4.4}
Let Assumptions~\ref{assump2.1}, \ref{assump2.2}, ~\ref{assump4.1}, and \eqref{e4.14} hold. For $p\in[2,\infty)$,
let $u_0,v_0:\Omega\to X_\kappa$ be strongly
$\mathcal F_0$-measurable. Fix $\Lambda\in\mathcal F_0$, let $u$ and $v$ be progressively measurable processes satisfying \eqref{e4.13} almost surely on $\Lambda$. If $\mathbf1_\Lambda u,\ \mathbf1_\Lambda v
\in\mathcal H_{p,\kappa}^{\nu}(T)$ and $u_0=v_0$ almost surely on $\Lambda$, then for every $t\in(0,T]$,
$$u(t)=v(t)$$
almost surely on $\Lambda$.
\end{lemma}
\begin{proof}
According to the  method in \cite[Lemma~6.4]{Veraar}, let $z:=\mathbf1_\Lambda(u-v)$, where
$$\mathbf1_\Lambda(\omega)
=
\begin{cases}
1, & \omega\in\Lambda,\\
0, & \omega\notin\Lambda.
\end{cases}$$
Since $\Lambda\in\mathcal F_0$ and $u_0=v_0$ almost surely on $\Lambda$, we get $\mathbf1_\Lambda H_{\mu,\nu}(t,0)(u_0-v_0)=0$. According to \eqref{e4.13}, for $\Lambda\in\mathcal F_0$, we have
\begin{align*}
z(t)=&\int_0^t\mathcal P_\mu(t,s)\mathbf1_\Lambda
[F(s,u(s))-F(s,v(s))]ds+\int_0^t\mathcal P_\mu(t,s)\mathbf1_\Lambda
[G(s,u(s))-G(s,v(s))]dW_H(s)\\
=:&z_F(t)+z_G(t).
\end{align*}
By \eqref{e4.8}, we obtain
\begin{align*}
\left\|\mathbf1_\Lambda
[F(s,u(s))-F(s,v(s))]\right\|_{L^p(\Omega;X_0)}\leq L_F
\left\|\mathbf1_\Lambda(u(s)-v(s))
\right\|_{L^p(\Omega;X_\kappa)}\leq L_F\|z(s)\|_{L^p(\Omega;X_\kappa)}.
\end{align*}
Similarly, it follows from \eqref{e4.9} that
$$\|\mathbf1_\Lambda[G(s,u(s))-G(s,v(s))]\|_
{L^p(\Omega;\gamma(H,X_0))}
\leq L_G\|z(s)\|_{L^p(\Omega;X_\kappa)}.$$
Combining \eqref{e4.29} and \eqref{e4.30}, we deduce
\begin{align*}
\|z_F\|_\zeta \leq  e^{-\zeta t}t^{1-\nu}
\|z_F(t)\|_{L^p(\Omega;X_\kappa)}\le
CL_F\zeta^{-a}\|z\|_\zeta,
\end{align*}
\begin{align*}
\|z_G\|_\zeta \leq  e^{-\zeta t}t^{1-\nu}
\|z_G(t)\|_{L^p(\Omega;X_\kappa)}
\le CL_G\zeta^{-(a-\frac12)}
\|z\|_\zeta.
\end{align*}
Hence, we have
$$\|z\|_\zeta
\leq C\left(L_F\zeta^{-a}
+L_G\zeta^{-(a-1/2)}\right)\|z\|_\zeta.$$
Choose $\zeta>0$ sufficiently large that
$C_\zeta:=C\left(L_F\zeta^{-a}
+L_G\zeta^{-(a-1/2)}\right)<1$. It follows that
$$\|z\|_\zeta\leq C_\zeta\|z\|_\zeta,$$
then $\|z\|_\zeta=0$. Thus, we have
$\mathbf1_\Lambda u=\mathbf1_\Lambda v$
in $\mathcal H_{p,\kappa}^{\nu}(T)$. Then for fixed $t\in(0,T]$,
$u(t)=v(t)$ almost surely on $\Lambda$. The proof is complete.
\end{proof}

For a strongly $\mathcal F_0$-measurable initial value $u_0$, set
$\Omega_n:=\{\|u_0\|_{X_\kappa}\leq n\}$. In the following theorem, a
mild solution means a progressively measurable process $u$ such that
$\mathbf1_{\Omega_n}u\in\mathcal H_{q,\kappa}^{\nu}(T)$ for $2\leq q<\infty$,
and the localized form of \eqref{e4.13} holds almost surely on
$\Omega_n$ for $n\in\mathbb N$ and $t\in(0,T]$.

\begin{theorem}\label{the4.2}
Let Assumptions~\ref{assump2.1}, \ref{assump2.2}, ~\ref{assump4.1}, and \eqref{e4.14} hold. Let $u_0:\Omega\to X_\kappa$ be strongly $\mathcal F_0$-measurable.
Then the following results hold.
\begin{itemize}
\item[\rm(i)] Problem \eqref{e4.1} admits  a unique progressively measurable mild solution $u$ on $(0,T]$.
\item[\rm(ii)] Let $\delta\in[0,1-\kappa)$ satisfy
$\mu(\kappa+\delta)\leq\vartheta$, and set $v(t):=u(t)-H_{\mu,\nu}(t,0)u_0$. For
$$0<k^\ast<\min\left\{\mu(1-\kappa-\delta)-\frac12,\vartheta\right\},$$
$v$ has a version such that almost surely,
$v\in C_{\mathrm{loc}}^{k^\ast}((0,T];X_{\kappa+\delta})$.
Moreover, for
$$0<k^\ast<\min\left\{\mu(1-\kappa-\delta)+\nu-\frac32,
\vartheta\right\},$$
$v$ admits a version such that almost surely,
$v\in C^{k^\ast}([0,T];X_{\kappa+\delta})$
and $v(0)=0$.
\end{itemize}
\end{theorem}
\begin{proof}
Set $\Omega_n:=\{\|u_0\|_{X_\kappa}\leq n\}$ for $n\in\mathbb N$. Let
$$\Omega_0:=\varnothing,\quad D_n:=\Omega_n\setminus\Omega_{n-1},\quad n\geq1,$$
and define the initial values $u_{0,n}:=\mathbf1_{\Omega_n}u_0$.
For $\omega\in\Omega$,
$\|u_{0,n}(\omega)\|_{X_\kappa}
=\mathbf1_{\Omega_n}(\omega)
\|u_0(\omega)\|_{X_\kappa}
\leq n$.
Then for finite $q\geq2$,
\begin{align*}
\|u_{0,n}\|_{L^q(\Omega;X_\kappa)}^q
=\int_{\Omega_n}\|u_0(\omega)\|_{X_\kappa}^q
\,d\mathbb P(\omega)\leq n^q\mathbb P(\Omega_n)
\leq n^q.
\end{align*}
Thus, we obtain
$$u_{0,n}\in L^q(\Omega;X_\kappa),
\quad
\|u_{0,n}\|_{L^q(\Omega;X_\kappa)}\leq n.$$
For each $q$, it follows from Theorem~\ref{the4.1} that there exists a mild solution
$u_n^{(q)}\in\mathcal H_{q,\kappa}^{\nu}(T)$.
Since $(\Omega,\mathcal F,\mathbb P)$ is a probability space, for every finite $q\ge2$,
we get
\begin{align*}
\|x\|_{L^2(\Omega;X_\kappa)}^2
=\int_\Omega\|x(\omega)\|_{X_\kappa}^2\,d\mathbb P(\omega)\le
\left(
\int_\Omega\|x(\omega)\|_{X_\kappa}^q\,d\mathbb P(\omega)
\right)^{2/q}
\mathbb P(\Omega)^{1-2/q}=\|x\|_{L^q(\Omega;X_\kappa)}^2.
\end{align*}
Hence,$$\|x\|_{L^2(\Omega;X_\kappa)}
\le
\|x\|_{L^q(\Omega;X_\kappa)}.$$
For $u\in\mathcal H_{q,\kappa}^{\nu}(T)$, we obtain
\begin{align*}
\|u\|_{\mathcal H_{2,\kappa}^{\nu}(T)}
=\sup_{0<t\le T}
t^{1-\nu}
\|u(t)\|_{L^2(\Omega;X_\kappa)}
\le\sup_{0<t\le T}
t^{1-\nu}
\|u(t)\|_{L^q(\Omega;X_\kappa)}=
\|u\|_{\mathcal H_{q,\kappa}^{\nu}(T)}.
\end{align*}
It yields
$$\mathcal H_{q,\kappa}^{\nu}(T)
\hookrightarrow\mathcal H_{2,\kappa}^{\nu}(T),
\quad
\|w\|_{\mathcal H_{2,\kappa}^{\nu}(T)}
\leq\|w\|_{\mathcal H_{q,\kappa}^{\nu}(T)}.$$
Thus $u_n^{(q)}$ and $u_n^{(2)}$ both belong to
$\mathcal H_{2,\kappa}^{\nu}(T)$ and have the same initial value
$u_{0,n}$. By the uniqueness of Theorem~\ref{the4.1}, for $p=2$, we deduce
$u_n^{(q)}=u_n^{(2)}$.
We denote this common mild solution by $u_n$. Then
$u_n\in\mathcal H_{q,\kappa}^{\nu}(T)$ for every finite $q\geq2$.
If $m\leq n$, then $u_{0,m}=u_{0,n}=u_0$ on $\Omega_m$. For $\Lambda=\Omega_m$ and for every $t\in(0,T]$, it follows from
Lemma~\ref{lem4.4} that
$$u_m(t)=u_n(t)$$
almost surely on $\Omega_m$.
Since the sets $D_n$ are disjoint and belong to $\mathcal F_0$, we define
$$u(t,\omega):=\sum_{n=1}^{\infty}
\mathbf1_{D_n}(\omega)u_n(t,\omega).$$
The sets $D_n$ are pairwise disjoint and their union equals $\Omega$ up to a null set. Hence, for almost every $\omega$, there is a unique
$n(\omega)\in\mathbb N$ such that $\omega\in D_{n(\omega)}$, and $u(t,\omega)=u_{n(\omega)}(t,\omega)$. Moreover, since $D_n\in\mathcal F_0$ and $u_n$ is progressively measurable, each process $\mathbf1_{D_n}u_n$ is progressively measurable. Hence, $u$ is progressively measurable. For every $n$, we also have $u=u_n$ almost surely on
$\Omega_n$. Since $\Omega_n\in\mathcal F_0$ and $u=u_n$ almost surely on
$\Omega_n$, the deterministic and stochastic integrals in
\eqref{e4.13} are the same for $u$ and $u_n$ on $\Omega_n$. Thus, $\mathbf1_{\Omega_n}u\in\mathcal H_{q,\kappa}^{\nu}(T)$ for $q\geq2$ and the localized form of \eqref{e4.13} holds almost surely on $\Omega_n$.

Then we prove uniqueness. Let $u$ and $\widetilde u$ be two progressively measurable mild solutions satisfying the localized formulation. For
each $n\in\mathbb N$ and any finite $q\geq2$, we have
$$\mathbf1_{\Omega_n}u,\ \mathbf1_{\Omega_n}\widetilde u
\in\mathcal H_{q,\kappa}^{\nu}(T),$$
and both mild solutions satisfy \eqref{e4.13} almost surely on $\Omega_n$ with the same initial value.
Combining with Lemma~\ref{lem4.4} for
$\Lambda=\Omega_n$, we obtain
$$\mathbf1_{\Omega_n}u=\mathbf1_{\Omega_n}\widetilde u
\quad\text{in }\mathcal H_{q,\kappa}^{\nu}(T).$$
Since $\Omega_n\to\Omega$, for every
$t\in(0,T]$, it follows that $u(t)=\widetilde u(t)$ almost surely.

We next prove the path continuity. For $\varepsilon\in(0,T)$, fix
$0<k^\ast<\eta<\min\left\{\mu(1-\kappa-\delta)-\frac12,
\vartheta\right\}$.
For finite $q\geq2$, let $v_n:=u_n-H_{\mu,\nu}(\cdot,0)u_{0,n}$. By applying
Theorem~\ref{the4.1} to the regular part, we get
$$v_n\in C^\eta\bigl([\varepsilon,T];
L^q(\Omega;X_{\kappa+\delta})\bigr).$$
Then there is a constant $C_{n,q,\varepsilon,\eta}>0$ such that
$$\|v_n(t_2)-v_n(t_1)\|_{L^q(\Omega;X_{\kappa+\delta})}
\leq C_{n,q,\varepsilon,\eta}|t_2-t_1|^\eta.$$
for $\varepsilon\leq t_1<t_2\leq T$. Then we get
\begin{align*}
\mathbb E\|v_n(t_2)-v_n(t_1)\|_{X_{\kappa+\delta}}^q
=\|v_n(t_2)-v_n(t_1)\|_{L^q(\Omega;X_{\kappa+\delta})}^q\leq C_{n,q,\varepsilon,\eta}^q
|t_2-t_1|^{q\eta}.
\end{align*}
Choose $q$ sufficiently large that $q(\eta-k^\ast)>1$. It follows from Kolmogorov's
continuity theorem that $v_n$ has a $k^\ast$-H\"older continuous version on $[\varepsilon,T]$. Let $\varepsilon_j=T/(j+1)$, then $0<\varepsilon_j<T$ as $\varepsilon_j\to 0$ and $\bigcup_{j=1}^{\infty}
[\varepsilon_j,T]=(0,T]$. By choosing the continuous versions consistently on the intervals $[\varepsilon_j,T]$, we obtain a locally $k^\ast$-H\"{o}lder continuous version of $v_n$ on $(0,T]$. Therefore, defining this version on the disjoint sets $D_n$, we deduce $v\in C_{\mathrm{loc}}^{k^\ast}((0,T];X_{\kappa+\delta})$ almost surely.

Finally, fix
$0<k^\ast<\eta<\min\left\{
\mu(1-\kappa-\delta)+\nu-\frac32,\vartheta\right\}$.
For each $n$, choose a finite $q\geq2$ sufficiently large that $q(\eta-k^\ast)>1$. Since $u_{0,n}\in L^q(\Omega;X_\kappa)$, for $ 0\leq t_1<t_2\leq T$, it follows from Theorem~\ref{the4.1} that
$$\mathbb E\|v_n(t_2)-v_n(t_1)\|_{X_{\kappa+\delta}}^q
\leq C_{n,q}|t_2-t_1|^{q\eta}.$$
Hence, by Kolmogorov's continuity theorem, $v_n$ has a version in $C^{k^\ast}([0,T];X_{\kappa+\delta})$ almost surely. By
Lemma~\ref{lem4.4}, these versions are compatible on the sets $\Omega_n$ for $m\leq n$. On the disjoint sets $D_n$, we obtain $v\in C^{k^\ast}([0,T];X_{\kappa+\delta})$ almost surely. The proof is complete.
\end{proof}

\section{Stochastic maximal regularity}\label{sec5}
In this section, we consider stochastic maximal $L^p$-regularity for problem \eqref{e1.1}. Let $H$ be a separable Hilbert space and $X_0$ be a UMD
space with type $2$. We assume that
$\mu\in(1/2,1)$ and $p>2$.  The restriction $\mu>1/2$ is forced by the local square integrability of the kernel of the stochastic convolution. Let $A_0:=A(0)$.

\begin{assumption}\label{assump5.1}
Let Assumptions~\ref{assump2.1} and~\ref{assump2.2} hold. The family $A:[0,T]\to L(X_1,X_0)$ is strongly  measurable, and
$$ \sup_{t\in[0,T]}
 \|A(t)\|_{L(X_1,X_0)}<\infty.$$
\end{assumption}

\subsection{The deterministic case}
For a fixed $\omega\in\Omega$, we consider the deterministic problem
\begin{equation}\label{e5.1}
 \left\{
\begin{aligned}
&D_{0+}^{\lambda,\mu}u(t)+A(t)u(t)=f(t), &&  t \in (0, T), \\
&I_{0+}^{(1-\lambda)(1-\mu)}u(0)=u_0.
\end{aligned}\right.
\end{equation}
Let $f\in L^1(0,T;X_0)$ and $u_0\in X_0$. A strongly measurable function
$u:(0,T)\to X_0$ is called an integrated strong solution of \eqref{e5.1} if
$u \in L^1(0,T; X_1)$, $A(\cdot)u(\cdot)\in L^1(0,T;X_0)$ and
\begin{align}\label{e5.2}
u(t)-g_\nu(t)u_0
+\int_0^t g_\mu(t-s)A(s)u(s)\,ds
=
\int_0^t g_\mu(t-s)f(s)\,ds,
\end{align}
holds for almost every $t\in(0,T)$, where $g_\mu(t) = t^{\mu-1}/\Gamma(\mu)$ and $1-\nu = (1-\lambda)(1-\mu)$.

\begin{definition}\label{def5.1}
We say that the non-autonomous operator $A(\cdot) \in \text{DMR}(p, \mu, T)$ if for $f \in L^p(0,T; X_0)$, the problem
\begin{equation}\label{e5.3}
 \left\{
\begin{aligned}
&D_{0+}^{\lambda,\mu}u(t)+A(t)u(t)=f(t), &&  t \in (0, T), \\
&I_{0+}^{(1-\lambda)(1-\mu)}u(0)=0.
\end{aligned}\right.
\end{equation}
admits a unique integrated strong solution satisfying $u\in L^p(0,T; X_1)$, $D_{0+}^{\lambda,\mu}u \in L^p(0,T; X_0)$
and
\begin{align}\label{e5.4}
\|D_{0+}^{\lambda,\mu}u\|_{L^p(0,T;X_0)}
+\|u\|_{L^p(0,T;X_1)}
\leq C
\|f\|_{L^p(0,T;X_0)},
\end{align}
where the constant $C > 0$ is independent of $f$.
\end{definition}

\subsection{Deterministic mixed derivative regularity}
Let $H^{\mu,p}(0,T;X_0)$ denote the $X_0$-valued Bessel potential space. Since $p>2$ and $\mu>1/2$, one has $\mu>1/p$, and the trace operator
$$\text{tr}_0:H^{\mu,p}(0,T;X_0)\to X_0$$
is well defined and bounded. We set
$${}_0H^{\mu,p}(0,T;X_0)
=
\left\{
v\in H^{\mu,p}(0,T;X_0):
\text{tr}_0 v=0
\right\}.$$

\begin{lemma}\label{lem5.1}
Let $v$ have zero  initial value $I_{0+}^{(1-\lambda)(1-\mu)}v(0)=0$ and suppose $D_{0+}^{\lambda,\mu}v \in L^p(0,T;X_0)$ and $v \in L^p(0,T;X_1)$.
Then for $\theta \in [0, \mu]$, we have $v \in H^{\theta,p}(0,T;X_{1-\theta/\mu})$ and there exists a constant $C_\theta>0$ such that
\begin{align}\label{e5.5}
\|v\|_{H^{\theta,p}(0,T;X_{1-\theta/\mu})}
\leq
C_\theta
\left(
\|D_{0+}^{\lambda,\mu}v\|_{L^p(0,T;X_0)}
+
\|v\|_{L^p(0,T;X_1)}
\right),
\end{align}
where the constant $C_\theta$ is independent of $v$.
Moreover, if $v$ satisfies problem \eqref{e5.3}, then
\begin{align}\label{e5.6}
\|v\|_{H^{\theta,p}(0,T;X_{1-\theta/\mu})}
\leq
C_{\theta}\|f\|_{L^p(0,T;X_0)}.
\end{align}
\end{lemma}
\begin{proof}
Noting that
$$I_{0+}^{\mu}D_{0+}^{\lambda,\mu}v(t)
=v(t)-g_\nu(t)I_{0+}^{(1-\lambda)(1-\mu)}v(0).$$
It follows from $I_{0+}^{(1-\lambda)(1-\mu)}v(0)=0$ that $v(t) = I_{0+}^\mu D_{0+}^{\lambda,\mu}v(t)$. By Lemma \ref{lem2.5}, the operator $I_{0+}^{\mu}: L^p(0,T;X_0) \to {}_0H^{\mu,p}(0,T;X_0)$ is bounded. Then $v \in {}_0H^{\mu,p}(0,T;X_0)$ and
\begin{align}\label{e5.7}
\|v\|_{H^{\mu,p}(0,T;X_0)} \leq C \|D_{0+}^{\lambda,\mu}v\|_{L^p(0,T;X_0)}.
\end{align}
For $0 < \theta < \mu$, set $\eta = \theta/\mu$. It follows from \cite[Proposition 2.4]{Portal} that
$$\left[
L^p(0,T;X_1),
H^{\mu,p}(0,T;X_0)
\right]_{\eta}
=
H^{\theta,p}
\left(
0,T;[X_1,X_0]_{\eta}
\right),$$
where $L^p(0,T;X_1)=H^{0,p}(0,T;X_1)$.
By the symmetry of complex interpolation,
$$[X_1,X_0]_{\eta}
=[X_0,X_1]_{1-\eta}
=X_{1-\theta/\mu},$$
we have
$$H^{\mu,p}(0,T;X_0)
\cap L^p(0,T;X_1)
\hookrightarrow
H^{\theta,p}
\left(
0,T;X_{1-\theta/\mu}
\right).$$
Then $v \in H^{\theta,p}\left(0,T; X_{1-\theta/\mu}\right)$ and
$$\|v\|_{H^{\theta,p}(0,T;X_{1-\theta/\mu})} \leq C_\theta \left( \|v\|_{H^{\mu,p}(0,T;X_0)} + \|v\|_{L^p(0,T;X_1)} \right).$$
Combining with \eqref{e5.7}, we deduce \eqref{e5.5}. For $\theta=0$ and $\theta=\mu$, we get the same result. It follows from \eqref{e5.4} that \eqref{e5.6} holds. The proof is complete.
\end{proof}
\begin{lemma}\label{lem5.2}
Assume $A(\cdot) \in \mathrm{DMR}(p, \mu, T)$. For progressively measurable $h \in L^p(\Omega; L^p(0, T; X_0))$, the zero initial value problem
$${}D_{0+}^{\lambda,\mu} v + A(\cdot)v = h$$
admits a unique progressively measurable integrated strong solution.

Moreover, for $\theta \in [0, \mu]$, we have
\begin{align*}
\|v\|_{L^p(\Omega; H^{\theta,p}(0,T; X_{1-\theta/\mu}))} \leq C_\theta \|h\|_{L^p(\Omega; L^p(0,T; X_0))}.
\end{align*}
\end{lemma}
\begin{proof}
Let $\mathcal R_D$ denote the solution operator associated with \eqref{e5.3}. For $U\in L^p(0,T;X_0)$, set $y=\mathcal R_DU$. By \eqref{e5.4} and Lemma~\ref{lem5.1}, we deduce
\begin{align*}
\|\mathcal R_DU\|_{H^{\theta,p}(0,T;X_{1-\theta/\mu})}
&\leq C_\theta\left(
\|D_{0+}^{\lambda,\mu}y\|_{L^p(0,T;X_0)}
+\|y\|_{L^p(0,T;X_1)}
\right)\leq C_\theta\|U\|_{L^p(0,T;X_0)}.
\end{align*}
Thus, $\mathcal R_D$ is a bounded linear operator from
$L^p(0,T;X_0)$ into $H^{\theta,p}(0,T;X_{1-\theta/\mu})$.

For almost every $\omega\in\Omega$, define $v(\cdot,\omega):=\mathcal R_D(h(\cdot,\omega))$. Then
$$D_{0+}^{\lambda,\mu}v(\cdot,\omega)
+A(\cdot)v(\cdot,\omega)=h(\cdot,\omega),
\quad I_{0+}^{1-\nu}v(0,\omega)=0.$$
Moreover, we deduce
\begin{align*}
\|v\|_{L^p(\Omega;
H^{\theta,p}(0,T;X_{1-\theta/\mu}))}^p&=\int_\Omega
\|\mathcal R_D(h(\cdot,\omega))\|_{H^{\theta,p}
(0,T;X_{1-\theta/\mu})}^p\,d\mathbb P(\omega)\\
&\leq C_\theta^p\int_\Omega
\|h(\cdot,\omega)\|_{L^p(0,T;X_0)}^p\,d\mathbb P(\omega)=C_\theta^p
\|h\|_{L^p(\Omega;L^p(0,T;X_0))}^p.
\end{align*}
It follows that
$$\|v\|_{L^p(\Omega; H^{\theta,p}(0,T; X_{1-\theta/\mu}))} \leq C_\theta \|h\|_{L^p(\Omega; L^p(0,T; X_0))}.$$

Next, we prove the progressive measurability of $v$. According to the uniqueness of problem \eqref{e5.3}, the solution on $(0,t)$ depends on the restriction of the forcing term to $(0,t)$. Hence, for the solution operator $\mathcal R_t$ on $(0,t)$, we deduce
$$(\mathcal R_TU)|_{(0,t)}
=\mathcal R_t(U|_{(0,t)}).$$
Since $h$ is progressively measurable and $h \in L^p(\Omega;L^p(0,T;X_0))$, there is a sequence of $X_0$-valued progressively measurable simple processes $(h_n)_{n\geq1}$ such that $h_n\to h$ in $L^p(\Omega;L^p(0,T;X_0))$. Set $v_n:=\mathcal R_Dh_n$, for $t\in(0,T]$, the restriction
$h_n|_{(0,t)}$ is $\mathcal F_t$-measurable. Combining with the boundedness of $\mathcal R_t$, we get
$$
v_n|_{(0,t)}
=\mathcal R_t(h_n|_{(0,t)}),
$$
is $\mathcal F_t$-measurable. Moreover, since
$v_n\in H^{\mu,p}(0,T;X_0)$ and $\mu>1/p$, $v_n$ has an $X_0$-valued continuous version and is progressively measurable. Then the version is adapted and progressively measurable. Furthermore, by \eqref{e5.4}, we deduce
$$\|v_n-v\|_{L^p(\Omega;L^p(0,T;X_1))}
\leq C\|h_n-h\|_{L^p(\Omega;L^p(0,T;X_0))}
\to 0.$$
Since the space of progressively measurable processes is closed in
$L^p(\Omega;L^p(0,T;X_1))$, the process $v$ admits a progressively measurable version.

Finally, if $v_1$ and $v_2$ are two such solutions, then $z:=v_1-v_2$
satisfies
$$D_{0+}^{\lambda,\mu}z+A(\cdot)z=0,
\quad I_{0+}^{1-\nu}z(0)=0.$$
According to the uniqueness in $A(\cdot)\in\mathrm{DMR}(p,\mu,T)$, we deduce $z=0$. The proof is complete.
\end{proof}

\subsection{Regularity of the autonomous stochastic convolution}
Let $(\Omega,\mathcal F,\{\mathcal F_t\}_{t\in[0,T]},\mathbb P)$ be a
complete filtered probability space and $W_H$ be a cylindrical Brownian
motion on a separable Hilbert space $H$. Denote by
$L^p(\Omega;L^p_{\mathcal F}(0,T;X_0))$ the space of all progressively
measurable processes in $L^p(\Omega;L^p(0,T;X_0))$.

For $\alpha \in (1/(2\mu), 1]$, we write $A_0^\alpha$ for $(A_0)^\alpha$. For a progressively measurable process $G: (0,T) \times \Omega \to \gamma(H, X_\alpha)$, we define the space $\mathbb{G}_\alpha$ as
$$\mathbb{G}_\alpha := L^p\Big(\Omega; L^p_{\mathcal{F}}\big(0,T; \gamma(H, X_\alpha)\big)\Big),$$
with the norm $\|G\|_{\mathbb G_\alpha}:=
\left(
\mathbb E\int_0^T
\|G(t)\|_{\gamma(H,X_\alpha)}^p\,dt
\right)^{1/p}$.
Applying Corollary~\ref{cor2.1} pointwise to $G(t,\omega)$ and
then integrating over $(0,T)\times\Omega$, we obtain
\begin{align*}
C_\alpha^{-1}\|G\|_{\mathbb G_\alpha}
\leq\left(
\mathbb E\int_0^T
\|A_0^\alpha G(t)\|_{\gamma(H,X_0)}^p\,dt
\right)^{1/p}
\leq C_\alpha\|G\|_{\mathbb G_\alpha}.
\end{align*}
Thus, the two norms  are equivalent, and the equivalence constant $C_\alpha$ is independent of $G$.

Let $-A_0$ generate the uniformly bounded analytic semigroup $Q(t) = e^{-tA_0}$ in $X_0$. For $G\in \mathbb{G}_\alpha$, then the autonomous stochastic equation  with the generalized fractional derivative
\begin{equation*}
\left\{
\begin{aligned}
&D_{0+}^{\lambda,\mu}u(t) + A_0u(t) = G(t)dW_H(t), && t \in (0, T],  \\
&I_{0+}^{(1-\lambda)(1-\mu)}u(0) = 0,
\end{aligned}\right.
\end{equation*}
admits a mild solution
\begin{equation}\label{e5.8}
u(t) = \int_0^t \psi_\mu(t-s)G(s)d W_H(s),
\end{equation}
where $M_\mu(r)$ denotes the Wright-type function and
$$\psi_\mu(t):=\psi_{\mu,A_0}(t)
=t^{\mu-1}\int_0^\infty\mu rM_\mu(r)Q(t^\mu r)\,dr.$$

\begin{lemma}\label{lem5.3}
Let Assumption \ref{assump5.1} hold. For any $\beta\in[0,1]$, $\zeta\in[0,1]$, and $t, h > 0$, the following uniform estimates hold
\begin{align}\label{e5.9}
\|A_0^\beta\psi_\mu(t)\|_{L(X_0)}\leq C_\beta t^{\mu(1-\beta)-1},
\end{align}
\begin{align}\label{e5.10}
 \|A_0^\beta\psi_\mu'(t)\|_{L(X_0)}\leq C_\beta t^{\mu(1-\beta)-2},
\end{align}
\begin{align}\label{e5.11}
\|A_0^\beta(\psi_\mu(t+h)-\psi_\mu(t))\|_{L(X_0)}
\leq C_{\beta,\zeta} h^\zeta t^{\mu(1-\beta)-1-\zeta}.
\end{align}
\end{lemma}
\begin{proof}
Applying the fractional power $A_0^\beta$ to $\psi_\mu(t)$, we obtain
$$
A_0^\beta\psi_\mu(t) = t^{\mu-1} \int_0^\infty \mu r M_\mu(r) A_0^\beta Q(t^\mu r) \, dr.
$$
According to the estimate $\|A_0^\beta Q(\tau)\|_{L(X_0)} \leq C_\beta \tau^{-\beta}$ for $\tau > 0$, we deduce
\begin{align*}
\|A_0^\beta\psi_\mu(t)\|_{L(X_0)}
&\leq
t^{\mu-1} \int_0^\infty \mu r M_\mu(r) \|A_0^\beta Q(t^\mu r)\|_{L(X_0)} \, dr \\
&\leq
t^{\mu-1} \int_0^\infty \mu r M_\mu(r) C_\beta (t^\mu r)^{-\beta} \, dr \leq
C_\beta \mu t^{\mu(1-\beta)-1}.
\end{align*}
Hence, we have \eqref{e5.9}. By the Dunford integral representation, it follows from
$$
\psi_\mu(t) = \frac{1}{2\pi i} \int_\Gamma e^{zt} (z^\mu + A_0)^{-1} \, dz
$$
that
$$
\psi_\mu'(t) = \frac{1}{2\pi i} \int_\Gamma z e^{zt} (z^\mu + A_0)^{-1} \, dz.
$$
Note that $\|A_0^\beta (z^\mu + A_0)^{-1}\|_{L(X_0)} \leq C |z|^{\mu(\beta-1)}$. Let $\xi = zt$, then $dz = t^{-1}d\xi$ and
\begin{align*}
\|A_0^\beta \psi_\mu'(t)\|_{L(X_0)}
&\leq
\frac{1}{2\pi i} \int_{\Gamma'} \left| \frac{\xi}{t} \right| |e^\xi| \|A_0^\beta \big( (\xi/t)^\mu + A_0 \big)^{-1}\|_{L(X_0)} \, \frac{|d\xi|}{t} \\
&\leq
C \int_{\Gamma'} \frac{|\xi|}{t} |e^\xi| \left| \frac{\xi}{t} \right|^{\mu(\beta-1)} \frac{|d\xi|}{t}
\leq C t^{\mu(1-\beta)-2}.
\end{align*}
Then \eqref{e5.10} holds. For $\zeta = 0$, it follows from the triangle inequality and \eqref{e5.9} that
\begin{align*}
\|A_0^\beta(\psi_\mu(t+h)-\psi_\mu(t))\|_{L(X_0)}
&\leq
\|A_0^\beta\psi_\mu(t+h)\|_{L(X_0)} + \|A_0^\beta\psi_\mu(t)\|_{L(X_0)} \\
&\leq
C_\beta (t+h)^{\mu(1-\beta)-1} + C_\beta t^{\mu(1-\beta)-1} \leq
2C_\beta t^{\mu(1-\beta)-1}.
\end{align*}
For $\zeta = 1$, by \eqref{e5.10}, we obtain
\begin{align*}
\|A_0^\beta(\psi_\mu(t+h)-\psi_\mu(t))\|_{L(X_0)}
&=
\left\| \int_t^{t+h} A_0^\beta\psi_\mu'(s) \, ds \right\|_{L(X_0)} 
\leq
C_\beta h t^{\mu(1-\beta)-2}.
\end{align*}
For the case $0 < \zeta < 1$, we employ the inequality $\min\{a, b\} \leq a^{1-\zeta}b^\zeta$ to interpolate between the bounds for $\zeta=0$ and $\zeta=1$
\begin{align*}
\|A_0^\beta(\psi_\mu(t+h)-\psi_\mu(t))\|_{L(X_0)}
&\leq
\big( 2C_\beta t^{\mu(1-\beta)-1} \big)^{1-\zeta} \big( C_\beta h t^{\mu(1-\beta)-2} \big)^\zeta
=
\widetilde{C}_{\beta,\zeta} \, h^\zeta \, t^{\mu(1-\beta)-1-\zeta}.
\end{align*}
Therefore, we deduce \eqref{e5.11}. The proof is complete.
\end{proof}

\begin{lemma}\label{lem5.4}
Let $\alpha\in(1/(2\mu),1]$ and $G\in\mathbb G_\alpha$.  Then the stochastic
convolution
$$u(t)=\int_0^t\psi_\mu(t-s)G(s)\,dW_H(s)$$
is well defined and, for every $\theta\in[0,\mu-1/2)$,
\begin{equation}\label{e5.12}
 \|u\|_{L^p(\Omega;
 H^{\theta,p}(0,T;X_{1-\theta/\mu}))}
 \leq C_\theta\|G\|_{\mathbb G_\alpha}.
\end{equation}
\end{lemma}
\begin{proof}
Fix $\theta<\mu-1/2$, let $\rho=1-\theta/\mu$ and $F=A_0^\alpha G$. By Corollary~\ref{cor2.1},
\begin{equation*}
\|F\|_{L^p(\Omega;L^p(0,T;\gamma(H,X_0)))}
\leq C_\alpha\|G\|_{\mathbb G_\alpha}.
\end{equation*}
If $\rho\geq\alpha$, we have $A_0^\rho \psi_\mu(r) G(s) = A_0^{\rho-\alpha} \psi_\mu(r) F(s)$.
It follows from Lemma \ref{lem5.3} that
\begin{equation}\label{e5.13}
\|A_0^\rho\psi_\mu(r)G(s)\|_{\gamma(H,X_0)}
\leq Cr^{q_\theta}\|F(s)\|_{\gamma(H,X_0)},
\end{equation}
with $q_\theta = \mu(1-\rho+\alpha)-1$.
If $\rho<\alpha$, by the boundedness of negative power $A_0^{\rho-\alpha}$, we get \eqref{e5.13} for $q_\theta = \mu-1$. Thus,  \eqref{e5.13} holds with
$$q_\theta=
\begin{cases}
\mu(1-\rho+\alpha)-1, & \rho\geq\alpha,\\[1mm]
\mu-1, & \rho<\alpha.
\end{cases}$$

We claim $q_\theta+1/2>\theta$. If $\rho\geq\alpha$, then, since
$\rho=1-\theta/\mu$ and $\alpha>1/(2\mu)$, we have
$q_\theta+1/2-\theta
=\mu\alpha-1/2>0$.
If $\rho<\alpha$, then
$q_\theta+1/2-\theta
=\mu-1/2-\theta>0$.
Thus, we have
$q_\theta+1/2>\theta$.

Choose numbers $s$ and $\zeta$ such that
\begin{equation}\label{e5.14}
 \theta<s<\zeta<
 \min\left\{1,q_\theta+\frac12\right\}.
\end{equation}
According to the BDG inequality and Young's
inequality, we obtain
\begin{equation}\label{e5.15}
\|u\|_{L^p(\Omega;L^p(0,T;X_\rho))}
\leq C\|r^{q_\theta}\|_{L^2(0,T)}
\|F\|_{L^p(\Omega;L^p(0,T;\gamma(H,X_0)))}.
\end{equation}
Since $\theta\geq0$ and $q_\theta+1/2>\theta$, we have
$q_\theta>-1/2$. Hence, the $L^2$ norm in \eqref{e5.15} is finite.

For $h\in(0,T)$, we obtain
\begin{align*}
 u(t+h)-u(t)
 =&\int_0^t
 [\psi_\mu(t+h-r)-\psi_\mu(t-r)]G(r)\,dW_H(r)+\int_t^{t+h}\psi_\mu(t+h-r)G(r)\,dW_H(r)\\
 =:&J_1(t,h)+J_2(t,h).
\end{align*}
From \eqref{e5.11} and \eqref{e5.14}, we have
$$\|A_0^\rho[\psi_\mu(r+h)-\psi_\mu(r)]G(s)\|_{\gamma(H,X_0)}
 \leq Ch^\zeta r^{q_\theta-\zeta}
 \|F(s)\|_{\gamma(H,X_0)}.$$
Since $q_\theta-\zeta>-1/2$, combining the BDG inequality and Young's
inequality, we also deduce
\begin{equation}\label{e5.16}
 \|J_1(\cdot,h)\|_{L^p(\Omega;
 L^p(0,T-h;X_\rho))}
 \leq Ch^\zeta\|G\|_{\mathbb G_\alpha}.
\end{equation}
It follows from \eqref{e5.13} that
\begin{equation}\label{e5.17}
 \|J_2(\cdot,h)\|_{L^p(\Omega;
 L^p(0,T-h;X_\rho))}
 \leq Ch^{q_\theta+1/2}\|G\|_{\mathbb G_\alpha}.
\end{equation}
Then by inequality $(a+b)^p \leq 2^{p-1}(a^p + b^p)$, we deduce
\begin{align*}
&\int_0^T \frac{\|u(\cdot+h) - u(\cdot)\|_{L^p(\Omega;
L^p(0,T-h;X_\rho))}^p}{h^{sp+1}} \, dh \\
= &\int_0^T h^{-sp-1}
\|u(\cdot+h)-u(\cdot)\|_{L^p(\Omega;
L^p(0,T-h;X_\rho))}^p\,dh\\
\leq& C\|G\|_{\mathbb G_\alpha}^p
\int_0^T
\left(h^{(\zeta-s)p-1}
+h^{(q_\theta+1/2-s)p-1}\right)\,dh<\infty.
\end{align*}
Together with \eqref{e5.15}, we have
$u\in L^p(\Omega;W^{s,p}(0,T;X_\rho))$.
For $0\leq\theta<s<1$, we deduce the embedding
$$ W^{s,p}(0,T;X_\rho)
 \hookrightarrow H^{\theta,p}(0,T;X_\rho).$$
Hence, it follows that \eqref{e5.12} holds. The proof is complete.
\end{proof}
\begin{remark}\label{rem5.1}
The conditions on $\alpha$ and $\theta$ ensure the integrability of the kernels in the proof. When $\theta=0$, the kernel is square-integrable near $r=0$ only if $\alpha>1/(2\mu)$. The proof of the time regularity estimate also requires $\theta<\mu-1/2$. Therefore, the endpoint cases are not covered by Lemma \ref{lem5.4}.
\end{remark}

\subsection{Stochastic maximal regularity}
First, we consider the following stochastic equation
\begin{equation}\label{e5.18}
 \left\{
\begin{aligned}
&D_{0+}^{\lambda,\mu}u(t)+A(t)u(t)=f(t)+ G(t)dW_H(t), && t \in (0, T],  \\
&I_{0+}^{(1-\lambda)(1-\mu)}u(0)=0.
\end{aligned}\right.
\end{equation}
Let $f\in L^p(\Omega;L^p(0,T;X_0))$ be progressively measurable and let $G\in\mathbb G_\alpha$. A $X_0$-valued process $U:(0,T)\times\Omega\to X_0$ is called an integrated strong solution if $U$ is strongly progressively measurable,
$$
U\in L^p\bigl(\Omega;L^1(0,T;X_1)\bigr),\quad
A(\cdot)U(\cdot)
\in L^p\bigl(\Omega;L^1(0,T;X_0)\bigr),
$$
and satisfies
\begin{align}\label{e5.19}
U(t)+\int_0^tg_\mu(t-s)A(s)U(s)\,ds=\int_0^t g_\mu(t-s)f(s)\,ds+
\int_0^t g_\mu(t-s)G(s)\,dW_H(s)
\end{align}
for almost every $t\in(0,T)$.
The process $s \mapsto g_\mu(t-s)G(s)$ is stochastically integrable on $(0,t)$ for almost every $t$ in $X_0$.

\begin{definition}\label{def5.2}
Let $\mu\in(1/2,1)$ and $\alpha\in(1/(2\mu),1]$. We say that $A(\cdot)$ has stochastic maximal $L^p$-regularity if for every progressively measurable $f\in L^p(\Omega;L^p(0,T;X_0))$ and $G\in\mathbb G_\alpha$, the problem \eqref{e5.18} admits a unique integrated strong solution  $U\in L^p(\Omega;L^p(0,T;X_1))$ satisfying \eqref{e5.19}.
Furthermore, for $0\leq\theta<\mu-1/2$,
the solution satisfies
$U\in L^p(\Omega;H^{\theta,p}(0,T;X_{1-\theta/\mu}))$ and
\begin{equation}\label{e5.20}
 \|U\|_{L^p(\Omega;
 H^{\theta,p}(0,T;X_{1-\theta/\mu}))}
 \leq C_\theta\left(
 \|f\|_{L^p(\Omega;L^p(0,T;X_0))}
 +\|G\|_{\mathbb G_\alpha}
 \right).
\end{equation}
\end{definition}

According to the result of Section \ref{sec3}, we have that $U_{I}(t)
 :=H_{\mu,\nu}(t,0)u_0$ is the solution of the homogeneous
non-autonomous equation
\begin{equation*}
 \left\{
\begin{aligned}
&D_{0+}^{\lambda,\mu}u(t)+A(t)u(t)=0, && t \in (0, T],  \\
&I_{0+}^{(1-\lambda)(1-\mu)}u(0)=u_0.
\end{aligned}\right.
\end{equation*}
We now consider the following stochastic problem
\begin{equation}\label{e5.21}
\left\{
\begin{aligned}
&D_{0+}^{\lambda,\mu}U(t)+A(t)U(t)
=f(t)+G(t)\,dW_H(t), && t\in(0,T],\\
&I_{0+}^{1-\nu}U(0)=u_0.
\end{aligned}
\right.
\end{equation}
Then we have the integrated equation
\begin{align}\label{e5.22}
 &\nonumber U(t)-g_\nu(t)u_0
 +\int_0^t g_\mu(t-s)A(s)U(s)\,ds\\
 =&\int_0^t g_\mu(t-s)f(s)\,ds
 +\int_0^t g_\mu(t-s)G(s)\,dW_H(s).
\end{align}

A strongly progressively measurable process $U:(0,T)\times\Omega\to X_0$
is called an integrated strong solution of  problem \eqref{e5.21} if
$$U\in L^p\bigl(\Omega;L^1(0,T;X_1)\bigr),
\quad
A(\cdot)U(\cdot)
\in L^p\bigl(\Omega;L^1(0,T;X_0)\bigr),$$
and \eqref{e5.22} holds for almost every $t\in(0,T)$.  For such a
solution, set $U_{R}:=U-U_{I}$. We call $U$ a regularity integrated strong solution if its regular part satisfies
$U_{R}
\in L^p\bigl(\Omega;L^p(0,T;X_1)\bigr)$.

\begin{theorem}\label{the5.1}
Suppose Assumption \ref{assump5.1} holds, operator $A(\cdot) \in \text{DMR}(p, \mu, T)$ and $\nu-\mu<\vartheta$.  Let $\alpha\in(1/(2\mu),1]$, $f\in L^p(\Omega;L^p(0,T;X_0))$ be
progressively measurable, $G\in\mathbb G_\alpha$ and let $u_0$ be strongly $\mathcal F_0$-measurable with
$u_0\in L^p(\Omega;X_1)$. Then
\eqref{e5.21} has a unique regularity integrated strong solution
$$ U=U_{I}+U_{R}.$$
For $\theta\in[0,\mu-1/2)$,
\begin{equation}\label{e5.23}
 \|U_{R}\|_{L^p(\Omega;
 H^{\theta,p}(0,T;X_{1-\theta/\mu}))}
 \leq C_\theta\left(
 \|f\|_{L^p(\Omega;L^p(0,T;X_0))}
 +\|G\|_{\mathbb G_\alpha}
 \right).
\end{equation}
Moreover, if $\nu>1-1/p$ and $\vartheta>1-1/p$,
then $$U_{I}\in L^p\bigl(\Omega;L^p(0,T;X_1)\bigr).$$
Then the solution $U$ satisfies $U\in L^p\bigl(\Omega;L^p(0,T;X_1)\bigr)$ and
$$\|U\|_{L^p(\Omega;L^p(0,T;X_1))}
\leq C\left(
\|u_0\|_{L^p(\Omega;X_1)}
+\|f\|_{L^p(\Omega;L^p(0,T;X_0))}
+\|G\|_{\mathbb G_\alpha}
\right).$$
\end{theorem}
\begin{proof}
We first separate out the autonomous stochastic part. Let
\begin{equation*}
 U_1(t)=\int_0^t\psi_\mu(t-s)G(s)\,dW_H(s).
\end{equation*}
By Lemma \ref{lem5.4}, we deduce
\begin{equation}\label{e5.24}
 \|U_1\|_{L^p(\Omega;
 H^{\theta,p}(0,T;X_{1-\theta/\mu}))}
 \leq C_\theta\|G\|_{\mathbb G_\alpha}
\end{equation}
for $\theta<\mu-1/2$. In particular, it follows from the case $\theta=0$ that
\begin{equation}\label{e5.25}
 \|U_1\|_{L^p(\Omega;L^p(0,T;X_1))}
 \leq C\|G\|_{\mathbb G_\alpha}.
\end{equation}
Combining with Assumption~\ref{assump5.1} and \eqref{e5.24}, we have
\begin{align*}
\|(A_0-A(\cdot))U_1\|_{L^p(\Omega;L^p(0,T;X_0))}\leq
\sup_{t\in[0,T]}\|A_0-A(t)\|_{L(X_1,X_0)}
\|U_1\|_{L^p(\Omega;L^p(0,T;X_1))}\leq C\|G\|_{\mathbb G_\alpha}.
\end{align*}
Set $h(t):=f(t)+(A_0-A(t))U_1(t)$. Then $h$ is progressively measurable and
\begin{equation}\label{e5.26}
\|h\|_{L^p(\Omega;L^p(0,T;X_0))}
\leq C\left(\|f\|_{L^p(\Omega;L^p(0,T;X_0))}
+\|G\|_{\mathbb G_\alpha}\right).
\end{equation}
By Lemma~\ref{lem5.2}, the equation
$$D_{0+}^{\lambda,\mu}U_2(t)+A(t)U_2(t)=h(t)$$
has a unique progressively measurable integrated strong solution. Moreover, together with \eqref{e5.26}, we deduce
\begin{align}\label{e5.27}
 \|U_2\|_{L^p(\Omega;
 H^{\theta,p}(0,T;X_{1-\theta/\mu}))}
 \leq C_\theta\left(
 \|f\|_{L^p(\Omega;L^p(0,T;X_0))}
 +\|G\|_{\mathbb G_\alpha}
 \right).
\end{align}
Define
$$U:=U_{I}+U_1+U_2,
\quad U_{R}:=U_1+U_2.$$
Then we verify the integrability of an
integrated strong solution. Since $u_0\in L^p(\Omega;X_1)$ and
$\nu-\mu<\vartheta$, it follows from Theorem~\ref{the3.1} that
\begin{align}\label{e5.28}
\|U_{I}\|_{L^p(\Omega;L^1(0,T;X_1))}
+\|A(\cdot)U_{I}(\cdot)\|_{L^p(\Omega;L^1(0,T;X_0))}
\leq C\|u_0\|_{L^p(\Omega;X_1)}.
\end{align}
Moreover, combining with \eqref{e5.25}, the case $\theta=0$ of \eqref{e5.27},
Assumption~\ref{assump5.1}, and the embedding
$L^p(0,T)\hookrightarrow L^1(0,T)$, we deduce
$$U\in L^p\bigl(\Omega;L^1(0,T;X_1)\bigr),
\quad A(\cdot)U(\cdot)
\in L^p\bigl(\Omega;L^1(0,T;X_0)\bigr).$$
Thus, $U$ satisfies the integrability requirements for an integrated
strong solution.

The autonomous stochastic convolution $U_1$ satisfies
$$U_1(t)+\int_0^tg_\mu(t-s)A_0U_1(s)\,ds
=\int_0^tg_\mu(t-s)G(s)\,dW_H(s),$$
and $U_2$ satisfies
$$U_2(t)+\int_0^tg_\mu(t-s)A(s)U_2(s)\,ds
=\int_0^tg_\mu(t-s)h(s)\,ds.$$
Note that
$$U_{I}(t)-g_\nu(t)u_0
+\int_0^t g_\mu(t-s)A(s)U_{I}(s)\,ds
=0.$$
Hence, we deduce
\begin{align*}
&U(t)-g_\nu(t)u_0
+\int_0^t g_\mu(t-s)
   \bigl[A(s)U_{I}(s)+A_0U_1(s)+A(s)U_2(s)\bigr]ds\\
=&\int_0^t g_\mu(t-s)f(s)\,ds
+\int_0^t g_\mu(t-s)(A_0-A(s))U_1(s)\,ds+
\int_0^t g_\mu(t-s)G(s)\,dW_H(s).
\end{align*}
Since $A_0U_1(s)=A(s)U_1(s)+(A_0-A(s))U_1(s)$, it is \eqref{e5.22}.
Furthermore, by \eqref{e5.24}, \eqref{e5.27}, and the triangle inequality,
\begin{align*}
\|U_{R}\|_{L^p(\Omega;
H^{\theta,p}(0,T;X_{1-\theta/\mu}))}
&\leq
\|U_1\|_{L^p(\Omega;
H^{\theta,p}(0,T;X_{1-\theta/\mu}))}
+\|U_2\|_{L^p(\Omega;
H^{\theta,p}(0,T;X_{1-\theta/\mu}))}\\
&\leq C_\theta\left(
\|f\|_{L^p(\Omega;L^p(0,T;X_0))}
+\|G\|_{\mathbb G_\alpha}
\right),
\end{align*}
then \eqref{e5.23} holds.

We next prove the uniqueness.  Let $U$ and $\widetilde U$ be two regularity integrated strong solutions with the same initial value and set $Z:=U-\widetilde U$. Then $Z$ satisfies
$$Z(t)+\int_0^t g_\mu(t-s)A(s)Z(s)\,ds=0$$
with zero initial value. Moreover, since
$$U-U_{I},
\,\, \widetilde U-U_{I}
\in L^p\bigl(\Omega;L^p(0,T;X_1)\bigr),$$
we get $Z\in L^p\bigl(\Omega;L^p(0,T;X_1)\bigr)$. For almost every $\omega\in\Omega$,
$D_{0+}^{\lambda,\mu}Z(\cdot,\omega)
+A(\cdot)Z(\cdot,\omega)=0$.
It follows from Assumption \ref{assump5.1} that
$$D_{0+}^{\lambda,\mu}Z(\cdot,\omega)
=-A(\cdot)Z(\cdot,\omega)
\in L^p(0,T;X_0).$$
By the uniqueness of $A(\cdot)\in\mathrm{DMR}(p,\mu,T)$, we have $Z(\cdot,\omega)=0$ for almost every $\omega$.  Hence, we deduce $U=\widetilde U$.

It remains to prove the result for the full solution. According to  \eqref{e5.23}, for the case
$\theta=0$, we have $U_{R}\in L^p\bigl(\Omega;L^p(0,T;X_1)\bigr)$.
Therefore, it suffices to prove the  regularity for $U_{I}$.
By Theorem \ref{the3.1}, we obtain
$$U_{I}(t)=\psi_{\nu,A(0)}(t)u_0+W(t),$$
where $$\phi_1(r):=\mathcal Q_1(r,0)u_0,
\quad
W(t):=\int_0^t\psi_{\mu,A(r)}(t-r)\phi_1(r)\,dr.$$
By Lemma~\ref{lem2.4} with
$\alpha=1$, the $X_1$-norm is equivalent to
$\|A(0)\,\cdot\|_{X_0}$.  Hence, it follows from Lemma~\ref{lem3.2}(i) and \eqref{e3.23} that
\begin{align*}
\|\psi_{\nu,A(0)}(t)u_0\|_{L^p(\Omega;X_1)}
&\leq C\|A(0)\psi_{\nu,A(0)}(t)u_0\|_{L^p(\Omega;X_0)}\\
&=C\|\psi_{\nu,A(0)}(t)A(0)u_0\|_{L^p(\Omega;X_0)}\leq Ct^{\nu-1}\|u_0\|_{L^p(\Omega;X_1)}.
\end{align*}
Since $\nu>1-1/p$, we obtain the function $t^{\nu-1} \in L^p(0,T)$.

According to Lemma~\ref{lem3.12}, $W(t)\in D(A(t))=X_1$. By Lemma~\ref{lem2.4} with $\alpha=1$, we get
$$\|W(t)\|_{X_1}\leq C\|A(t)W(t)\|_{X_0},$$
where $C$ is independent of $t$. It follows from Lemmas~\ref{lem3.11} and \ref{lem3.12} that
$A(t)W(t)=B_{\phi_1}(t,0)+J_{\phi_1}(t,0)$, $\|B_{\phi_1}(t,0)\|_{L^p(\Omega;X_0)}
\leq Ct^{2\vartheta+\nu-\mu-1}
   \|u_0\|_{L^p(\Omega;X_1)}$ and
$$\|J_{\phi_1}(t,0)\|_{L^p(\Omega;X_0)}
\leq C\left(t^{\vartheta-1}
   +t^{\vartheta+\nu-\mu-1}
   +t^{2\vartheta+\nu-\mu-1}\right)
   \|u_0\|_{L^p(\Omega;X_1)}.$$
Finally, for $0<t\leq T$, it follows that
\begin{align*}
\|U_{I}(t)\|_{L^p(\Omega;X_1)}
\leq C_T\left(t^{\nu-1}+t^{\vartheta-1}
+t^{\vartheta+\nu-\mu-1}\right)
\|u_0\|_{L^p(\Omega;X_1)}.
\end{align*}
Since $\nu\geq\mu$, $\nu>1-1/p$ and
$\vartheta>1-1/p$, we obtain
\begin{align}\label{e5.29}
\|U_{I}\|_{L^p(\Omega;L^p(0,T;X_1))}
\leq C_T\|u_0\|_{L^p(\Omega;X_1)}.
\end{align}
Together with \eqref{e5.23} for
$\theta=0$, we get
\begin{align*}
\|U\|_{L^p(\Omega;L^p(0,T;X_1))}
&\leq
\|U_{I}\|_{L^p(\Omega;L^p(0,T;X_1))}
+\|U_{R}\|_{L^p(\Omega;L^p(0,T;X_1))}\\
&\leq C\left(
\|u_0\|_{L^p(\Omega;X_1)}
+\|f\|_{L^p(\Omega;L^p(0,T;X_0))}
+\|G\|_{\mathbb G_\alpha}
\right).
\end{align*}
The proof is complete.
\end{proof}

\begin{remark}\label{rem5.2}
Theorem \ref{the5.1} does not assert the fractional derivative $D_{0+}^{\lambda,\mu} U_1 in L^p(\Omega; L^p(0,T; X_0))$. Since $U_1$ is driven by temporal white noise, it does not belong to $L^p(0,T;X_0)$. Thus, the fractional derivative of $U_1$ generally does not have temporal $L^p$-regularity. However,  \eqref{e5.23} provides the space-time regularity of $U_1$ in the subcritical range established in Lemma \ref{lem5.4}.
\end{remark}

\begin{corollary}\label{cor5.1}
Under the hypotheses of Theorem \ref{the5.1}, let $\mu-1/2>1/p$. Then for $1/p<\theta<\mu-1/2$,
the regular part $U_{R}$ admits an
$X_{1-\theta/\mu}$-valued H\"{o}lder continuous version and
\begin{equation}\label{e5.30}
 U_{R}\in L^p(\Omega;
 C^{\theta-1/p}([0,T];X_{1-\theta/\mu})),
\end{equation}
\begin{align}\label{e5.31}
\|U_R\|_{L^p(\Omega;
C^{\theta-1/p}([0,T];X_{1-\theta/\mu}))}
\leq C_\theta\left(
\|f\|_{L^p(\Omega;L^p(0,T;X_0))}
+\|G\|_{\mathbb G_\alpha}
\right).
\end{align}
\end{corollary}
\begin{proof}
According to \cite[Proposition 2.3]{Portal}, we obtain the embedding
$$ H^{\theta,p}(0,T;X)
\hookrightarrow C^{\theta-1/p}([0,T];X)$$
for $\theta>\frac1p$. By Theorem \ref{the5.1}, we deduce \eqref{e5.30} and \eqref{e5.31}. The proof is complete.
\end{proof}

\begin{remark}\label{rem5.3}
Theorem~\ref{the5.1} concerns the regular part
$U_{R}=U-H_{\mu,\nu}(\cdot,0)u_0$. The full solution generally has the
initial singularity $t^{\nu-1}u_0$. It is  natural to study the full solution in a weighted space with time weight $w_a(t)=t^a$. It is necessary to establish an appropriate weighted Bessel potential restriction space, the corresponding extension theorem, and the restriction of the initial value $u_0$. Therefore, this paper proves the maximal regularity on $U_{R}$, while the weighted regularity of the full solution is reserved as a possible extension problem.
\end{remark}

\subsection{Semilinear problem}
We finally consider the semilinear equation
\begin{equation}\label{e5.32}
\left\{
\begin{aligned}
&D_{0+}^{\lambda,\mu}U(t)+A(t)U(t)
=F(t,U(t))+G(t,U(t))\,dW_H(t), &&t\in(0,T],\\
&I_{0+}^{1-\nu}U(0)=u_0.
\end{aligned}
\right.
\end{equation}
Then we have the integrated equation
\begin{align}\label{e5.33}
&\nonumber U(t)-g_\nu(t)u_0
+\int_0^t g_\mu(t-s)A(s)U(s)\,ds\\
=&\int_0^t g_\mu(t-s)F(s,U(s))\,ds
+\int_0^t g_\mu(t-s)G(s,U(s))\,dW_H(s).
\end{align}

\begin{assumption}\label{assump5.2}
Let $\kappa\in[0,1]$ and $\alpha\in(1/(2\mu),1]$. The maps $F:[0,T]\times\Omega\times X_\kappa\to X_0$, $G:[0,T]\times\Omega\times X_\kappa
\to\gamma(H,X_\alpha)$ are measurable with respect to the product of the progressive $\sigma$-field on $[0,T]\times\Omega$ and the Borel $\sigma$-field of $X_\kappa$. Let
$F(\cdot,\cdot,0)\in L^p(\Omega;L^p(0,T;X_0))$, $G(\cdot,\cdot,0)\in\mathbb G_\alpha$. Suppose that there exist constants $L_F,L_G\geq0$ such that for all $x,y\in X_\kappa$ and almost all $(t,\omega)$,
\begin{align}\label{e5.34}
\|F(t,\omega,x)-F(t,\omega,y)\|_{X_0}
\leq L_F\|x-y\|_{X_\kappa},
\end{align}
\begin{align}\label{e5.35}
\|G(t,\omega,x)-G(t,\omega,y)\|_{\gamma(H,X_\alpha)}
\leq L_G\|x-y\|_{X_\kappa}.
\end{align}
Taking $y=0$ in \eqref{e5.34}, \eqref{e5.35} and applying the triangle inequality, we obtain
\begin{align}\label{e5.36}
\|F(t,\omega,x)\|_{X_0}
\leq\|F(t,\omega,0)\|_{X_0}+L_F\|x\|_{X_\kappa},
\end{align}
\begin{align}\label{e5.37}
\|G(t,\omega,x)\|_{\gamma(H,X_\alpha)}
\leq \|G(t,\omega,0)\|_{\gamma(H,X_\alpha)}
+L_G\|x\|_{X_\kappa}.
\end{align}
\end{assumption}
A strongly progressively measurable process $U$ is called an integrated strong solution of \eqref{e5.32} if for $F(\cdot,\cdot,U(\cdot,\cdot))
\in L^p\bigl(\Omega;L^p(0,T;X_0)\bigr)$,
$G(\cdot,\cdot,U(\cdot,\cdot))
\in\mathbb G_\alpha$, it satisfies
$$U\in L^p\bigl(\Omega;L^1(0,T;X_1)\bigr), \quad A(\cdot)U(\cdot)\in L^p\bigl(\Omega;L^1(0,T;X_0)\bigr),$$  \eqref{e5.33} holds and $U-U_{I}\in L^p\bigl(\Omega;L^p(0,T;X_1)\bigr)$.

For $x\in X_1$, let $C_\kappa>0$ be the embedding $X_1\hookrightarrow X_\kappa$ constant  with
\begin{align}\label{e5.38}
\|x\|_{X_\kappa}\le C_\kappa\|x\|_{X_1}.
\end{align}
Set $M:=C_\kappa(L_F+L_G)$. Let $U_R$ denote the regular part of the solution to \eqref{e5.21}. If $h\in L^p(\Omega;L^p(0,T;X_0))$ is progressively measurable and $K\in\mathbb G_\alpha$, it follows from the case $\theta=0$ of Theorem~\ref{the5.1} that there exists a constant $C_{L}>0$ such that
\begin{equation}\label{e5.39}
\|U_{R}\|_{L^p(\Omega;L^p(0,T;X_1))}
\leq C_{L}
\left(
\|h\|_{L^p(\Omega;L^p(0,T;X_0))}
+\|K\|_{\mathbb G_\alpha}
\right),
\end{equation}
where $C_{L}$ is independent of
$h,K,L_F$, and $L_G$.
\begin{theorem}\label{the5.2}
Suppose that Assumptions~\ref{assump5.1} and \ref{assump5.2} hold, $A(\cdot)\in\mathrm{DMR}(p,\mu,T)$, and $\nu-\mu<\vartheta$, $\nu>1-1/p$, $\vartheta>1-1/p$. Let $u_0\in L^p(\Omega,\mathcal F_0;X_1)$. If $C_{L}M<1$, then problem \eqref{e5.32} admits a unique regularity integrated strong solution $U=U_{I}+U_{R}$.
For every $\theta\in[0,\mu-1/2)$,
\begin{align}\label{e5.40}
&\nonumber
\|U_{R}\|_{L^p(\Omega;
H^{\theta,p}(0,T;X_{1-\theta/\mu}))}\\
\leq& C_{\theta,L_F,L_G}\left(
\|F(\cdot,\cdot,0)\|_{L^p(\Omega;L^p(0,T;X_0))}
+\|u_0\|_{L^p(\Omega;X_1)}
+\|G(\cdot,\cdot,0)\|_{\mathbb G_\alpha}
\right).
\end{align}
Moreover, $U\in L^p\bigl(\Omega;L^p(0,T;X_1)\bigr)$.
\end{theorem}
\begin{proof}
Let
$$\mathbb E_{\mathcal F}
:=\left\{V\in L^p\bigl(\Omega;L^p(0,T;X_1)\bigr):
V\text{ is progressively measurable}\right\},$$
with the norm $\|V\|_{\mathbb E_{\mathcal F}}:=\left(\mathbb E\int_0^T
\|V(t)\|_{X_1}^p\,dt\right)^{1/p}$.
Then the space $\mathbb E_{\mathcal F}$ is a linear subspace. Moreover, it
is closed in $L^p(\Omega;L^p(0,T;X_1))$, since an $L^p$-convergent
sequence of progressively measurable processes has a subsequence
converging almost everywhere. Therefore,
$\mathbb E_{\mathcal F}$ is a Banach space.
By \eqref{e5.29} and the definition of the initial component,
$U_{I}\in\mathbb E_{\mathcal F}$. Hence, for $V\in\mathbb E_{\mathcal F}$, we deduce $U_{I}+V\in\mathbb E_{\mathcal F}$.
For $V\in\mathbb E_{\mathcal F}$, define
$$H_V(t,\omega):=F\bigl(t,\omega,
U_{I}(t,\omega)+V(t,\omega)\bigr),\quad K_V(t,\omega):=G\bigl(t,\omega,
U_{I}(t,\omega)+V(t,\omega)\bigr).$$
Since $U_{I}+V$ is progressively measurable and $F,G$ are measurable with respect to the product of the progressive $\sigma$-field, the processes $H_V$ and $K_V$ are
progressively measurable.
Applying \eqref{e5.36} pointwise with
$x=U_{I}(t,\omega)+V(t,\omega)$,
combining with \eqref{e5.38}, we obtain
\begin{align*}
\|H_V\|_{X_0}\le
\|F(t,\omega,0)\|_{X_0}
+C_\kappa L_F
\left(
\|U_{I}(t,\omega)\|_{X_1}
+\|V(t,\omega)\|_{X_1}
\right).
\end{align*}
Therefore, it follows from Minkowski's inequality that
\begin{align*}
\|H_V\|_{L^p(\Omega;L^p(0,T;X_0))}
\le
\|F(\cdot,\cdot,0)\|
_{L^p(\Omega;L^p(0,T;X_0))}
+C_\kappa L_F
\left(
\|U_{I}\|_{{\mathbb E_{\mathcal F}}}
+\|V\|_{\mathbb E_{\mathcal F}}
\right)<\infty.
\end{align*}
Similarly, by \eqref{e5.37} pointwise, we deduce
\begin{align*}
\|K_V\|_{\gamma(H,X_\alpha)}
\le
\|G(t,\omega,0)\|_{\gamma(H,X_\alpha)}
+C_\kappa L_G
\left(
\|U_{I}(t,\omega)\|_{X_1}
+\|V(t,\omega)\|_{X_1}
\right).
\end{align*}
Then we deduce $\|K_V\|_{\mathbb G_\alpha}\le\|G(\cdot,\cdot,0)\|_{\mathbb G_\alpha}+C_\kappa L_G\left(
\|U_{I}\|_{\mathbb E_{\mathcal F}}
+\|V\|_{\mathbb E_{\mathcal F}}\right)
<\infty$. Hence, we deduce $H_V
\in L^p(\Omega;L^p(0,T;X_0))$ and $K_V\in\mathbb G_\alpha$.

According to Theorem~\ref{the5.1}, the following linear problem
\begin{equation}\label{e5.41}
\left\{
\begin{aligned}
&D_{0+}^{\lambda,\mu}Y_V(t)+A(t)Y_V(t)
=H_V(t)+K_V(t)\,dW_H(t),\\
&I_{0+}^{1-\nu}Y_V(0)=u_0
\end{aligned}
\right.
\end{equation}
has a unique integrated strong solution
$Y_V=U_{I}+Y_{V,R}$ with $Y_{V,R}\in \mathbb E_{\mathcal F}$. Define the map $\mathbf T:\mathbb E_{\mathcal F} \to \mathbb E_{\mathcal F}$ by $\mathbf T V:=Y_{V,R}$. Let $V,W\in\mathbb E_{\mathcal F}$. Define $Y_V$ and $Y_W$ with the same initial value $u_0$, then
$$Y_V-Y_W
=Y_{V,R}-Y_{W,{\rm reg}}
=\mathbf T V-\mathbf T W.$$
Moreover, $Z:=\mathbf TV-\mathbf TW$ is the regular part of the
zero initial value linear problem
$$\left\{
\begin{aligned}
&D_{0+}^{\lambda,\mu}Z(t)+A(t)Z(t)
=H_V(t)-H_W(t)
+\bigl(K_V(t)-K_W(t)\bigr)\,dW_H(t),\\
&I_{0+}^{1-\nu}Z(0)=0.
\end{aligned}
\right.$$
For almost every $(t,\omega)\in(0,T)\times\Omega$, by \eqref{e5.34} and the embedding
$X_1\hookrightarrow X_\kappa$, we have
\begin{align*}
\|H_V(t,\omega)-H_W(t,\omega)\|_{X_0}
\leq& L_F\bigl\|\bigl(U_{I}(t,\omega)+V(t,\omega)\bigr)-\bigl(U_{I}(t,\omega)+W(t,\omega)\bigr)\bigr\|_{X_\kappa}\\
\leq& C_\kappa L_F\|V(t,\omega)-W(t,\omega)\|_{X_1}.
\end{align*}
Then we deduce
\begin{align}\label{e5.42}
\|H_V-H_W\|_{L^p(\Omega;L^p(0,T;X_0))}
\leq C_\kappa L_F\|V-W\|_{\mathbb E_{\mathcal F}}.
\end{align}
Similarly, by \eqref{e5.35}, we deduce
\begin{align*}
\|K_V(t,\omega)-K_W(t,\omega)\|_{\gamma(H,X_\alpha)}
\leq L_G\|V(t,\omega)-W(t,\omega)\|_{X_\kappa}\leq C_\kappa L_G\|V(t,\omega)-W(t,\omega)\|_{X_1}.
\end{align*}
It follows that
\begin{align}\label{e5.43}
\|K_V-K_W\|_{\mathbb G_\alpha}
&\leq C_\kappa L_G\|V-W\|_{\mathbb E_{\mathcal F}}.
\end{align}
Combining with \eqref{e5.39}, \eqref{e5.42} and \eqref{e5.43}, we obtain
\begin{align*}
\|\mathbf T V-\mathbf T W\|_{\mathbb E_{\mathcal F}}
&\leq C_L
\left(
\|H_V-H_W\|_{L^p(\Omega;L^p(0,T;X_0))}
+\|K_V-K_W\|_{\mathbb G_\alpha}
\right)\\
&\leq C_LC_\kappa(L_F+L_G)\|V-W\|_{\mathbb E_{\mathcal F}}=C_L M\|V-W\|_{\mathbb E_{\mathcal F}}.
\end{align*}
Since $C_L M<1$, $\mathbf T$ is a contraction. Hence, it follows from the Banach fixed point theorem that there exists a unique
$U_{R}\in\mathbb E_{\mathcal F}$ such that
$\mathbf T U_{R}=U_{R}$.
Thus $U=U_{I}+U_{R}$ satisfies
\eqref{e5.33}. Moreover,
\begin{align*}
\|U_{R}\|_{\mathbb E_{\mathcal F}}
&\leq
\|\mathbf T U_{R}-\mathbf T0\|_{\mathbb E_{\mathcal F}}
+\|\mathbf T0\|_{\mathbb E_{\mathcal F}}\\
&\leq C_L M
\|U_{R}\|_{\mathbb E_{\mathcal F}}+C\left(
\|F(\cdot,\cdot,0)\|_{L^p(\Omega;L^p(0,T;X_0))}
+M\|U_{I}\|_{\mathbb E_{\mathcal F}}
+\|G(\cdot,\cdot,0)\|_{\mathbb G_\alpha}
\right).
\end{align*}
It follows that
\begin{align}\label{e5.44}
\|U_{R}\|_{\mathbb E_{\mathcal F}}
\leq  C_{L_F,L_G}\left(
\|F(\cdot,\cdot,0)\|_{L^p(\Omega;L^p(0,T;X_0))}
+\|u_0\|_{L^p(\Omega;X_1)}
+\|G(\cdot,\cdot,0)\|_{\mathbb G_\alpha}
\right).
\end{align}
According to \eqref{e5.23} of Theorem~\ref{the5.1} with the $F(\cdot,\cdot,U)$ and
$G(\cdot,\cdot,U)$, we obtain
\begin{align}\label{e5.45}
\nonumber\|U_{R}\|_{L^p(\Omega;
H^{\theta,p}(0,T;X_{1-\theta/\mu}))}
\leq & C_\theta\left(
\|F(\cdot,\cdot,U)\|_{L^p(\Omega;L^p(0,T;X_0))}
+\|G(\cdot,\cdot,U)\|_{\mathbb G_\alpha}
\right)\\
\leq &C_\theta\left(
\|F(\cdot,\cdot,0)\|_{L^p(\Omega;L^p(0,T;X_0))}
+\|G(\cdot,\cdot,0)\|_{\mathbb G_\alpha}
+M\|U\|_{\mathbb E_{\mathcal F}}
\right).
\end{align}
It follows from \eqref{e5.44} that
\begin{align*}
\|U\|_{\mathbb E_{\mathcal F}}
\leq\|U_{I}\|_{\mathbb E_{\mathcal F}}+\|U_{R}\|_{\mathbb E_{\mathcal F}}\leq C_{L_F,L_G}
\left(
\|F(\cdot,\cdot,0)\|_{L^p(\Omega;L^p(0,T;X_0))}
+\|U_{I}\|_{\mathbb E_{\mathcal F}}
+\|G(\cdot,\cdot,0)\|_{\mathbb G_\alpha}
\right).
\end{align*}
Combining \eqref{e5.29}, \eqref{e5.44} and \eqref{e5.45}, we obtain \eqref{e5.40}.

Finally, since both $U_{I}$ and $U_{R}$ belong to $\mathbb E_{\mathcal F}$, we have
$U=U_{I}+U_{R}
\in L^p\bigl(\Omega;L^p(0,T;X_1)\bigr)$.

To prove uniqueness, let $\widetilde U$ be another regularity integrated strong solution and set
$$
\widetilde U_{R}:=\widetilde U-U_{I}.
$$
Then $\widetilde U_{R}\in\mathbb E_{\mathcal F}$. By the
definition of $\mathbf T$ and the integral equation satisfied by
$\widetilde U$, then $\widetilde U_{R}=\mathbf T\widetilde U_{R}$.
Hence, $\widetilde U_{R}$ and $U_{R}$ are both fixed points
of $\mathbf T$. It follows from the uniqueness of the fixed point that $\widetilde U_{R}=U_{R}$ and  $\widetilde U=U$. The proof is complete.
\end{proof}

\section{Applications}\label{sec6}
In this section, we provide some examples to illustrate the effectiveness of our main results.
\begin{example}[A non-autonomous stochastic fractional diffusion equation]\label{exam6.1}
Let $X_0=L^2(0,\pi)$ and $X_1=H^2(0,\pi)\cap H_0^1(0,\pi)$. Let $W$ be a real-valued Brownian motion adapted to $(\mathcal F_t)_{t\in[0,T]}$. We consider the following non-autonomous stochastic  diffusion equation with the generalized fractional derivative
\begin{equation}\label{e6.1}
\left\{
\begin{aligned}
&D_{0+}^{\lambda,\mu}u(t)+A(t)u(t)
=F(t,\omega,u(t))+G(t,\omega,u(t))\,dW(t),
&&t\in(0,T),\\
&u(t,0)=u(t,\pi)=0,
&&t\in(0,T),\\
&I_{0+}^{(1-\lambda)(1-\mu)}u(0)=u_0,
\end{aligned}
\right.
\end{equation}
where $D(A(t))=X_1$ and
\begin{align}\label{e6.2}
A(t)v=-\partial_x\bigl(a(t,\cdot)\partial_xv\bigr)
+c(t,\cdot)v.
\end{align}
Fix $p>2$, $\mu\in(1/2,1)$ and $\lambda\in[0,1]$, and set $1-\nu=(1-\lambda)(1-\mu)$.
Choose $\vartheta_0\in(0,1)$ such that
$\nu-\mu<\vartheta_0$, $\nu>1-1/p$ and $\vartheta_0>1-1/p$.
Assume that coefficients $a$ and $c$ are real-valued and
\begin{align}\label{e6.3}
a\in C^{\vartheta_0}\bigl([0,T];W^{1,\infty}(0,\pi)\bigr),
\quad
c\in C^{\vartheta_0}\bigl([0,T];L^\infty(0,\pi)\bigr).
\end{align}
Let the coefficient $a$ be uniformly strongly elliptic, that is, there exists a constant $a_0>0$ such that
\begin{equation}\label{e6.4}
a(t,x)\xi^2\geq a_0|\xi|^2,
\qquad
(t,x,\xi)\in[0,T]\times(0,\pi)\times\mathbb R.
\end{equation}
Assume there is a constant $c_0>0$ such that $c(t,x)\geq c_0$ for
$(t,x)\in[0,T]\times(0,\pi)$. Then we have
\begin{align}\label{e6.5}
a(t,x)\geq a_0,
\quad
c(t,x)\geq c_0.
\end{align}
It follows from \eqref{e6.2}-\eqref{e6.5} that $A(t)$ is positive and
self-adjoint in $X_0$. Hence, $A(t)$ is an invertible sectorial operator in $X_0$ with common domain $X_1$ and uniformly equivalent graph norms. Furthermore, by \eqref{e6.3}, we deduce
\begin{align*}
\|a(t,\cdot)-a(s,\cdot)\|_{W^{1,\infty}}
\leq C|t-s|^{\vartheta_0},\quad
\|c(t,\cdot)-c(s,\cdot)\|_{L^\infty}
\leq C|t-s|^{\vartheta_0}.
\end{align*}
It follows that $\|A(t)-A(s)\|_{L(X_1,X_0)}
\leq C|t-s|^{\vartheta_0}$.
Since $A(\tau)^{-1}:X_0\to X_1$ is bounded uniformly in $\tau$, we obtain
$$\|(A(t)-A(s))A(\tau)^{-1}\|_{L(X_0)}
\leq C|t-s|^{\vartheta_0}.$$
Hence, Assumption~\ref{assump2.1} holds with
$\vartheta=\vartheta_0$.

Since $X_0=L^2(0,\pi)$ is a Hilbert space, it is a UMD space with type $2$. Since $A(t)$ is positive and
self-adjoint in $X_0$, it has a uniformly bounded $H^\infty$-calculus of angle less than $\pi/2$. Hence, Assumption~\ref{assump2.2}
holds. For $v\in X_1$, it follows that
\begin{align}\label{e6.6}
\|A(t)v-A(s)v\|_{X_0}
\leq\|A(t)-A(s)\|_{L(X_1,X_0)}\|v\|_{X_1}
\leq C|t-s|^{\vartheta_0}\|v\|_{X_1}.
\end{align}
From \eqref{e6.6}, $t\mapsto A(t)v$ is continuous and strongly measurable. For $v\in X_1$, we deduce
\begin{align*}
\|A(t)v\|_{X_0}
\leq C\left(
\|a(t,\cdot)\|_{W^{1,\infty}}
+\|c(t,\cdot)\|_{L^\infty}
\right)\|v\|_{X_1}.
\end{align*}
Then $\sup_{t\in[0,T]}\|A(t)\|_{L(X_1,X_0)}<\infty$ and
Assumption~\ref{assump5.1} holds.

Next, we prove the deterministic maximal regularity of $A(t)$. Set $A_0:=A(0)$. The zero initial value autonomous equation
$$D_{0+}^{\lambda,\mu}u+A_0u=f$$
is equivalent to the Volterra equation
$$u+g_\mu*(A_0u)=g_\mu*f, $$
where the kernel $g_\mu(t)=t^{\mu-1}/\Gamma(\mu)$ has Laplace transform
$\widehat g_\mu(z)=z^{-\mu}$. Since
$A_0$ is positive and self-adjoint on the Hilbert space $X_0$, it is
$\mathcal R$-sectorial with angle zero.
According to \cite[Theorem~3.4]{Zacher}, we deduce $A_0\in\mathrm{DMR}(p,\mu,T)$. Let $\mathcal S_0$ denote the solution operator for the autonomous problem with $A_0$. Thus, for each $g\in L^p(0,T;X_0)$, $w=\mathcal S_0g$ is the unique solution of
$$D_{0+}^{\lambda,\mu}w+A_0w=g$$
with zero initial value. Moreover, there exists $M>0$ such that $\|\mathcal S_0g\|_{L^p(0,T;X_1)}
\le M\|g\|_{L^p(0,T;X_0)}$.
Assume that
\begin{align}\label{e6.7}
M_0:=M\sup_{t\in[0,T]}
\|A(t)-A_0\|_{L(X_1,X_0)}<1.
\end{align}
By \eqref{e6.2}, we have
$$\|A(t)-A_0\|_{L(X_1,X_0)}
\leq C\left(
\|a(t,\cdot)-a(0,\cdot)\|_{W^{1,\infty}}
+\|c(t,\cdot)-c(0,\cdot)\|_{L^\infty}
\right).$$
Thus, condition \eqref{e6.7} is satisfied if $a(t,\cdot)$ remains sufficiently close to $a(0,\cdot)$ in $W^{1,\infty}(0,\pi)$ and $c(t,\cdot)$ remains sufficiently close to $c(0,\cdot)$ in $L^\infty(0,\pi)$ uniformly for $t\in[0,T]$.

For $f\in L^p(0,T;X_0)$, define
$$\Phi_f(u):=\mathcal S_0
\bigl(f-(A(\cdot)-A_0)u\bigr).$$
For $u,v\in L^p(0,T;X_1)$, we get
$\|\Phi_f(u)-\Phi_f(v)\|_{L^p(0,T;X_1)}
\leq M_0\|u-v\|_{L^p(0,T;X_1)}$.
Since $M_0<1$, it follows from the Banach fixed point theorem that there exists a unique solution $u\in L^p(0,T;X_1)$ such that
$\|u\|_{L^p(0,T;X_1)}
\le\frac{M}{1-M_0}
\|f\|_{L^p(0,T;X_0)}$. Hence, $D_{0+}^{\lambda,\mu}u=f-A(\cdot)u \in L^p(0,T;X_0)$. Since
$f\in L^p(0,T;X_0)$, $u\in L^p(0,T;X_1)$, and $A(\cdot)$ is uniformly bounded in $L(X_1,X_0)$, we obtain $ A(\cdot)\in\mathrm{DMR}(p,\mu,T)$.

Choose $\alpha\in\left(1/(2\mu),1\right]$, $X_\alpha=[X_0,X_1]_\alpha$ and $\kappa=1$.
Let $f_0$ and $g_0$ be progressively measurable processes belonging to
$L^p(\Omega;L^p(0,T;X_0))$ and
$L^p(\Omega;L^p(0,T;X_\alpha))$, respectively. Let
$k$ and $\sigma$ be bounded progressively measurable real-valued processes. Since $H=\mathbb R$, define a $\gamma$-radonifying operator from
$\mathbb R$ to $X_\alpha$ by the isometric isomorphism $J_\alpha:X_\alpha\longrightarrow\gamma(\mathbb R,X_\alpha)$ with
\begin{align}\label{e6.8}
(J_\alpha z)h=hz
\end{align}
for $h\in\mathbb R$ and $z\in X_\alpha$. For $v\in X_1$, define
\begin{align}\label{e6.9}
F(t,\omega,v):=f_0(t,\omega)
+\varepsilon k(t,\omega)\arctan(v),  \end{align}
\begin{align}\label{e6.10}
G(t,\omega,v):=J_\alpha\bigl(g_0(t,\omega)
+\varepsilon\sigma(t,\omega)v\bigr).
\end{align}
Then it follows from \eqref{e6.8}-\eqref{e6.10} that
$F(\cdot,\cdot,0)=f_0 \in L^p\bigl(\Omega;L^p(0,T;X_0)\bigr)$ and
$G(\cdot,\cdot,0)=J_\alpha g_0 \in\mathbb G_\alpha$.
For $v,w\in X_1$, by
$|\arctan r-\arctan q|\leq|r-q|$ and the embedding
$X_1\hookrightarrow X_0$, we have
\begin{align*}
\|F(t,\omega,v)-F(t,\omega,w)\|_{X_0}
\leq \varepsilon |k(t,\omega)|\|v-w\|_{X_0}\leq C\varepsilon\|k\|_\infty\|v-w\|_{X_1}.
\end{align*}
Similarly, by $X_1\hookrightarrow X_\alpha$ and the isometry of
$J_\alpha$, we obtain
\begin{align*}
\|G(t,\omega,v)-G(t,\omega,w)\|_{\gamma(\mathbb R,X_\alpha)}\leq
C\varepsilon\|\sigma\|_\infty\|v-w\|_{X_1}.
\end{align*}
Thus, Assumption~\ref{assump5.2} holds with
$L_F\leq C\varepsilon\|k\|_\infty$ and $L_G\leq C\varepsilon\|\sigma\|_\infty.$
In addition, by choosing $\varepsilon>0$ sufficiently small, we get $C_{L}C_\kappa(L_F+L_G)<1$  with $\kappa=1$.

For $u_0\in L^p(\Omega,\mathcal F_0;X_1)$, it follows from Theorem~\ref{the5.2} that problem~\eqref{e6.1} admits a unique regularity integrated strong
solution $u=u_{I}+u_{R}$. Moreover, for
$\theta\in[0,\mu-1/2)$,
$$u_{R}\in
L^p\bigl(\Omega;H^{\theta,p}
(0,T;X_{1-\theta/\mu})\bigr),
\quad
u\in L^p\bigl(\Omega;L^p(0,T;X_1)\bigr).$$
\end{example}

\begin{example}[A stochastic fractional SIR type reaction diffusion system]
Let $\mathcal O\subset\mathbb R^d$, $d\leq3$, be a bounded domain with $C^2$ boundary. Let $S=S(t,x)$, $I=I(t,x)$, and $R=R(t,x)$ denote the susceptible, infectious, and removed components of the system, respectively and set $U=(S,I,R)^\top$. We consider the following non-autonomous stochastic Caputo fractional reaction diffusion system
\begin{equation}\label{e6.11}
\left\{
\begin{aligned}
&{}^{C}\partial_t^\mu S-
\operatorname{div}\bigl(d_S(t,x)\nabla S\bigr)+m_SS
=\mathbf{K}(t,\omega,x)-\mathbf{L}(t,\omega,S,I)+[G(t,\omega,U)\,dW_H(t)]_S,\\
&{}^{C}\partial_t^\mu I-
\operatorname{div}\bigl(d_I(t,x)\nabla I\bigr)+m_II
=\mathbf{L}(t,\omega,S,I)-\gamma I+[G(t,\omega,U)\,dW_H(t)]_I,\\
&{}^{C}\partial_t^\mu R-\operatorname{div}\bigl(d_R(t,x)\nabla R\bigr)+m_RR
=\gamma I+[G(t,\omega,U)\,dW_H(t)]_R,\\
&\partial_{n}S=\partial_{n}I=\partial_{n}R=0,
\quad (t,x)\in(0,T]\times\partial\mathcal O,\\
&U(0)=U_0,
\end{aligned}
\right.
\end{equation}
where ${}^{C}\partial_t^\mu$ denotes the Caputo fractional derivative of order $\mu\in(0,1)$, $\mathbf{K}$ is the recruitment rate, $\gamma>0$ is the recovery
rate, $m_S,m_I,m_R>0$ are the natural removal rates, and $d_S,d_I,d_R$ are the diffusion coefficients. The problem \eqref{e6.11} is based on the classical SIR model introduced by Kermack and McKendrick \cite{Kermack}. Fractional SIR and SIS models have subsequently been developed to incorporate memory and non-Markovian effects in disease transmission, see \cite{Angstmann, Wu}.

For $ r\in\mathbb R$, define the function $\mathbf{Q}:\mathbb R\to[0,1)$ by $\mathbf{Q}(r):=\frac{r_+}{1+r_+}$ with $r_+:=\max\{r,0\}$. For $r,s\in\mathbb R$, we have $0\leq \mathbf{Q}(r)\leq1$ and $|\mathbf{Q}(r)-\mathbf{Q}(s)|\leq|r-s|$.
For $S,I\in L^2(\mathcal O)$ and $\chi\geq0$, define the $L^2(\mathcal O)$-valued saturated incidence mapping by
\begin{equation}\label{e6.12}
[\mathbf{L}(t,\omega,S,I)](x)
:=\beta(t,\omega,x)
\frac{\mathbf{Q}(S(x))\mathbf{Q}(I(x))}{1+\chi \mathbf{Q}(I(x))}.
\end{equation}
If $\beta$ is uniformly bounded,  then the incidence mapping $(S,I)\mapsto\mathbf{L}(t,\omega,S,I)$ is globally Lipschitz from
$L^2(\mathcal O)\times L^2(\mathcal O)$ into $L^2(\mathcal O)$ uniformly in $(t,\omega)$.

Set $X_0=L^2(\mathcal O)^3$, $X_1=\left\{v\in H^2(\mathcal O)^3:\partial_{\boldsymbol n}v=0\text{ on }\partial\mathcal O\right\}$. For $v=(v_S,v_I,v_R)^\top\in X_1$, define
$$A(t)v=
\begin{pmatrix}
-\operatorname{div}\bigl(d_S(t,\cdot)\nabla v_S\bigr)+m_Sv_S\\
-\operatorname{div}\bigl(d_I(t,\cdot)\nabla v_I\bigr)+m_Iv_I\\
-\operatorname{div}\bigl(d_R(t,\cdot)\nabla v_R\bigr)+m_Rv_R
\end{pmatrix}.$$
For $\vartheta\in(0,1)$ and $j\in\{S,I,R\}$, assume
\begin{equation}\label{e6.13}
d_j\in C^\vartheta\bigl([0,T];W^{1,\infty}(\mathcal O)\bigr),
\end{equation}
and for all $j\in\{S,I,R\}$ and
$(t,x,\xi)\in[0,T]\times\mathcal O\times\mathbb R^d$, there exists $d_0>0$ such that
\begin{equation}\label{e6.14}
d_j(t,x)|\xi|^2\geq d_0|\xi|^2.
\end{equation}

For $m_j>0$, it follows from \eqref{e6.13} and \eqref{e6.14} that $A(t)$ is an invertible sectorial operator in $X_0$ with the domain $D(A(t))=X_1$ and uniformly equivalent graph norms.
In addition, we obtain
$$\|A(t)-A(s)\|_{L(X_1,X_0)}
\leq C|t-s|^\vartheta.$$
Since $A(\tau)^{-1}:X_0\to X_1$ is bounded uniformly in $\tau$, we get
\begin{equation}\label{e6.15}
\|(A(t)-A(s))A(\tau)^{-1}\|_{L(X_0)}
\leq C|t-s|^\vartheta.
\end{equation}
Hence, by \eqref{e6.13}-\eqref{e6.15}, Assumption~\ref{assump2.1} holds.
Since $X_0$ is a Hilbert space and the operators $A(t)$ are positive
self-adjoint operators with uniformly bounded $H^\infty$-calculus, Assumption~\ref{assump2.2} holds.

For $U=(S,I,R)^\top\in X_0$, define
\begin{equation}\label{e6.16}
[F(t,\omega,U)](x)
:=
\begin{pmatrix}
\mathbf K(t,\omega,x)-[\mathbf L(t,\omega,S,I)](x)\\
[\mathbf L(t,\omega,S,I)](x)-\gamma I(x)\\
\gamma I(x)
\end{pmatrix},
\quad x\in\mathcal O.
\end{equation}
Let $\mathbf{K},\beta:[0,T]\times\Omega\times\mathcal O\to\mathbb R$ be progressively measurable. Assume $\mathbf{K}\ge0$ and $\beta\ge0$ and
\begin{equation}\label{e6.17}
\operatorname*{ess\,sup}_{(t,\omega)}
\|\mathbf{K}(t,\omega,\cdot)\|_{L^2(\mathcal O)}<\infty,
\quad
\operatorname*{ess\,sup}_{(t,\omega,x)}
|\beta(t,\omega,x)|<\infty.
\end{equation}
It follows from \eqref{e6.12}, \eqref{e6.16}, and \eqref{e6.17} that
there exist constants $L_F,C_F>0$ such that
\begin{align}\label{e6.18}
\|F(t,\omega,U)-F(t,\omega,V)\|_{X_0}
\leq L_F\|U-V\|_{X_0},
\end{align}
\begin{align}\label{e6.19}
\|F(t,\omega,U)\|_{X_0}
\leq C_F\bigl(1+\|U\|_{X_0}\bigr).
\end{align}

Let $G_0:[0,T]\times\Omega\to\gamma(H,X_0)$ be progressively measurable and let $\mathbf B:[0,T]\times\Omega
\to L\bigl(X_0,\gamma(H,X_0)\bigr)$
be progressively measurable. Assume that
\begin{equation}\label{e6.20}
\operatorname*{ess\,sup}_{(t,\omega)}
\left(
\|G_0(t,\omega)\|_{\gamma(H,X_0)}
+\|\mathbf B(t,\omega)\|_{L(X_0,\gamma(H,X_0))}
\right)<\infty.
\end{equation}
Define $G(t,\omega,U):=G_0(t,\omega)+\mathbf B(t,\omega)$. By \eqref{e6.20}, there exist constants
$L_G,C_G>0$ such that
\begin{align}\label{e6.21}
\|G(t,\omega,U)-G(t,\omega,V)\|_{\gamma(H,X_0)}
\leq L_G\|U-V\|_{X_0},
\end{align}
\begin{align}\label{e6.22}
\|G(t,\omega,U)\|_{\gamma(H,X_0)}
\leq C_G\bigl(1+\|U\|_{X_0}\bigr).
\end{align}
It follows from \eqref{e6.18}, \eqref{e6.19}, \eqref{e6.21} and \eqref{e6.22} that Assumption~\ref{assump4.1} holds.

Let $p\in[2,\infty)$, $\mu\in(1/2,1)$. Since $\kappa=0$, we have $\mu(1-\kappa)=\mu>1/2$ and $\mu\kappa=0<\vartheta$.
For $U_0\in L^p(\Omega,\mathcal F_0;X_0)$ with $U_0\geq0$ almost
everywhere, and $\mathbf{K}=1$, it follows from Theorem~\ref{the4.1} that problem \eqref{e6.11} admits a unique mild solution $U\in C\bigl([0,T];L^p(\Omega;X_0)\bigr)$.
Here, $U_0=(S_0,I_0,R_0)^\top$ represents the initial spatial distributions of the susceptible, infectious, and
removed populations.

Set $V(t):=U(t)-H_{\mu,1}(t,0)U_0$. The initial component $H_{\mu,1}(t,0)U_0$ describes the evolution of the initial population distribution under diffusion and natural removal. The regular component $V$ represents the population changes induced by recruitment, disease transmission, recovery, and random perturbations. Let $\delta\in[0,1)$ satisfy $\mu\delta\leq\vartheta$, and let
$0<k^\ast<\min\left\{\mu(1-\delta)-\frac12,\vartheta\right\}$.
Since $\nu=1$ and $\kappa=0$, it follows from Theorem~\ref{the4.1} that
$V$ extends to $t=0$ by $V(0)=0$ and $V\in C^{k^\ast}\bigl(
[0,T];L^p(\Omega;X_\delta)\bigr).$
In particular, if $\delta>0$, the regular part has improved spatial regularity and is H\"older continuous up to the initial time.
\end{example}

\vskip0.5cm
\noindent
{\bf Funding declaration:} This paper was supported by   the Guangxi Natural Science Foundation (2026GXNSFAA00640023).\\[2mm]
{\bf Author contributions:} All authors read and approved the final manuscript.\\[2mm]
{\bf Conflict of interest:} The authors declare that they have no competing interest.\\[2mm]
{\bf Data availability statement:} Data sharing is not applicable to this article as no datasets were generated or analyzed during this study.

\end{document}